\PassOptionsToPackage{reqno}{amsmath}
\documentclass[11pt]{amsart}
\usepackage{amsmath}
\usepackage{mathrsfs}
\usepackage{amssymb,graphicx}
\usepackage{IEEEtrantools}
\usepackage{enumerate}
\usepackage{tikz}
\usepackage{comment}
\usepackage{bm}
\usepackage{graphicx}
\usepackage{bbm}
\usepackage{mathtools}
\usepackage{microtype}
\usepackage[margin=1.15in]{geometry}
\usepackage{xcolor}
\usepackage[colorlinks=true, linkcolor=blue, citecolor=red, urlcolor=cyan]{hyperref}
\usepackage{fancyhdr}
\usepackage{amsmath}             
\usepackage{amssymb}           
\usepackage{amsfonts}           
\usepackage{latexsym}           
\usepackage{amsthm}              
\usepackage{bbm}
\usepackage{tikz}
\usetikzlibrary{decorations.pathreplacing}
\usetikzlibrary{angles}
\usetikzlibrary{calc}
\usepackage[justification=centering]{caption}
\usepackage{type1cm}
\usepackage{todonotes}
\usepackage{standalone}

\theoremstyle{plain}
\newtheorem{theorem}{Theorem}[section]
\newtheorem{lemma}[theorem]{Lemma}
\newtheorem{corollary}[theorem]{Corollary}
\newtheorem{proposition}[theorem]{Proposition}

\theoremstyle{definition}

\newtheorem{remark}[theorem]{Remark}
\newtheorem{example}[theorem]{Example}

\newcommand{\pcl}[1]{\overline{#1}^{\,p}}
\newcommand{\oor}{\overline{r}}
\newcommand{\ur}{\underline{r}}

\newcommand{\de}{\delta}
\newcommand{\dd}{\,\text{d}}
\newcommand{\eps}{\varepsilon}
\newcommand{\e}{\text{e}}
\newcommand{\f}{\frac}
\newcommand{\Si}{\Sigma}
\newcommand{\Sad}{\Sigma_{\text{admis}}}
\newcommand{\Z}{\mathbb{Z}}
\newcommand{\fo}{\mathcal{O}^+}

\newcommand{\F}{\mathcal{F}}
\newcommand{\Q}{\mathbb{Q}}
\newcommand{\R}{\mathbb{R}}
\newcommand{\N}{\mathbb{N}}

\newcommand{\ind}{\mathbf{1}}

\DeclareMathOperator{\diam}{diam}
\DeclareMathOperator{\dimH}{\dim_{\mathrm{H}}}
\DeclareMathOperator{\dimdH}{\dim_{\mathrm{DH}}}
\DeclareMathOperator{\dimpB}{\dim_{\mathrm{B},p}}
\DeclareMathOperator{\dimB}{\dim_{\mathrm{B}}}

\DeclareMathOperator{\dimM}{\dim_{\mathrm{M}}}
\DeclareMathOperator{\dimLM}{\dim_{\mathrm{LM}}}
\DeclareMathOperator{\dimUM}{\dim_{\mathrm{UM}}}
\DeclareMathOperator{\dimL}{\dim_{\mathrm{L}}}

\DeclareMathOperator{\dimbe}{\dim_{\mathrm{Be}}}

\newcommand{\mm}{\mathcal{M}}

\newcommand{\mko}{\scalebox{1.15}{$\mathfrak{o}$}}
\newcommand{\bi}{\mathbf{i}}
\newcommand{\tbi}{\tilde{\mathbf{i}}}
\newcommand{\bj}{\mathbf{j}}
\newcommand{\bu}{\mathbf{u}}
\newcommand{\bv}{\mathbf{v}}

\allowdisplaybreaks[4]

\title{Reverse Iterated Function Systems: Density, Dimensions, and $p$-adic Extension}
\date{}
\author{Jun Jie Miao}
\address{School of Mathematical Sciences,  Key Laboratory of MEA(Ministry of Education) \& Shanghai Key Laboratory of PMMP,  East China Normal University, Shanghai 200241, China}

\email{jjmiao@math.ecnu.edu.cn}

\author{Minghui Xu}
\address{School of Mathematical Sciences, East China Normal University, No. 500, Dongchuan Road, Shanghai 200241, P. R. China}
\email{xmhhh@stu.ecnu.edu.cn}

\begin{document}
\begin{abstract}
	In 1996, Strichartz initiated the study of reverse iterated function systems (RIFSs) on integer lattices and the dimensions of their invariant sets. We develop a dimension theory for such systems through their forward orbits, first on locally compact complete metric spaces and then for one-dimensional affine systems.
	
	For an affine RIFS $\F=\{f_i(x)=r_i x+b_i\}_{i=1}^m$ on $\R$, we introduce the semigroup-growth exponent
	\[
	d_{\F}=\lim_{R\to\infty}\f{\log\#\{g\in G^*(\F):|g'|\le R\}}{\log R},
	\]
	where $G^*(\F)$ is the semigroup of distinct maps generated by $\F$. We prove that this limit exists and satisfies
	$\dimB K(\F^{-1})\le d_{\F}\le s$. A forward orbit has mass dimension $d_{\F}$ precisely when it is locally finite, and has Beurling dimension $d_{\F}$ precisely when it is uniformly locally finite; otherwise, the corresponding dimension is infinite. These formulas extend to suitably discrete invariant sets through a finite-orbit decomposition. Moreover, $d_{\F}=s$ exactly when the dual IFS has no exact overlaps, whereas $d_{\F}=0$ exactly when $\F$ is degenerate. Thus $d_{\F}$ is the intrinsic dimension parameter in the presence of overlaps.
	
	Under finite overlaps, we establish two-sided polynomial orbit-counting estimates. For non-overlapping uniformly locally finite orbits, renewal theory yields a limiting central density in the non-arithmetic case and a multiplicatively periodic asymptotic profile in the arithmetic case. We also determine the lower and discrete Hausdorff dimensions of rational lattice orbits with finite overlaps. Finally, when the expansion ratios are signed powers of a fixed prime $p$, we identify the Euclidean mass and Beurling dimensions of every rational orbit with the $p$-adic box-counting dimensions of the orbit and its associated attractor, without any separation assumption.
\end{abstract}
	\maketitle
	\setcounter{tocdepth}{1}
	\tableofcontents
	\section{Introduction}
	
	\subsection{Background and setting}
	Iterated function systems (IFSs) consisting of finitely many contractions and their associated attractors are fundamental objects in classical fractal geometry; see, for example, \cite{Bk_KJF2} for background. Let $(Y,\rho)$ be a complete metric space, and let
	\[
	\boldsymbol{\Psi}=\{\Psi_i:Y\to Y\}_{i=1}^m
	\]
	be a finite family of contractions. By Hutchinson's classical theorem \cite{Hutchinson1981}, there exists a unique nonempty compact set $E\subseteq Y$ such that
	\begin{equation*}
		E=\bigcup_{i=1}^m\Psi_i(E).
	\end{equation*}
	The family $\boldsymbol{\Psi}$ is called a \emph{contractive iterated function system}, and $E$ is its \emph{attractor}. If $Y=\mathbb R^d$ and every $\Psi_i$ is an affine contraction, then $\boldsymbol{\Psi}$ is called a \emph{self-affine iterated function system} and $E$ a \emph{self-affine set}. If, moreover, every $\Psi_i$ is a similarity, then $E$ is called a \emph{self-similar set}. Determining the dimensions of such attractors, especially self-affine ones, is a central and challenging problem in fractal geometry; see \cite{Falco88,LalGa92,FH08,Hochman,BHR19,Das,Feng23} for some major developments.
	
	The compact-attractor paradigm changes fundamentally when contractions are replaced by expanding maps. In this reverse setting, the maps generate an expanding semigroup for which invariant sets need not exist, while the most natural geometric objects are often unbounded forward orbits. Consequently, the relevant dimension theory is governed by large-scale orbit growth rather than by small-scale covering geometry.
	
	Let $(X,d)$ be an unbounded, locally compact, complete metric space. A finite family $\mathcal F=\{f_i\}_{i=1}^m$ of distinct self-maps of $X$ is called a \emph{reverse iterated function system} (RIFS) if there exists $r>1$ such that
	\begin{equation}\label{def_r}
		d\bigl(f_i(x),f_i(y)\bigr)
		\ge r\,d(x,y)
		\qquad
		(x,y\in X,\ 1\le i\le m).
	\end{equation}
	Such a number $r$ is called an \emph{expansion factor} of $\mathcal F$. A nonempty set $K\subseteq X$ is said to be \emph{invariant} under $\mathcal F$ if
	\begin{equation}\label{eq: FUTF}
		K=\bigcup_{i=1}^m f_i(K).
	\end{equation}
	
	We denote the collection of all invariant sets of $\mathcal F$ by $\mathcal I(\mathcal F)$. An invariant set $K$ is called \emph{non-overlapping} if the union in \eqref{eq: FUTF} is disjoint, or equivalently, if $f_i(K)\cap f_j(K)=\varnothing$ whenever $i\ne j$. More generally, we say that $K$ has \emph{finite overlaps} if $f_i(K)\cap f_j(K)$ is finite whenever $i\ne j$. Finally, the system $\mathcal F$ is called \emph{degenerate at $x_0$} if $x_0\in X$ is a common fixed point of all the maps $f_i$; in this case, we also simply say that $\mathcal F$ is \emph{degenerate}.
	
	A subset $E\subseteq X$ is \emph{locally finite} if its intersection with every bounded set is finite, and it is \emph{uniformly locally finite} if
	\[
		\sup_{x\in X}\#\bigl(E\cap B(x,R)\bigr)<\infty
		\qquad\text{for every }R>0.
	\]
	It is \emph{uniformly discrete} if 
	\[
	\inf_{\substack{x,y\in E\\ x\ne y}}d(x,y)>0.
	\] Uniform local finiteness always implies local finiteness. For subsets of Euclidean space, uniform discreteness implies uniform local finiteness, but the converse need not hold; see Proposition \ref{ud>ulf>lf}. 
	
	Strichartz \cite{Strichartz1996} initiated the study of RIFSs in 1996, concentrating on homogeneous systems
	\[
	\{x\mapsto rx+b_i:\mathbb Z\to\mathbb Z\}_{i=1}^m,
	\qquad r\in\mathbb Z_{>1},\quad b_i\in\mathbb Z,
	\]
	and showing that the mass dimension of a non-overlapping invariant set is equal to the similarity dimension $s=\log m/\log r$ of the dual IFS. Allen, Cruttwell, Hare, and R\"onning \cite{Allen2007} subsequently treated arithmetic systems of the form $\{r^{k_i}x+b_i\}_{i=1}^m$ on $\mathbb Z$ and proved that several discrete dimensions of a non-overlapping invariant set coincide with the dimension of the corresponding dual attractor, and hence with its similarity dimension. Deng \cite{Deng2009} later considered certain overlapping systems under a finite type condition and identified the mass dimension of a forward orbit with the Hausdorff dimension of the dual attractor, expressed through a scale-wise count of distinct cylinder maps. 
	
	Thus, the non-overlapping theory developed in the earlier literature was centered on identifying discrete dimensions with the similarity dimension $s$, while the overlapping case was accessible only under additional finite-type and commensurability assumptions. Although Deng's scale-wise counting formula may be viewed as a finite-type precursor of the exponent introduced below, the existing theory did not provide a presentation-independent semigroup-growth exponent, nor a dimension formula valid for general affine RIFSs with arbitrary overlaps.
	
	A further common feature of the preceding works is that uniform discreteness is built into the ambient setting from the outset. More precisely, the RIFS is assumed to act on a uniformly discrete subset of Euclidean space, typically $\mathbb Z^d$. Consequently, every forward orbit and every invariant set under consideration is automatically uniformly discrete and hence uniformly locally finite and locally finite. The finer distinction among these three properties is therefore invisible in the earlier framework, which does not address how the failure of each property affects the associated dimensions.
	
	The present work removes this a priori discreteness assumption. We first study RIFSs on general unbounded, locally compact, complete metric spaces and then specialize to one-dimensional affine expanding systems on $\R$, without assuming that their forward orbits are discrete in any of the preceding senses. Discreteness thus becomes a dynamical property to be established from the system rather than a standing hypothesis. Section \ref{sec_LFUD} develops arithmetic criteria for the relevant finiteness and separation properties, which provide the hypotheses needed for the dimension theory developed in the subsequent sections.

	The forward orbit of $a\in X$ is
	\[
		\fo_{\F}(a)=\{g(a):g\in G^*(\F)\},
	\]
	where $G^*(\F)$ denotes the semigroup of distinct maps generated by $\F$. Forward orbits are the basic building blocks of locally finite invariant sets; this is one reason for taking them, rather than invariant sets alone, as the primary objects of study. More precisely, within our framework, we reprove a classical result of Strichartz \cite{Strichartz1996}: every locally finite invariant set is a finite union of forward orbits starting at fixed points of maps in
	$G^*(\F)$. Consequently, the finite stability of mass and Beurling dimensions transfers our orbit formulas directly to invariant sets satisfying the corresponding discreteness assumptions. In the lattice setting, the orbit estimates likewise lead to the lower and discrete Hausdorff dimensions of invariant sets with finite overlaps. 
	
	When the maps $f_i$ are surjective, their inverses form a contractive IFS $\F^{-1}=\{f_i^{-1}\}_{i=1}^m$, called the \emph{dual system}, whose attractor is denoted by $K(\F^{-1})$. A number-theoretic motivation for studying the discreteness of forward orbits comes from the spectra
	\[
		\mathcal E_{\lambda,D}=\left\{\sum_{j=0}^n a_j\lambda^j:n\ge0,\ a_j\in D\right\}
	\]
	which are precisely the forward orbits of $0$ under the maps $x\mapsto\lambda x+d$, $d\in D$. For the symmetric digit set $D_m=\{0,\pm1,\ldots,\pm m\}$, Akiyama and Komornik \cite{Akiyama-Komornik} proved that $\mathcal E_{q,D_m}$ has an accumulation point if and only if $q<m+1$ and $q$ is not a Pisot number. Feng \cite{Feng} strengthened this result by proving that $\mathcal E_{q,D_m}$ is dense in $\R$ under exactly the same condition. These theorems may therefore be viewed as sharp arithmetic criteria for the local topology of a special class of RIFS orbits.

	Section \ref{sec_LFUD} develops this connection between arithmetic and orbit separation beyond the symmetric-digit setting. Baker's theorem on linear forms in logarithms \cite{Baker1966} yields uniform discreteness for degenerate systems with algebraic expansion ratios, while Garsia's estimate for bounded-coefficient polynomials \cite{Garsia1962} gives separation criteria for systems associated with Pisot numbers. We also obtain local-finiteness criteria for rational and algebraic data and use Hochman's exponential-separation theorem \cite{Hochman} to relate dimension drop of the dual attractor to a failure of separation or finite overlaps in the forward system. These arithmetic results supply concrete discreteness hypotheses for the dimension theorems of Section \ref{sec_MBdim} and clarify how recurrence near the dual
	attractor obstructs local finiteness.
	
	The dimensions used in this paper come from two complementary large-scale counting procedures. For  $E\subseteq\R$ and $\alpha>0$, define its upper and lower $\alpha$-central densities by
	\[
		\mathcal C_\alpha^+(E)=\limsup_{h\to\infty}
		\f{\#(E\cap[-h,h])}{h^\alpha},\qquad\mathcal C_\alpha^-(E)=\liminf_{h\to\infty}\f{\#(E\cap[-h,h])}{h^\alpha}.
	\]
	The upper and lower mass dimensions are the corresponding critical exponents, or jumping points:
	\begin{gather*}
			\dimUM E=\sup\{\alpha\ge0:\mathcal C_\alpha^+(E)>0\}=\inf\{\alpha\ge0:\mathcal C_\alpha^+(E)<\infty\},\\
			\dimLM E=\sup\{\alpha\ge0:\mathcal C_\alpha^-(E)>0\}=\inf\{\alpha\ge0:\mathcal C_\alpha^-(E)<\infty\}.
	\end{gather*}
	Equivalently,
	\[
		\dimLM E=\liminf_{h\to\infty}
		\f{\log\#(E\cap[-h,h])}{\log h},\qquad\dimUM E=
		\limsup_{h\to\infty}
		\f{\log\#(E\cap[-h,h])}{\log h}.
	\]
	If they coincide, their common value is denoted by $\dimM E$.

	The analogous translation-uniform quantity is obtained from the upper $\alpha$-Beurling density
	\[
		\mathcal D_\alpha^+(E)=
		\limsup_{h\to\infty}\sup_{x\in\R}
		\f{\#(E\cap[x-h,x+h])}{h^\alpha}.
	\]
	Its jumping point is the Beurling dimension:
	\[
		\dimbe E=\sup\{\alpha\ge0:\mathcal D_\alpha^+(E)>0\}=\inf\{\alpha>0:\mathcal D_\alpha^+(E)<\infty\}.
	\]
	Equivalently,
	\[
		\dimbe E=\limsup_{h\to\infty}\sup_{x\in\R}
		\f{\log\#(E\cap[x-h,x+h])}{\log h}.
	\]

	Mass dimensions were introduced by Barlow and Taylor \cite{BarlowTaylor1992} in their study of fractal subsets of $\Z^d$. Glasscock \cite{G} subsequently established Marstrand-type projection theorems for counting and upper mass dimensions, with applications to the dimensions of typical sumsets, and Pathak \cite{Pa} proved a complementary Marstrand-type slicing theorem for mass dimension. Beurling dimension was introduced by Czaja, Kutyniok, and Speegle \cite{Wojciech2008} through generalized Beurling densities in their study of Gabor pseudoframes; their framework also contains mass dimension as a special case. Dutkay, Han, Sun, and Weber \cite{Dutkay2011} later showed, under a mild technical condition, that the Beurling dimension of a Fourier frame for certain affine fractal measures agrees with the Hausdorff dimension of the underlying fractal. These developments illustrate the roles of mass and Beurling dimensions in discrete geometric measure theory and harmonic analysis, respectively. The discrete Hausdorff and lower dimensions of Barlow and Taylor, which require a lattice structure, are recalled in Section \ref{sec_DHDlat}.
	
	\subsection{Main results}
	
	We first develop a general orbit framework that does not require the maps to be surjective. Section \ref{sec_RIFS} introduces the compact space $\Sad^\infty(\F)$ of admissible infinite words and its coding map $\pi$. The compact set
	\[
		\pi\bigl(\Sad^\infty(\F)\bigr)
	\]
	may be viewed as the \emph{admissible core} of the system. Proposition \ref{decoposition} gives an exit decomposition of every forward orbit and, in particular, the criterion
	\[
		\fo_{\F}(a)\text{ is locally finite}\quad\Longleftrightarrow\quad\fo_{\F}(a)\cap\pi\bigl(\Sad^\infty(\F)\bigr)\text{ is finite}.
	\]
	We also prove that $\F$ admits an invariant set if and only if $\Sad^\infty(\F)\ne\varnothing$; see Theorem \ref{existenceofK}. In the locally finite category, Theorem \ref{discreteinvariant} shows that an invariant set exists if and only if $P(\F)\ne\varnothing$, where $P(\F)$ is the set of fixed points of elements of $G^*(\F)$, and that every locally finite invariant set is a finite union of forward orbits based at points of $P(\F)$. Combined with the finite stability of mass and Beurling dimensions, this decomposition explains how the orbit formulas yield the corresponding dimension formulas for locally finite or uniformly locally finite invariant sets.
	
	We next specialize to the affine RIFS
	\[
		\F=\{f_i(x)=r_i x+b_i\}_{i=1}^m
		\qquad\text{on }\R,\qquad |r_i|>1.
	\]
	Its dual attractor $K(\F^{-1})$ has similarity dimension $s$, determined by
	\[
		\sum_{i=1}^m |r_i|^{-s}=1.
	\]
	The similarity dimension counts symbolic words. When distinct words define the same affine map, it therefore overcounts the actual semigroup growth. This leads to the following intrinsic exponent:
	\begin{equation}\label{defofdf}
		d_{\F}:=\lim_{R\to\infty}\f{\log G_{\F}(R)}{\log R},
	\end{equation}
	where $	G_{\F}(R):=\#\{g\in G^*(\F):|g'|\le R\}.$
	The fact that the limit exists is part of the theorem.
	
	\begin{theorem}\label{thm:intro-main}
		Let $\F$ be as above. Then the box-counting dimension of $K(\F^{-1})$ exists and
		\[
			\dimB K(\F^{-1})\le d_{\F}\le s.
		\]
		For every $a\in\R$,
		\[
			\dimM\fo_{\F}(a)=\begin{cases}
				d_{\F},&\text{if }\fo_{\F}(a)\text{ is locally finite},\\
				\infty,&\text{otherwise},
			\end{cases}
		\]
		and
		\[
			\dimbe\fo_{\F}(a)=\begin{cases}
				d_{\F}\le1,&\text{if }\fo_{\F}(a)\text{ is uniformly locally finite},\\
				\infty,&\text{otherwise}.
			\end{cases}
		\]
	\end{theorem}
	
	This is Theorem \ref{massandbeu}. Proposition \ref{equivalencedefofdf} further identifies $d_{\F}$ both as the critical exponent
	\[
		d_{\F}=\inf\left\{t\ge0:\sum_{g\in G^*(\F)}|g'|^{-t}<\infty\right\}
	\]
	and as the unique zero of a subadditive pressure. Thus $d_{\F}$ is an invariant of the affine semigroup rather than of a chosen symbolic presentation. Its two extremal cases have exact structural characterizations:
	\[
		d_{\F}=s\quad\Longleftrightarrow\quad
		\F^{-1}\text{ has no exact overlaps},\qquad d_{\F}=0
		\quad\Longleftrightarrow\quad
		\F\text{ is degenerate};
	\]
	see Propositions \ref{df=s} and \ref{degenerate}.
	
	These results separate two phenomena that could not be distinguished in the uniformly discrete setting of the earlier literature. The spatial distribution of a forward orbit determines whether its mass or Beurling dimension is finite, whereas algebraic identifications among elements of the affine semigroup determine the finite value of that dimension through $d_\F$. To the best of our knowledge, Theorem \ref{thm:intro-main} is the first general result identifying the mass and Beurling dimensions of RIFS forward orbits---and, through the orbit decomposition, of suitably discrete invariant sets---with an intrinsic, presentation-independent semigroup-growth exponent, without imposing any separation or finite-type condition.
	
	The similarity dimension $s$ therefore does not, in general, give the dimension of an RIFS orbit or invariant set. Rather, the classical formulas involving $s$ are recovered precisely in the special case $d_\F=s$, equivalently, when the dual IFS has no exact overlaps. When exact overlaps occur, the reduction in the number of distinct semigroup elements may lead to $d_\F<s$, and it is $d_\F$ that records the resulting dimension drop.

	 When finite overlaps are imposed, one obtains the stronger polynomial estimates of Theorem \ref{mass}. If $\F$ is non-degenerate and $\fo_{\F}(a)$ is locally finite with finite overlaps, then
	\[
		\#\bigl(\fo_{\F}(a)\cap[-h,h]\bigr)\asymp h^s.
	\]
	Here $f(h)\asymp g(h)$ means that there exists a constant $C\geq1$, independent of $h$, such that $C^{-1}g(h)\le f(h)\le Cg(h)$ for all sufficiently large $h$.
	If, in addition, $\fo_\F(a)$ is
	uniformly locally finite, then
	\[
		\sup_{x\in\R}\#\bigl(\fo_{\F}(a)\cap[x-h,x+h]\bigr)\asymp h^s.
	\]
	In particular, $d_{\F}=s$, and the latter situation can occur only when $s\le1$. The distinction between finite overlaps of orbit pieces and exact overlaps of maps is important: the former is a geometric condition on a particular orbit or invariant set, whereas the latter is an algebraic relation in the semigroup.
	
	The polynomial estimates can be sharpened for non-overlapping uniformly locally finite orbits. Theorem \ref{asymptotic} shows that in the non-arithmetic case
	\[
		\f{\#\bigl(\fo_{\F}(a)\cap[-h,h]\bigr)}{h^s}
		\longrightarrow C_{\F,a}>0,
	\]
	where $C_{\F,a}$ is given explicitly. In the arithmetic case there is a positive bounded function $\varphi$, multiplicatively periodic with respect to the common base $r>1$, such that
	\[
		\f{\#\bigl(\fo_{\F}(a)\cap[-h,h]\bigr)}{h^s}
		-\varphi(h)\longrightarrow0.
	\]
	Thus the arithmetic/non-arithmetic dichotomy governs whether the normalized orbit count converges or exhibits persistent logarithmic oscillations.
	
	For lattice systems the preceding estimates also determine the covering-based dimensions of Barlow and Taylor. More precisely, Theorem \ref{thm:main} states that if $\F=\{r_i x+b_i\}_{i=1}^m$ is non-degenerate, $|r_i|\in\Z_{>1}$, $b_i\in\Q$, $a\in\Q$, and $\fo_{\F}(a)$ has finite overlaps, then there exists $N\in\Z_{>0}$ such that $N\fo_{\F}(a)\subseteq\Z$ and
		\[
				\dimL\bigl(N\fo_{\F}(a)\bigr)=\dimdH\bigl(N\fo_{\F}(a)\bigr)=\dimM\fo_{\F}(a)=\dimbe\fo_{\F}(a)=\dimH K(\F^{-1})=d_{\F}=s.
		\]
	This yields the corresponding formula for invariant subsets of $\Z$ with finite overlaps; see Corollary \ref{cor_dssSt}.
	
	Finally, suppose that the expansion ratios are signed powers of a prime $p$. The same affine maps are expanding in the Euclidean metric but contracting in the $p$-adic metric. This permits us to use the Euclidean large-scale geometry of forward orbits to determine the small-scale geometry of attractors in $\Q_p$. More precisely, Theorem \ref{thm_dimpadic} shows that the resulting $p$-adic attractor $K$ satisfies, for every $a\in\Q$,
	\[
		\dimM\fo_{\F}(a)=\dimbe\fo_{\F}(a)=\dim_{\mathrm{B},p}\fo_{\F}(a)=\dimpB K=d_{\F}.
	\]
This identifies the large-scale Euclidean growth of a rational forward orbit with the small-scale $p$-adic geometry of both its closure and the associated
attractor. No separation condition is required, since exact overlaps are automatically encoded by $d_\F$ through the growth of distinct
affine maps in the generated semigroup.
	
	\subsection{Ideas of the proofs and organization}
	
	The existence of $d_{\F}$ follows from a nearly submultiplicative estimate for $G_{\F}(R)$. To transfer semigroup growth to orbit growth, we use
	\[
		f_{\bi}(a)=r_{\bi}\big(a-f_{\bi}^{-1}(0)\bigr)
	\]
	together with a polylogarithmic bound for the number of possible expansion ratios. Points outside the admissible core are handled directly, while the exit decomposition reduces a locally finite orbit meeting the core to finitely many orbits starting outside it. The pressure description detects exact overlaps, finite-overlap counting gives two-sided polynomial estimates, and a renewal equation yields the finer asymptotic results. The lattice theorem then follows from a discrete mass distribution argument, while the $p$-adic theorem uses the fact that the $p$-adic closure of a rational orbit is the union of that orbit and the corresponding attractor.
	
	The paper is organized as follows. Section \ref{sec_RIFS} develops admissible coding, orbit decompositions, and the structure of invariant sets on locally compact complete metric spaces. Section \ref{sec_LFUD} links local finiteness and separation of affine forward orbits on $\R$ to algebraic and number-theoretic properties of the system; these criteria supply the discreteness hypotheses needed by the subsequent dimension theorems. Section \ref{sec_MBdim} introduces $d_{\F}$, proves the mass and Beurling dimension formulas for orbits and invariant sets, and analyzes exact and finite overlaps. Section \ref{sec_Cden} derives precise orbit-counting asymptotics by renewal theory. Section \ref{sec_DHDlat} determines the discrete Hausdorff dimensions of lattice orbits and invariant sets. Finally, Section \ref{sec_padic} applies the orbit results to determine the box-counting dimensions of the associated attractors in $\Q_p$ without any separation assumption.

	\section{Forward orbits and invariant sets in general metric spaces}\label{sec_RIFS}
	
	Let $(X,d)$ be a locally compact complete metric space and let $\F = \{f_i\}_{i=1}^m$ be a reverse iterated function system on $X$ with expansion factor $r$, as defined in \eqref{def_r}. In this section, we study forward orbits and fixed points and provide necessary and sufficient conditions for the existence of invariant sets.
	
	\subsection{Admissible words and the coding map}

	First, we introduce the code space. For each integer $n > 0$, we set
	\[
	\Sigma^n = \{1,\dots,m\}^n,\qquad \Sigma^* = \bigcup_{n=1}^\infty \Sigma^n,
	\]
	with $\Sigma^0 = \{\emptyset\}$ containing only the empty word. The set of infinite words is denoted by $\Sigma^\infty = \{1,\dots,m\}^{\mathbb{N}}$. We topologize $\Sigma^* \cup \Sigma^\infty$ using the metric
	\[
	d_\Sigma(\bi,\bj) =
	\begin{cases}
		0 & \text{if }  \bi= \bj,\\[4pt]
		2^{-|\bi \wedge \bj|} & \text{if } \bi \neq \bj,
	\end{cases}
	\]
	where $|\bi \wedge \bj|$ denotes the length of the longest common prefix of $\bi$ and $\bj$. Under this metric, $\Sigma^\infty$ is a compact metric space. Moreover, $\Sigma^\infty$ is equipped with the left‑shift mapping $\sigma:\Sigma^\infty \to \Sigma^\infty$, defined by $\sigma(i_1 i_2 i_3 \ldots) = i_2 i_3 \ldots$.
	
	For an infinite word $\bi = i_1 i_2 \ldots \in \Sigma^\infty$, its truncation to the first $n$ symbols is denoted by $\bi|_n = i_1 \ldots i_n$. For a finite word $\bi = i_1 \ldots i_n \in \Sigma^*$, we write $|\bi| = n$ for its length, $\bi^- = i_1 \ldots i_{n-1}$ for the word with the last symbol removed, and $\widetilde{\bi} = i_n \ldots i_1$ for its reversal. Given $\bi \in \Sigma^*$ and $\bj \in \Sigma^* \cup \Sigma^\infty$, we denote by $\bi\bj$ their concatenation. Finally, we define the associated compositions
	\[
	f_{\bi} = f_{i_1} \circ f_{i_2} \circ \cdots \circ f_{i_n},\qquad
	f_{\bi}^{-1} = (f_{\bi})^{-1} = f_{i_n}^{-1} \circ \cdots \circ f_{i_1}^{-1},\qquad
	f_{\bi}^n = \underbrace{f_{\bi} \circ \cdots \circ f_{\bi}}_{n},
	\]
	where $f_{\emptyset}$ is taken to be the identity mapping on $X$.
	
	The following two standard facts are needed in our proofs.
	
	\begin{lemma}\label{diag}
		For every sequence $\{\bi_n\}_{n=1}^\infty$ in $\Sigma^*$ with $|\bi_n|\to\infty$, there exist a subsequence $\{\bi_{n_k}\}$ and an infinite word $\bi\in\Sigma^\infty$ such that $\lim_{k\to\infty} |\bi_{n_k}\wedge\bi| = \infty$.
	\end{lemma}
	
	\begin{corollary}\label{denseofsymbolic}
		In the metric space $(\Sigma^*\cup\Sigma^\infty,d_\Sigma)$, the set $\Sigma^*$ of finite words is dense; i.e., $\overline{\Sigma^*} = \Sigma^*\cup\Sigma^\infty$. In particular, $(\Sigma^*\cup\Sigma^\infty,d_\Sigma)$ is a compact space.
	\end{corollary}

	For a given bounded subset $A \subset X$,  we define the set of admissible words of $\F$ in $\Sigma^\infty$ with respect to $A$ as
	\[
	\Sad^\infty(\F,A)=\bigg\{\bi\in \Si^\infty:f_{\widetilde{\bi|_n}}(X)\cap A=f_{i_n i_{n-1}\cdots i_1}(X)\cap A\ne\emptyset \text{ for } n\in \Z_{>0}\bigg\}.
	\]
	Note that $\Sad^\infty(\F,A)$ may be empty for any such $A$; see Example \ref{exmp_nonexist}. 
	Choose a base point $\mko \in X$ such that $\mm  =\max_{1 \le i \le m} d(\mko, f_i(\mko)) > 0.$
	We denote the open ball centered at $\mko$ with radius $\f{\mm}{r-1} + 1$ by 
	\begin{equation}\label{def_B}
		B = B\Big(\mko, \f{\mm}{r-1} + 1\Big).
	\end{equation}
	\begin{lemma}\label{admis}
		Let $\F$ be an RIFS on $X$ with expansion factor $r$. Then  
		\begin{enumerate}[(i)]
			\item  for all $\bi\in \Si^*$ and all $\bj\in\{\emptyset\}\cup\Si^*$, $f_\bi^{-1}(B\cap f_{\bi\bj}(X))\subset B\cap f_{\bj}(X)$; 
			\item for each bounded set $A$ containing $B$,  $\Sad^\infty(\F,A)=\Sad^\infty(\F,B)$.
		\end{enumerate}
	\end{lemma}	
	\begin{proof}
		(i) Given  $\bi=i_1\ldots i_n\in \Si^*,\bj\in \{\emptyset\}\cup\Si^*$ and $x\in B\cap f_{\bi\bj}(X)$, it is clear that $d(x,\mko)< \f{\mm}{r-1}+1$ and $d(\mko,f_{i_1}(\mko))\leq \mm$,which implies
		\[
		d(f_{i_1}^{-1}(x),\mko)\le\f 1r d(x,f_{i_1}(\mko))\le \f 1r\big(d(x,\mko)+d(\mko,f_{i_1}(\mko))\big) <\f{\mm}{r-1}+1.
		\]
		Hence, $f_{i_1}^{-1}(B\cap f_{\bi\bj}(X))\subset B\cap f_{i_2}\circ\cdots\circ f_{i_n}(f_{\bj}(X))$, and by induction, the conclusion holds.
		
		(ii) Given a bounded set $A\supset B$, it suffices to show $\Sad^\infty(\F, A) \subseteq \Sad^\infty(\F, B)$. 
		
		Choose an arbitrary   $\bi \in \Sad^\infty(\F, A)$ and fix an integer $n>0$. Since $f_{\widetilde{\bi|_{n+k}}}(X) \cap A \neq \emptyset$ for all $k >0$,  there exists  a sequence  $\big\{x_k^{(n)}\big\}_{k=1}^\infty$ such that $x_k^{(n)} \in f_{\widetilde{\bi|_{n+k}}}(X) \cap A$. Therefore, we obtain
		\begin{equation*}
			\begin{aligned}
				d\big(\mko,f_{i_{n+1}}^{-1}\circ\cdots\circ f_{i_{n+k}}^{-1}(x_k^{(n)})\big)&\le\f1 rd\big(f_{i_{n+1}}(\mko),f_{i_{n+2}}^{-1}\circ\cdots\circ f_{i_{n+k}}^{-1}(x_k^{(n)})\big)\\
				&\le\f1 r \Big(d\big(f_{i_{n+1}}(\mko),\mko\big)+d\big(\mko,f_{i_{n+2}}^{-1}\circ\cdots\circ f_{i_{n+k}}^{-1}(x_k^{(n)})\big)\Big)\\
				&\le \f1 rd\big(\mko,f_{i_{n+2}}^{-1}\circ\cdots\circ f_{i_{n+k}}^{-1}(x_k^{(n)})\big)+\f \mm r.
			\end{aligned}
		\end{equation*}
		By repeating the above process, we deduce that for all $k \geq 1$,
		\begin{align}
			d\big(\mko,f_{i_{n+1}}^{-1}\circ\cdots\circ f_{i_{n+k}}^{-1}(x_k^{(n)})\big)&\le r^{-k} d\big(\mko,x_k^{(n)}\big)+\mm\sum_{i=1}^k r^{-i}\notag\\
			&\le r^{-k} \diam A+\mm\sum_{i=1}^k r^{-i} \to \f \mm{r-1}\ (k\to\infty).\label{iterate}
		\end{align}
		Therefore, for each $n>0$, there exists $x_{k_n}^{(n)}\in f_{\widetilde{\bi|_{n+k_n}}}(X)\cap B$.  By (i), we have
		\[
		f_{i_{n+1}}^{-1}\circ\cdots\circ f_{i_{n+k_n}}^{-1}\big(x_{k_n}^{(n)}\big)\in f_{i_{n+1}}^{-1}\circ\cdots\circ f_{i_{n+k_n}}^{-1}\big(f_{\widetilde{\bi|_{n+k_n}}}(X)\cap B\big)\subset  f_{\widetilde{\bi|_n}}(X)\cap B, 
		\]
		which implies $\bi\in \Sad^\infty(\F,B)$, and we have $\Sad^\infty(\F, A) \subseteq \Sad^\infty(\F, B)$. 
	\end{proof}
	In view of Lemma \ref{admis}, the set $\Sad^\infty(\F,A)$ coincides with $\Sad^\infty(\F,B)$ for any sufficiently large bounded set $A\subset X$. This independence allows us to unambiguously define the \emph{set of admissible words} of $\F$ as
	\begin{equation}\label{def_admWS}
		\Sad^\infty(\F) \coloneq \Sad^\infty(\F,B).
	\end{equation}
	Observe that the inclusion $f_{\widetilde{\bi|_{n+1}}}(X) \subseteq f_{\widetilde{\sigma(\bi)|_n}}(X)$ holds for all $\bi\in\Sigma^\infty$ and $n>0$, which immediately implies that 
	\[
	\sigma\big(\Sad^\infty(\F)\big) \subseteq \Sad^\infty(\F).
	\] 
	Moreover, if all mappings in $\F$ are surjective, the admissible words trivially reduce to the full code space, i.e., $\Sad^\infty(\F) = \Sigma^\infty$. 
	
	With this symbolic structure in place, we proceed to define the coding mapping 
	\[
	\pi \colon \Sad^\infty(\F) \longrightarrow X.
	\] 
	Since $r^{-n}\operatorname{diam}(B) \to 0$ as $n \to \infty$, the sequence of preimages converges, allowing us to define
	\begin{equation}\label{def_pi}
		\pi(\bi) = \lim_{n\to\infty} f_{\widetilde{\bi|_n}}^{-1}(x_n) = \lim_{n\to\infty} f_{i_1}^{-1} \circ \cdots \circ f_{i_n}^{-1}(x_n),
	\end{equation}
	where $x_n \in f_{\widetilde{\bi|_n}}(X) \cap B$. It is straightforward to verify that $\pi(\mathbf{i})$ is well-defined; that is, the limit is independent of both the particular sequence $\{x_n\}$ and the particular bounded set $B$. Explicitly, for any bounded set $A \supset B$ and any sequence $x_n' \in f_{\widetilde{\mathbf{i}|_n}}(X) \cap A$, we have
	\[
	\lim_{n\to\infty} f_{\widetilde{\mathbf{i}|_n}}^{-1}(x_n) = \lim_{n\to\infty} f_{\widetilde{\mathbf{i}|_n}}^{-1}(x_n').
	\]
	Moreover, the mapping $\pi$ is readily seen to be continuous.

	Recall that $G^*(\F)$ denotes the semigroup generated by $\F$, and that
	$P(\F)$ is the set of fixed points of maps in $G^*(\F)$; namely,
	\[
	P(\F)=\bigl\{x\in X:g(x)=x\text{ for some }g\in G^*(\F)\bigr\}.
	\]
	Recall also that the forward orbit of $a\in X$ under $\F$ is
	\[
	\fo_\F(a)=\bigl\{g(a):g\in G^*(\F)\bigr\}=\bigl\{f_\bi(a):\bi\in\Si^*\bigr\}.
	\]
	 As with invariant sets, we say that $\fo_\F(a)$ is \emph{non-overlapping} if
	\[
	f_i\bigl(\fo_\F(a)\bigr)\cap f_j\bigl(\fo_\F(a)\bigr)=\emptyset\qquad\text{whenever}\quad i\neq j.
	\]
	More generally, we say that $\fo_\F(a)$ has \emph{finite overlaps} if
	\[
	f_i\bigl(\fo_\F(a)\bigr)\cap f_j\bigl(\fo_\F(a)\bigr)
	\]
	is finite whenever $i\neq j$.
	
	First, we establish the following basic properties.

	\begin{proposition}	\label{onX}
		Let $\F$ be an RIFS on $X$ with expansion factor $r$ and $B$ be given by \eqref{def_B}. 
		\begin{enumerate}[(i)]
			\item $\pi\big(\Sad^\infty(\F)\big)\subseteq \bigcup_{i=1}^m f_i^{-1}\big(\pi\big(\Sad^\infty(\F)\big)\big)$.
			\item $\Sad^\infty(\F)$ and $\pi\big(\Sad^\infty(\F)\big)$ are compact.
			\item $\overline{P(\F)}\subseteq \pi\big(\Sad^\infty(\F)\big)\subset B$. Moreover, if all mappings in $\F$ are surjective, then  $\overline{P(\F)}=\overline{P(\F^{-1})}=K(\F^{-1}).$
			\item If $a\in P(\F)$, then $\fo_\F(a)$ is an invariant set of $\F$. Otherwise, we have
			\[
			\fo_\F(a)=\bigcup_{i=1}^m f_i\big(\fo_\F(a)\big)\cup \big\{f_1(a),\ldots,f_m(a)\big\}.
			\]
			\item The forward orbit $\fo_{\F}(a)$ is either a singleton or an infinite set. Moreover, it is a singleton if and only if $\F$ is degenerate at $a$.
			\item For every  $a\notin \pi\big(\Sad^\infty(\F)\big)$, $\fo_\F(a)$ is locally finite. In particular, if all mappings in $\F$ are surjective, then for every $a \notin K(\F^{-1})$, $\fo_\F(a)$ is locally finite.
		\end{enumerate}
	\end{proposition}
	
	\begin{proof}
		(i) Let $S = \big\{ \bi|_1 \colon \bi \in \Sad^\infty(\F) \big\} \subseteq \Sigma$. For each $i \in S$, we define
		$$
		A_i = \big\{ \bj \in \Sigma^\infty \colon i\bj \in \Sad^\infty(\F) \big\}. 
		$$
		Since  $\sigma\big(\Sad^\infty(\F)\big) \subseteq \Sad^\infty(\F)$, it follows that $A_i \subseteq \Sad^\infty(\F)$. Therefore, we obtain
		$$
		\pi\big(\Sad^\infty(\F)\big) = \bigcup_{i\in S} f_i^{-1}(\pi(A_i)) \subseteq \bigcup_{i\in S} f_i^{-1}\big(\pi\big(\Sad^\infty(\F)\big)\big) \subseteq \bigcup_{i=1}^m f_i^{-1}\big(\pi\big(\Sad^\infty(\F)\big)\big).
		$$

		(ii) Since $\Si^\infty$ is compact and   $\pi$ is continuous, it is sufficient to prove  $\Sad^\infty(\F)$ is closed.

		Let  $\{\bi_k\}$ in $\Sad^\infty(\F)$ be a sequence converging to $\bi\in \Si^\infty$. Then $\lim_{k\to\infty}|\bi_k\wedge\bi|= \infty$. For each fixed $n>0$, there exists $k_n$ such that $|\bi_{k_n}\wedge\bi|>n$, and it is clear that 
		\[
		\bi_{k_n}|_{|\bi_{k_n}\wedge\bi|}=\bi|_{|\bi_{k_n}\wedge\bi|}=\bi|_n\,\bj_n
		\]
		for some $\bj_n\in\Si^*$. Recall that $\widetilde{\bi|_n}=i_ni_{n-1}\cdots i_1$.  Since $\bi_{k_n}\in\Sad^\infty(\F)$, it follows from Lemma \ref{admis} (i) that
		\[
		\emptyset\ne f_{\widetilde{\bj_n}}^{-1}\Big(f_{\widetilde{\bi_{k_n}|_{|\bi_{k_n}\wedge\bi|}}}(X)\cap B\Big)\subset f_{\widetilde{\bi|_n}}(X)\cap B.
		\]
		Therefore, $\bi \in \Sad^\infty(\F)$,  and  $\Sad^\infty(\F)$ is closed. 
		
		(iii) The inclusion $\pi(\Sad^\infty(\F)) \subset B$  follows directly from Lemma \ref{admis}  (i) and \eqref{def_pi}. Since  $\pi(\Sad^\infty(\F))$ is compact,  it is sufficient to show that $P(\F)\subseteq\pi(\Sad^\infty(\F))$. 
		
		For every  $x\in P(\F)$, there exists $\bi=i_1\ldots i_n\in \Si^*$ such that $f_{\bi}(x)=x$. Observe that
		\[
		d(x,\mko)\le \f 1r d(f_{i_n}(x),f_{i_n}(\mko))\le \f1r d(f_{i_n}(x),\mko)+\f \mm r.
		\]
		Iterating this estimate along the word $\bi$, we obtain that
		\[
		d(x,\mko)=d(f_\bi(x),\mko)\le r^{-n}d(f_{\bi}(x),\mko)+\mm\sum_{i=1}^n r^{-i}= r^{-n}d(x,\mko)+\f{\mm(1-r^{-n})}{r-1}.
		\]
		Thus $d(x,\mko)\le \f{\mm}{r-1}$, which confirms that $x\in B$. Let $x_k=f_{\bi|_k}^{-1}(x)$ for $1\le k\le n$. By Lemma \ref{admis} (i), we have
		\[
		x_k=f_{\bi|_k}^{-1}(x)\in f_{\bi|_k}^{-1}\big(f_{\bi}(X)\cap B\big)\subset f_{\widetilde{\tilde{\bi}|_{n-k}}}(X)\cap B,\qquad 1\le k\le n-1.
		\]
		Furthermore, note that $x=f_{\bi}^{-1}(x)=f_{\widetilde{\tilde{\bi}|_{n-k}}}^{-1}(x_k)$. Let $\overline{\bi}=\tbi\,\tbi\,\tbi\ldots\in \Si^\infty$. It then follows from Lemma \ref{admis} (i) that for every integer $p = \ell n + k$ with $0 \le k < n-1$, we have
		\[
		x_{n-k}=f_{\bi|_{n-k}}^{-1}(x)\in f_{\bi|_{n-k}}^{-1}\Big(f_{\bi}^{\ell+1}(X)\cap B\Big)\subset f_{\widetilde{\overline{\bi}|_p}}(X)\cap B.
		\]
		Consequently, this implies that $\overline{\bi}\in \Sad^\infty(\F)$ and $x = \pi\big(\overline{\bi}\big) \in \pi\big(\Sad^\infty(\F)\big)$.

		If all mappings in $\F$ are surjective, then it is straightforward that $\overline{P(\F^{-1})}=\overline{P(\F)}$. By a classical result of Hutchinson \cite{Hutchinson1981}, we have $\overline{P(\F^{-1})}=K(\F^{-1})$.

		(iv) Since $G^*(\F)=\bigcup_{i=1}^mf_iG^*(\F) \cup\{f_1,\ldots,f_m\}$, the conclusion is straightforward.
		
		(v) Assume that $\diam (\fo_\F(a))<\infty$. Since $f_1(\fo_\F(a))\subset \fo_\F(a)$, it follows that
		\[
		r\diam \big(\fo_\F(a)\big)\le \diam \big(f_1(\fo_\F(a))\big)\le \diam \big(\fo_\F(a)\big).
		\]
		Hence $\diam\big(\fo_\F(a)\big)=0$, i.e., $\fo_\F(a)$ is a singleton. Furthermore, we have
		\[
		f_i(a) = f_i \circ f_j(a) \qquad \text{for all } i,j\in\Sigma,
		\]
		which implies that $a$ is a common fixed point of all mappings in $\F$. The converse is obvious.

		(vi)  Given  $a \notin \pi(\Sad^\infty(\F))$, we assume $\fo_\F(a)$ is not locally finite, i.e.,   there exists a bounded set $A \subset X$ containing $B$ such that the intersection $\fo_\F(a) \cap A$ is infinite. 
		
		Hence there exists a sequence of distinct finite words $\{\bi_k\}_{k=1}^\infty \subset \Si^*$ satisfying $|\bi_k| \to \infty$ and $f_{\bi_k}(a) \in A$. By Lemma \ref{diag}, there exist $\bi\in \Sigma^\infty$ and a subsequence $\{\bi_{k_p}\}$ such that  $N_p\coloneq\big|\widetilde{\bi_{k_p}}\wedge \bi\big|\to\infty\ (p\to\infty)$. For each $n>0$, there exists $p_n$ such that $N_{p_n}>n$, and   we have
		\[
		\widetilde{\bi_{k_{p_n}}}=\bi|_n\,\bj_{p_n}
		\]
		for some $\bj_{p_n}\in \Si^*$. Note that $f_{\bi_{k_{p_n}}}(a)\in A$. By the argument to \eqref{iterate}, we  choose $p_n$ sufficiently large such that $f_{\widetilde{\bj_{p_n}}}^{-1}(f_{\bi_{k_{p_n}}}(a))\in A$, and we have
		\[
		f_{\widetilde{\bj_{p_n}}}^{-1}\big(f_{\bi_{k_{p_n}}}(a)\big)\in f_{\widetilde{\bi|n}}(X)\cap A.
		\]
		Therefore, by Lemma \ref{admis} (ii), we obtain $\bi\in \Sad^\infty(\F,A)=\Sad^\infty(\F)$, and it implies that 
		\[
		a\equiv f_{\widetilde{\bi|n}}^{-1}\big(f_{\widetilde{\bj_{p_n}}}^{-1}\big(f_{\bi_{k_{p_n}}}(a)\big)\big)=\pi(\bi)\in \pi\big(\Sad^\infty(\F)\big),
		\]
		which contradicts the assumption that $a \notin \pi(\Sad^\infty(\F))$.
		
		If all mappings in $\F$ are surjective, then $\pi(\Sad^\infty(\F))=\pi(\Si^\infty)=K(\F^{-1})$.
	\end{proof}

	\begin{remark}
		Proposition \ref{onX} (iv) implies that the forward orbit $\fo_\F(a)$ is always structurally close to an invariant set of $\F$.
	\end{remark}

	The following simple lemma implies that the invariant set of an RIFS consists of forward orbits.
	\begin{lemma}\label{invariant}
		Let $\F=\{f_i\}_{i=1}^m$ be an RIFS on $(X,d)$. If $K$ is an invariant set of $\F$, then
		\[
		K=\bigcup_{a\in K}\fo_\F(a).
		\]
	\end{lemma}
	\begin{proof}
		By \eqref{eq: FUTF}, 		it is clear  that $f(K)\subseteq K$ for every $f\in G^*(\F)$. Thus, $\fo_\F(a)\subseteq K$ for all $a\in K$. Conversely, for  each $x\in K$, the self-similarity of $K$ implies that there exist $a\in K$ and $i\in\{1,\ldots,m\}$ such that $x=f_i(a)$, which means $x\in \fo_\F(a)$. Hence the conclusion follows.
	\end{proof}
	\begin{remark}
		Indeed, the proof reveals that the preceding result remains valid for a general iterated function system on a metric space.
	\end{remark}

	\subsection{Orbit decompositions and local finiteness}

			Proposition \ref{onX} (vi) provides a sufficient condition for a
			forward orbit to be locally finite. The following decomposition yields
			a complete characterization of local finiteness for forward orbits
			and will play a crucial role in deriving exact dimension formulas for
			the orbits and invariant sets of similarity RIFSs on $\R$.
		\begin{proposition}\label{decoposition}
			For any $a\in X$ and $A\subseteq X$, define
			\[
			A_\F(a)=\bigcup_{i=1}^mf_i\Big(\{a\}\cup\big(\fo_{\F}(a)\cap A\big)\Big)
			\Big\backslash A.
			\]
			Then
			\[
			\fo_{\F}(a)=(\fo_{\F}(a)\cap A)\cup A_\F(a)\cup \bigcup_{b\in A_\F(a)}\fo_\F(b).
			\]
			In particular, $\fo_{\F}(a)$ is locally finite if and only if $\fo_{\F}(a)\cap \pi\big(\Sad^\infty(\F)\big)$ is finite.
		\end{proposition}
		\begin{proof}
			It is immediate that
			\[
			(\fo_{\F}(a)\cap A)\cup A_\F(a)\cup \bigcup_{b\in A_\F(a)}\fo_\F(b)\subseteq \fo_{\F}(a).
			\]
			To prove the reverse inclusion, fix $\bi=i_1\ldots i_n\in\Si^*$. If $f_{i_n}(a)\notin A$, then
			$f_{i_n}(a)\in A_{\F}(a)$, and hence
			\[
			f_\bi(a)\in \{f_{i_n}(a)\}\cup \fo_{\F}(f_{i_n}(a))\subseteq A_\F(a)\cup \bigcup_{b\in A_\F(a)}\fo_\F(b).
			\] 
			Suppose now that $f_{i_n}(a)\in A$, and let $
			k=\min\{j:f_{i_j\ldots i_n}(a)\in A\}$. If $k=1$, then $f_\bi(a)\in\fo_{\F}(a)\cap A$ trivially. If
			$1<k\leq n$, the minimality of $k$ yields $f_{i_{k-1}\ldots i_n}(a)\in A_\F(a)$ which implies that 
			\[
			f_\bi(a)\in \{f_{i_{k-1}\ldots i_n}(a)\}\cup \fo_{\F}(f_{i_{k-1}\ldots i_n}(a))\subseteq A_\F(a)\cup \bigcup_{b\in A_\F(a)}\fo_\F(b).
			\]
			Thus we obtain 
			\[
			\fo_{\F}(a)=(\fo_{\F}(a)\cap A)\cup A_\F(a)\cup \bigcup_{b\in A_\F(a)}\fo_\F(b).
			\]
			This proves the asserted decomposition.

			If $\fo_{\F}(a)$ is locally finite, then Proposition \ref{onX} (iii) implies that $\fo_{\F}(a)\cap \pi\big(\Sad^\infty(\F)\big)$ is finite. Conversely, suppose that this intersection is finite and set $A=\pi\big(\Sad^\infty(\F)\big)$. Then $A_{\F}(a)$ is finite as well. Furthermore, for each
			$b\in A_{\F}(a)\subseteq X\setminus A$,
			Proposition \ref{onX}(vi) shows that $\fo_{\F}(b)$ is locally
			finite. The decomposition therefore expresses
			$\fo_{\F}(a)$ as the union of a finite set and finitely many
			locally finite orbits. Hence $\fo_{\F}(a)$ is locally finite.
		\end{proof}
		
	\begin{corollary}
	Let $\F=\{f_i\}_{i=1}^m$ be an RIFS on $(X,d)$ consisting of
	surjective maps. Then $\fo_{\F}(a)$ is locally finite for every
	$a\in X$ if and only if $\F$ is degenerate.
	\end{corollary}
\begin{proof}
	Suppose first that $\F$ is degenerate at $x_0\in X$. Since every map in $\F$ is
	surjective, we have $\pi\big(\Sad^\infty(\F)\big)=K(\F^{-1})=\{x_0\}$. Hence, for every $a\in X$, the intersection $\fo_{\F}(a)\cap K(\F^{-1})$ is finite. Proposition \ref{decoposition} therefore implies that
	$\fo_{\F}(a)$ is locally finite for every $a\in X$.
	
	Conversely, suppose that $\fo_{\F}(a)$ is locally finite for every
	$a\in X$. Fix $a\in K(\F^{-1})$, and choose
	$\bi=i_1i_2\ldots\in\Si^\infty$ such that $a=\pi(\bi)$. Since $\fo_{\F}(a)$ is locally finite, by the same Proposition, $\fo_{\F}(a)\cap K(\F^{-1})$ is finite. On the other hand, for every $n\in\Z_{>0}$,
	\[
	\pi(\sigma^n(\bi))=f_{\widetilde{\bi|_n}}(a)\in \fo_{\F}(a)\cap K(\F^{-1}).
	\]
	It follows that there exist integers $1\leq p<q$ such that
	\[
	\pi(\sigma^p(\bi))=\pi(\sigma^q(\bi))\Longrightarrow f_{i_q}\circ\cdots\circ f_{i_{p+1}}(\pi(\sigma^p(\bi)))=\pi(\sigma^p(\bi))\Longrightarrow a\in f^{-1}_{\widetilde{\bi|_p}}(P(\F^{-1})).
	\]
	 This yields that
	\[
	K(\F^{-1})\subseteq\bigcup_{g\in G^*(\F)}g^{-1}(P(\F)).
	\]
	Every $g\in G^*(\F)$ has at most one fixed point. Since
	$G^*(\F)$ is countable, $P(\F)$ is countable. Furthermore, every
	$g\in G^*(\F)$ is injective, and hence
	$g^{-1}(P(\F))$ is countable. The preceding inclusion therefore
	shows that $K(\F^{-1})$ is at most countable. Therefore the attractor $K(\F^{-1})$ is singleton. Proposition \ref{onX}(v) now
	implies that $\F$ is degenerate.
\end{proof}

	\subsection{Existence and structure of invariant sets}

	Strichartz \cite{Strichartz1996} provided the following fundamental result concerning invariant sets of RIFSs. For the sake of completeness, we provide a slightly modified proof below. The crux of the proof lies in the observation that the equation \eqref{eq: FUTF} for an invariant set $K$ makes it possible to choose preimages within $K$. In other words, for every $x\in K$, one can construct a sequence $\{x_n\}_{n=1}^\infty$ in $K$ satisfying $x_1=x$ and $x_n=f_{i_n}^{-1}(x_{n-1})$ for a suitable sequence of indices $i_n\in\Si$.
	\begin{theorem}\label{discreteinvariant}
		Let $\F=\{f_i\}_{i=1}^m$ be an RIFS on $(X,d)$. If $K$ is a locally finite invariant set of $\F$, then $K$ is a finite union of forward orbits. More precisely, we have
		\[
		K=\bigcup_{a\in K\cap P(\F)}\fo_\F(a),
		\]
		where $K\cap P(\F)$ is finite.  In particular, a locally finite invariant set exists if and only if $P(\F) \neq \emptyset$.
	\end{theorem}
	
	\begin{proof}
		By Lemma \ref{invariant},  it suffices to prove that $K\subseteq\bigcup_{a\in K\cap P(\F)}\fo_\F(a).$ Since $K$ is an invariant set, for each $x \in K$, there exists $\bi = i_1 i_2 \dots \in \Si^\infty$ such that
		\[
		x_n \coloneq f_{\bi|_n}^{-1}(x)= f_{i_n}^{-1} \circ \dots \circ f_{i_1}^{-1}(x) \in K, \qquad n = 1, 2, \dots.
		\]
		By  \eqref{iterate}, we have 
		\begin{equation}\label{inB}
			\limsup_{n\to\infty}d(x_n,\mko)\le \limsup_{n\to\infty}\Big(r^{-n} d\big(\mko,x)\big)+\mm\sum_{i=1}^n r^{-i}\Big)=\f{\mm}{r-1}.
		\end{equation}
		Thus, $\{x_n\}$ is bounded. Since $K$ is locally finite,   there exist positive integers $k$ and $\ell$ such that
		\[
		x_k=x_{k+\ell}=f_{i_{k+\ell}}^{-1}\circ\dots\circ f_{i_{k+1}}^{-1}(x_k).
		\]
		Therefore, $x_k$ is the fixed point of $f_{i_{k+1}}\circ\dots\circ f_{i_{k+\ell}}$ and $x_k\in K\cap P(\F)$. Thus 
		\[
		x=f_{\bi|_k}(x_k)\in \fo_\F(x_k)\subseteq\bigcup_{a\in K\cap P(\F)}\fo_\F(a).
		\]
		By Proposition \ref{onX} (iii), $P(\F)\subset B$ is bounded. Since $K$ is locally finite,   $K \cap P(\F)$ is finite, and $K$ is a finite union of forward orbits.
		
		The preceding argument shows that $P(\F)$ is non‑empty whenever a locally finite invariant set exists. Conversely, the conclusion follows directly from Proposition \ref{onX}~(iv).
	\end{proof}
	The following example shows that the local finiteness condition in Theorem \ref{discreteinvariant} is indispensable, while also demonstrating that a uniformly expanding system may still admit an invariant set with accumulation points.
	\begin{example}\label{exmp_nonuniq}
		Let $\F = \{f_1(x)=2x, f_2(x)=4x\}$  on $\R$ equipped with the Euclidean metric. The sets $K_1 = \{2^n \colon n \in \Z\}$ and $K_2 = \{0\} \cup K_1$ are invariant. However, $K_1$ is discrete but not closed, while $K_2$ is closed but not discrete. Neither of them is the union of forward orbits of points in $P(\F)$, nor can they be expressed as a finite union of forward orbits.
	\end{example}
	\begin{proof}
		Since $P(\F)=\fo_\F(0)=\{0\}$, neither $K_1$ nor $K_2$ is a union of forward orbits of points in $P(\F)$.
		
		By Lemma \ref{invariant}, we assume that $K_1 = \bigcup_{k=1}^N \fo_\F(a_k)$ for some integer $N \ge 1$, where $a_k = 2^{n_k} \in K_1$ for $1 \le k \le N$. Fix $1 \le k \le N$, every $x \in \fo_\F(a_k)$ is of the form $x = 2^{n+n_k}$ for some  $n>0$. Then for every $x= 2^i \in \bigcup_{k=1}^N \fo_\F(a_k)$, its exponents have the global lower bound: $i > \min_{1\le k\le N} n_k.$  This directly contradicts the assumption that $K_1 = \bigcup_{k=1}^N \fo_\F(a_k)$. Therefore, $K_1$ cannot be expressed as a finite union of forward orbits. 
		
		By a similar argument, the same conclusion holds for $K_2$.
	\end{proof}
	
	Lemma \ref{invariant} and Theorem \ref{discreteinvariant} motivate an investigation into the properties of individual forward orbits. 
	
	Recall that $\mathcal{I}(\F)$ is the collection of all invariant sets of $\F$. Since the union of invariant sets is invariant, if $\mathcal{I}(\F) \neq \emptyset$, then $\F$ admits a maximal invariant set, namely $\bigcup_{K \in \mathcal{I}(\F)} K$. Below we characterize this maximal invariant set and investigate equivalent conditions for the existence of an invariant set for the RIFS.
	\begin{lemma}\label{maximal}
		Let
		\[
		K^* = \big\{ x \in X \colon \text{there exists } \bi \in \Sigma^\infty \text{ such that } x \in f_{\bi|n}(X) \text{ for all } n>0 \big\}.
		\]
		If $K^* \neq \emptyset$, then $\mathcal{I}(\F) \neq \emptyset$, and $K^* = \bigcup_{K \in \mathcal{I}(\F)} K$ is the maximal invariant set of $\F$.
	\end{lemma}
	\begin{proof}
		Since $K^*\ne \emptyset$, we first show that $K^*$ is an invariant set of $\F$. For each $x \in K^*$,   there exists  $\bi = i_1 i_2 \dots \in \Sigma^\infty$ such that $x \in f_{\bi|_n}(X)$ for all $n>0$. It implies that $f_{i_1}^{-1}(x) \in f_{\sigma(\bi)|_{n}}(X)$ for all $n>0$, i.e., $f_{i_1}^{-1}(x) \in K^*$. Hence $K^* \subseteq \bigcup_{i=1}^m f_i(K^*)$. Conversely, arbitrarily choosing $y \in \bigcup_{i=1}^m f_i(K^*)$, there exist  $x \in K^*$ and   $j \in \Sigma$ such that $y = f_j(x)$. Since $x \in K^*$, there exists $\bi' \in \Sigma^\infty$ such that $x \in f_{\bi'|_n}(X)$ for all $n>0$.  Let $\bj=j\,\bi'$. Then $y=f_j(x)\in f_{\bj|_n}(X)\text{ for all } n>0$, i.e., $y \in K^*$, and we have $\bigcup_{i=1}^m f_i(K^*) \subseteq K^*$.  Hence  $K^*$ is an invariant set of $\F$ and $\mathcal{I}(\F)\neq \emptyset$.
		
		Let $K \in \mathcal{I}(\F)$. For each $x \in K$, it follows from \eqref{eq: FUTF} that there exists $\bi \in \Sigma^\infty$ such that $x \in f_{\bi|_n}(X)$ for all $n > 0$, meaning $x \in K^*$. Therefore, $K \subseteq K^*$ for any $K \in \mathcal{I}(\F)$. Combining this inclusion with the fact that $K^*$ itself is an invariant set, the desired conclusion follows immediately.
	\end{proof}
	
	\begin{theorem}\label{existenceofK}
		An RIFS $\F=\{f_i\}_{i=1}^m$ on $X$  has an invariant set if and only if $\Sad^\infty(\F) \neq \emptyset$.
	\end{theorem}
	\begin{proof}
		Suppose that $\F$ has an invariant set $K$. For each $x\in K$, by \eqref{eq: FUTF}, there exists $\bi \in \Si^\infty$ such that $f_{\bi|_n}^{-1}(x)\in K$ for $n>0$. By \eqref{inB}, there exists $n(x)$ such that $y\coloneq f_{\bi|_{n(x)}}^{-1}(x)\in B\cap K$. Again, by \eqref{eq: FUTF}, there exists $\bi' \in \Si^\infty$ such that $f_{\bi'|_n}^{-1}(y)\in K$ for $n>0$. By Lemma \ref{admis} (i), we  have
		\[
		f_{\bi'|_n}^{-1}(y)\in f_{\bi'|_n}^{-1}\big(B\cap f_{\bi'|_{n}}(X)\big)\subset B.
		\]
		By Lemma \ref{diag},  there exist a subsequence $\{n_k\}$ and $\bi''\in \Si^\infty(\F)$ such that $\big|\widetilde{\bi'|_{n_k}}\wedge \bi''\big|\to\infty\ (k\to\infty)$. Similar to the proof of Proposition \ref{onX} (vi), this implies that $\bi'' \in \Sad^\infty(\F)$, and  $\Sad^\infty(\F) \neq \emptyset$.
		
		Conversely, let $\bj \in \Sad^\infty(\F)$. Set $x_0=\pi(\bj)$ and $x_n=f_{\widetilde{\bj|_n}}(x_0)$ for $n>0$, and we have
		\[
		x_0=\lim_{t\to\infty}f_{\widetilde{\bj|_{n+t}}}^{-1}(y_{n+t})=f_{\widetilde{\bj|_n}}^{-1}\Big(\lim_{t\to\infty}f_{\widetilde{\sigma^n(\bj)|_t}}^{-1}(y_{n+t})\Big)=f_{\widetilde{\bj|_n}}^{-1}\big(\pi(\sigma^n(\bj))\big),
		\]
		where the sequence $\{y_{n+t}\}_t$ is chosen such that $y_{n+t} \in f_{\widetilde{\bj|_{n+t}}}(X) \cap B$ for each integer $t>0$.
		Hence, $x_n=\pi\big(\sigma^n(\bj)\big)\in \pi\big(\Sad^\infty(\F)\big)$ for $n>0$. By Proposition \ref{onX} (ii), there exist $x^*\in \pi\big(\Sad^\infty(\F)\big)$ and a subsequence $\{x_{n_k}\}$ such that $x_{n_k}=f_{\widetilde{\bj|_{n_k}}}(x_0)\to x^*\ (k\to \infty)$. By Lemma \ref{diag},  there exist a subsequence $\{n_{k_p}\}$ and $\bj'\in \Si^\infty$ such that $\big|\widetilde{\bj|_{n_{k_p}}}\wedge \bj'\big|\to\infty\ (p\to\infty)$. For each $\ell\in \Z_{>0}$, there exists $P_\ell$ such that $\big|\widetilde{\bj|_{n_{k_p}}}\wedge \bj'\big|>\ell$ for $p>P_\ell$. It follows that
		\[
		f_{\bj'|_\ell}^{-1}\big(x_{n_{k_p}}\big)=f_{\bj'|_\ell}^{-1}\circ f_{\widetilde{\bj|_{n_{k_p}}}}(x_0)\longrightarrow f_{\bj'|_\ell}^{-1}(x^*)\quad \text{as}\quad p\to\infty.
		\]
		Consequently, $x^* \in f_{\bj'|_\ell}(X)$ for all $\ell \in \Z_{>0}$, which implies that $x^* \in K^*$, where $K^*$ is the set defined in Lemma \ref{maximal}. Hence $K^* \neq \emptyset$, and by Lemma \ref{maximal},   $\F$ admits an invariant set.
	\end{proof}

Proposition \ref{decoposition} and
Theorem \ref{existenceofK} immediately yield the following result.

\begin{corollary}
	Let $\F=\{f_i\}_{i=1}^m$ be an RIFS on $X$. If $\F$ admits no
	invariant set, then $\fo_{\F}(a)$ is locally finite for every
	$a\in X$.
\end{corollary}

	An RIFS $\F$ may not have any invariant set (in which case $P(\F) = \emptyset$); see Example \ref{exmp_nonexist}. Prior studies that rely on the existence of invariant sets are therefore inapplicable to such systems. Nevertheless, as we demonstrate in the subsequent sections, it remains possible to completely characterize the topological properties and determine the dimensions of their forward orbits.
	\begin{example}\label{exmp_nonexist}
		Let $\F = \{ 2x+1, 2x+2 \}$ be an RIFS on the space $X = [0, \infty)$ equipped with the Euclidean metric. Then $\F$ admits no invariant set.
	\end{example}
	\begin{proof}
		Observe that for every infinite word $\bi \in \Sigma^\infty$, we have
		$$
		\bigcap_{n \ge 1} f_{\widetilde{\bi|_n}}(X) \subseteq \bigcap_{n \ge 1} [2^n - 1, \infty) = \emptyset.
		$$
		Consequently, $\Sad^\infty(\F) = \emptyset$, and the desired result follows immediately from Theorem \ref{existenceofK}.
	\end{proof}

	\section{Local finiteness and separation of affine forward orbits}\label{sec_LFUD}
	From now on, we study the one-dimensional RIFS $\F=\{f_i \}_{i=1}^m$  on $\R$ given by 
	\begin{equation}\label{IFSform}
		f_i(x) = r_i x + b_i,  \qquad (x\in \R)
	\end{equation}	
	where $r_i, b_i \in \R$ and $|r_i| > 1$ for $1 \le i \le m$. It is clear that each mapping $f_i \in \F$ possesses a unique fixed point. For a word $\bi = i_1 \dots i_n \in \Si^*$, we write $r_\bi=\prod_{j=1}^{n}r_{i_j}$. 
	
	Since every mapping $f_i\in \F$ is surjective, its dual system $\F^{-1}=\{f_i^{-1}(x)=r_i^{-1}(x-b_i)\}_{i=1}^m$ is well-defined.   Let $s$ denote the unique solution to the equation
	\begin{equation}\label{dimension}
		\sum_{i=1}^m|r_i|^{-s}=1.
	\end{equation}
	Recall that  $s$ is  called the \emph{similarity dimension} of the classical self-similar set $K(\F^{-1})$ generated by  the IFS $\F^{-1}$.

	\subsection{Separation notions and overlaps}

	We begin by clarifying the relationships among the three notions of
	discreteness introduced at the beginning of the paper. In particular,
	the following proposition shows that, in Euclidean spaces, uniform
	discreteness is stronger than uniform local finiteness.
	\begin{proposition}\label{ud>ulf>lf}
		Let $E\subseteq \R^d$. Then the following statements are equivalent.\\
		(i) $E$ is uniformly locally finite.\\
		(ii) $E$ is a finite union of uniformly discrete sets.\\
		(iii) There exists $R>0$ such that
		\[
		\sup_{x\in \R^d}\#(E\cap B(x,R))<\infty.
		\]
	\end{proposition}
\begin{proof}
	It is straightforward that (i) and (iii) are equivalent. The equivalence of  (ii) and
	 (iii) is established in \cite[Observation 1.1]{discrete}, where uniformly
	locally finite sets are referred to as relatively separated sets.
\end{proof}

	The following elementary results play a key role in the investigation of dimensions of forward orbits and invariant sets of systems whose dual systems exhibit exact overlaps or even degenerate. 
	\begin{proposition}\label{propA}
		Let $\F$ be an RIFS given by \eqref{IFSform}. Then the following conclusions hold.
		\begin{enumerate}[(i)]
			
			\item $f_i$ and $f_j$ share a common fixed point if and only if $f_i\circ f_j=f_j\circ f_i$.
			
			\item If the dual system $\F^{-1}$ has exact overlaps (that is, $f^{-1}_\bi = f^{-1}_\bj$ for some $\bi \ne \bj \in \Sigma^*$), then there exist $i\ne j$ such that $\#\big( f_i(\fo_\F(a)) \cap f_j(\fo_\F(a)) \big) = \infty$ whenever $\#\fo_\F(a) = \infty$. Otherwise, with the exception of at most countably many points $a$, the forward orbit $\fo_\F(a)$ is non-overlapping.
			
			\item If $r_1 = \dots = r_m = r \in \Z$ and $b_1, \dots, b_m\in \Z$ satisfy $b_i \not\equiv b_j \pmod{r}$ for all $i \ne j$, then for all $a \in \R$, $\fo_\F(a)$ has finite overlaps.
		\end{enumerate}
	\end{proposition}
	
	\begin{proof}
		(i) The following equivalence is straightforward:
		\[
		f_i\circ f_j=f_j\circ f_i\ \Longleftrightarrow\ r_ib_j+b_i=r_jb_i+b_j \ \Longleftrightarrow\ \f{b_i}{1-r_i}=\f{b_j}{1-r_j}\text{ is the common fixed point}.
		\]
		
		(ii) Suppose $f^{-1}_{\bi} = f^{-1}_{\bj}$ for distinct words $\bi, \bj \in \Sigma^*$. The fact that $r_i^{-1} < 1$ for each $i$ ensures that neither $\bi$ nor $\bj$ is a proper prefix of the other.  Consequently, the longest common prefix length $k = |\bi \wedge \bj|$ satisfies $k < \min\{|\bi|, |\bj|\}$, and hence $i_{k+1} \neq j_{k+1}$. Since every $f_i$ is injective and $f_i(\fo_\F(a))\subseteq \fo_\F(a)$,  we obtain 
		\[
		f_\bi\big(\fo_\F(a)\big)=f_\bi\big(\fo_\F(a)\big)\cap f_\bj\big(\fo_\F(a)\big) \subseteq f_{\bi\wedge\bj}\Big( f_{i_{k+1}}(\fo_\F(a))\cap f_{j_{k+1}} (\fo_\F(a)) \Big) .    
		\]		
		If $\#\fo_\F(a) = \infty$, letting $i = i_{k+1}$ and $j = j_{k+1}$, it follows that
		\[
		\#\big(f_i\big(\fo_\F(a)\big)\cap f_j\big(\fo_\F(a)\big)\big) = \infty.
		\]

		Suppose that $f_\bi \ne f_\bj$ for any two distinct finite words $\bi, \bj \in \Sigma^*$. Under this assumption, we obtain that the equation $f_\bi(x) = f_\bj(x)$ admits at most one solution in $\R$. Consequently, the exceptional set of starting points for which the forward orbit exhibits overlaps satisfies
		\begin{align*}
			\R\setminus\{a\in\R:\fo_\F(a)\text{ is non-overlapping}\}\subseteq \{a\in \R: f_\bi(a)=f_\bj(a) \text{ for some }\bi\ne \bj\in \Si^* \}.
		\end{align*}
		Since $\{a\in \R: f_\bi(a)=f_\bj(a) \text{ for some }\bi\ne \bj\in \Si^* \}$ is at most countable, the result  follows immediately.
		
		(iii) Fix $a\in \R$. Suppose that $f_i(\fo_\F(a))\cap f_j(\fo_\F(a))\ne \emptyset$ for some $i\ne j$. Then there exist $x,y\in \fo_\F(a)$ such that 
		\begin{equation}\label{intersect}
			\f {b_j-b_i}r =x-y=a(r^{k}-r^{\ell})+\sum_{p=0}^{k-1}c_pr^p-\sum_{q=0}^{\ell-1}d_qr^q,
		\end{equation}
		where $k,\ell\ge 1$ and the coefficients $c_p, d_q \in \{b_1,\ldots,b_m\}$. Since $\sum_{p=0}^{k-1}c_p r^p - \sum_{q=0}^{\ell-1}d_q r^q \in \mathbb{Z}$ and $b_p \not\equiv b_q \pmod{r}$, it follows from \eqref{intersect} that $a$ can be expressed in the form $a = \f{u}{r^n v}$ for some integers $u$ and $v$ satisfying $\gcd(u,r) = \gcd(v,r) = 1$, where $n \ge 2$ and $n > \min\{k, \ell\}$. Observing that $n$ depends exclusively on $a$, we consequently obtain
		\begin{align*}
			\#\Big\{(x,y)\in \fo_\F(a)\times\fo_\F(a):x-y=\f {b_j-b_i}r\Big\}\le \sum_{p=1}^{n-1}m^p.
		\end{align*}     
		This implies that $f_i(\fo_\F(a))\cap f_j(\fo_\F(a))$ is finite, and       $\fo_\F(a)$ has finite overlaps for all $a \in \R$.
	\end{proof}

	\subsection{Arithmetic criteria for separation}
	The analysis of degenerate RIFS with algebraic expansion ratios necessitates the application of two fundamental results from number theory. First, a landmark result by Baker \cite{Baker1966} on effective lower bounds for linear combinations of logarithms of algebraic numbers is instrumental in characterizing the density of forward orbits. Second, we leverage Garsia’s bound \cite{Garsia1962} for integer polynomials evaluated at algebraic integers to establish the uniform discreteness of orbits associated with Pisot numbers.

	\begin{lemma} \label{baker}
		Let $\alpha_1,\ldots,\alpha_m\in \overline{\Q}\,\backslash\{0,1\}$ and $\beta_1,\ldots,\beta_m\in \overline{\Q}$ with degrees at most $d$. Assume that
		\[
		\Lambda \coloneq \beta_1\log \alpha_1 + \cdots + \beta_m\log\alpha_m \ne 0.
		\]
		Then there exists an effectively computable constant $C = C(m,d,\alpha_1,\ldots,\alpha_m)>0$ such that
		\[
		|\Lambda| > (\mathrm{e} H)^{-C},
		\]
		where $H$ denotes the maximum of the heights of $\beta_1,\ldots,\beta_m$.
	\end{lemma}
	
	\begin{lemma}\label{Garsia}
		Let $\alpha$ be an algebraic integer greater than $1$, and let
		$\alpha_{1},\dots,\alpha_{s}$ be its conjugates different from $\alpha$.
		Denote by $\sigma$ the number of indices $i$ with $|\alpha_{i}| = 1$.
		If $A(x)$ is a polynomial with integer coefficients of height $\le M$
		and degree $\le n$ such that $A(\alpha) \neq 0$, then
		\[
		|A(\alpha)| \ge 
		\f{ \prod_{|\alpha_{i}| \neq 1} \bigl| |\alpha_{i}| - 1 \bigr| }
		{ (n+1)^{\sigma} \bigl( \prod_{|\alpha_{i}| > 1} |\alpha_{i}| \bigr)^{n+1} M^{s} }.
		\]
	\end{lemma}
Here $\alpha_1,\ldots,\alpha_s$ are precisely the algebraic conjugates of $\alpha$ other than $\alpha$ itself. See \cite{Garsia1962} for
details.

Recall that the height of a nonzero rational integer coincides with its absolute value. We therefore obtain the following result.

	\begin{lemma}\label{algebraicratios}
		Let $\F$ be an RIFS given by \eqref{IFSform}. If $\F$  is degenerate at $x_0$ and $r_1, \ldots, r_m$ are algebraic numbers, then for all $a\ne x_0$, there exist
		constants $C,C'>0$ such that, for all distinct
		$x,y\in\fo_{\F}(a)$,
		\begin{equation}\label{uniformdis}
			|x-y|\ge C'\f{\max\{|x-x_0|,|y-x_0|\}}{\big(\log \big(1+\max\{|x-x_0|,|y-x_0|\}\big)\big)^C}.
		\end{equation}
		In particular, $\fo_\F(a)$ is uniformly discrete.
	\end{lemma}
	\begin{proof}
		Since $\F$ is degenerate at $x_0$, given $a\neq x_0$, by  Proposition \ref{propA} (i), we have 
		\[
		\fo_\F(a)=\big\{r_1^{k_1}r_2^{k_2}\ldots r_m^{k_m}(a-x_0)+x_0:(k_1,\ldots,k_m)\in \N^m\backslash\{\textbf{0}\}\big\}.
		\]
		
		Fix   $x\neq y \in \fo_\F(a)$,	and  assume $x=r_1^{k_1}r_2^{k_2}\ldots r_m^{k_m}(a-x_0)+x_0, y=r_1^{k'_1}r_2^{k'_2}\ldots r_m^{k'_m}(a-x_0)+x_0$. 
		If $x-x_0$ and $y-x_0$ have opposite signs, then 
		\[
		|x-y| = |x-x_0|+|y-x_0| \ge \max\{|x-x_0|,|y-x_0|\},
		\] and the bound \eqref{uniformdis} holds trivially. Otherwise without loss of generality, we assume they are positive. Since $(k_1,\ldots,k_m)$ and $(k_1',\ldots,k_m')$ are distinct, 
		$$
		|x-y|= |x-x_0|\Big|\f{y-x_0}{x-x_0}\Big|\ge |x-x_0|\Big(1-\exp\Big(-\Big|\sum_{i=1}^m(k_i'-k_i)\log |r_i|\Big|\Big)\Big).
		$$
		Since $r_1, \ldots, r_m$ are algebraic numbers, by Lemma \ref{baker}, there exists a constant $C>0$ such that
		\begin{align*}
			|x-y|\ge |x-x_0|\Big(1-\exp\Big(-(\text{e}H)^{-C}\Big)\Big),
		\end{align*}               
		where $H=\max_{1\le i\le m}|k_i'-k_i|$. Similarly, $|x-y|\ge |y-x_0|(1-\exp(-(\text{e}H)^{-C})), $ and we obtain that 
		$$
		|x-y|\ge \max\{ |x-x_0|, |y- x_0| \}\Big(1-\exp\Big(-(\text{e}H)^{-C}\Big)\Big),
		$$
		Since $1-e^{-x}\geq \f{x}{2}$ for $0<x<1$, for large $H$, we have $1-\exp(-(\text{e}H)^{-C})\geq C_1 H^{-C}$ for some constant $C_1>0$, which implies that 
		\begin{equation}\label{|x-y|}
		|x-y|\ge C_1\f{\max\{ |x-x_0|, |y- x_0| \}}{ H^C},
	\end{equation}

		Since $|x-x_0|>|a-x_0|>0$, for each $1 \le i \le m$, we have
		$$
		k_i\le \f 1{\log |r_i|}\log\left|\f{x-x_0}{a-x_0}\right| 
		\le\f {\log(1+|x-x_0|)+|\log|a-x_0||}{\log |r_i|}	\le  C_2\log(1+|x-x_0|),
		$$
		where the constant $C_2$ only depends on $r_i, a, x_0$.	 Analogously, we have $k_i' \leq C_2\log(1+|y-x_0|)$, and it implies   $H=O\big(\log\big(1+\max\{|x-x_0|,|y-x_0|\}\big)\big)$.  Therefore, by \eqref{|x-y|}, there exists $C'>0$ such that
		\[
		|x-y|\ge C'\f{\max\{|x-x_0|,|y-x_0|\}}{\big(\log \big(\max\{|x-x_0|,|y-x_0|\}+1\big)\big)^C}.
		\]
		This yields the desired conclusion.
	\end{proof}
	
		The following example illustrates that if the expansion factors of the mappings in $\F$ are not algebraic numbers, the elements of the orbit may exhibit ``algebraic resonance,'' thereby violating the uniform discreteness.
	\begin{example}\label{Be=infty}
		Let $\{a_n\}_{n=1}^\infty, \{b_n\}_{n=1}^\infty$ be two sequences defined by 
		\[
		a_1=b_1=1,\quad a_{n+1}=a_n\big\lceil\e^{nb_n}\big\rceil,\quad b_{n+1}=a_{n+1}\f{b_n}{a_n}+1,\quad n\ge 1.
		\]
		Then $\lim_{n\to\infty}\f{b_n}{a_n}$ exists, which we denote by $\de$.
		Let $\F=\big\{\e x,\e^\de x\big\}$. Then for every $a\ne 0$, $\fo_\F(a)$ is locally finite but not uniformly locally finite.   
	\end{example}
	\begin{proof}
		Note that $K(\F^{-1})=\{0\}$. The local finiteness of $\fo_\F(a)$ follows immediately from Proposition \ref{onX} (vi).
		
		Observe that $a_{n+1}\ge 3a_n$ and $\f{b_{n+1}}{a_{n+1}}=\f{b_{n}}{a_{n}}+\f1{a_{n+1}}$ for $n\ge 1$. Thus $\sum_{n=1}^\infty\f1{a_n}$ converges, 
		\[
		\lim_{n\to\infty}\f{b_n}{a_n}=\sum_{n=1}^\infty\f1{a_n} \quad \text{ and }\quad \lim_{n\to\infty}(n-1)e^{-b_n}=0.
		\]
		Meanwhile, we obtain
		\begin{equation}\label{accuracy1}
			0<\de-\f{b_n}{a_n}=\sum_{k=n+1}^\infty\f1{a_k}\le \f1{a_{n+1}}\sum_{k=0}^\infty3^{-k}<\f2{a_{n+1}}.
		\end{equation}
		Observe that the forward orbit satisfies the scaling property $\fo_\F(a) = a \fo_\F(1)$. Next, we prove that for any $h>0$, we have
		\begin{equation}\label{manypoints}
			\sup_{x\in\R}\#(\fo_\F(1)\cap[x-h,x+h])=+\infty.
		\end{equation}
		By \eqref{accuracy1}, we have
		\begin{equation}\label{accuracy2}
			0<a_n\de-b_n<\f{2a_n}{a_{n+1}}\le2e^{-nb_n}\to 0\ (n\to\infty).
		\end{equation}
		Fix $h>0$. We may choose $n > N > 0$ sufficiently large to ensure that
		\begin{equation}\label{N}
			(N-1)\e^{-b_N}<\f{h}2,\qquad \e^{(N-1)(a_n\de-b_n)}-1<2(N-1)(a_n\de-b_n).
		\end{equation}
		Note that $a_n$ and $b_n$ are  integers for all $n$. Let
		\[
		x_i=\e^{ia_n\de}\e^{(N-1-i)b_n}=\e^{(N-1)b_n+i(a_n\de-b_n)}\in\fo_\F(1),\quad i=0,1,\ldots,N-1.
		\]
		It follows from \eqref{accuracy2}   that $x_0<x_1<\ldots<x_{N-1}$, while \eqref{accuracy2} and \eqref{N} together imply that
		\begin{align*}
			x_{N-1}-x_0&=\e^{(N-1)b_n}\Big(\e^{(N-1)(a_n\de-b_n)}-1\Big)\\
			&<4(N-1)\e^{(N-1-n)b_n}\\
			&<4(N-1)\e^{-b_N}\\
			&<2h.
		\end{align*}
		Therefore, $\sup_{x\in\R}\#(\fo_\F(1)\cap[x-h,x+h])\ge N$ and \eqref{manypoints} holds.
	\end{proof}
	
	While Proposition \ref{onX} (vi) provides a sufficient condition for the local finiteness of forward orbits of a general RIFS, the following proposition establishes several criteria for local finiteness and uniform discreteness—some of which are both necessary and sufficient. These results  play a pivotal role in determining the dimensions of orbits in the subsequent sections.
	
	\begin{proposition}\label{discrete}
		Let $\F$ be an RIFS given by \eqref{IFSform}. Then the following conclusions hold.
		\begin{enumerate}[(i)]
			\item If $r_i\in \Z, b_i\in \Q, 1\le i\le m$, then for every  $a\in \Q$,  $\fo_\F(a)$ is uniformly discrete.
			\item If $\F$  is degenerate at $x_0$, with $|r_i|=r^{k_i}$ for some $r>1$ and positive integers $k_1, \ldots, k_m$, then for every $a\ne x_0$, $\fo_\F(a)$ is uniformly discrete.
			\item  If there exists a Pisot number $\lambda$ and positive integers $k_1, \ldots, k_m$ such that $|r_i|=\lambda^{k_i}$ and $b_1,\ldots, b_m\in \Q(\lambda)$, then $\fo_\F(a)$ is locally finite if and only if $a\notin K(\F^{-1})\backslash \Q(\lambda)$.
			\item If $r_1,\ldots,r_m\in \Z$ and $b_1,\ldots,b_m\in \Q$. Then $\fo_{\F}(a)$ is locally finite if and only if $a\notin K(\F^{-1})\backslash \Q$.
			\item If there exists a Pisot number $\lambda$ and positive integers $k_1, \ldots, k_m$ such that $|r_i|=\lambda^{k_i}$ and $b_1,\ldots, b_m\in \Z[\lambda]$, then for every $a\in \Z[\lambda]$, $\fo_\F(a)$ is uniformly discrete.
		\end{enumerate}
	\end{proposition}
	\begin{proof}
		(i) Suppose $a=\f{p_0}{q_0}$ and $ b_i=\f{p_i}{q_i},1\le i\le m$, and let $N=\text{lcm}(q_0,q_1,\ldots,q_m)$. Then
		\begin{align*}
			\fo_\F(a)=\left\{f_\bi(a):\bi \in\Si^*\right\}
			=\Big\{a\prod_{j=1}^{n}r_{i_j}+\sum_{k=1}^{n-1}b_{i_{k+1}}\prod_{j=1}^kr_{i_j}+b_{i_1}:i_1\ldots i_n\in\Si^*\Big\}\subseteq \f1{N}\Z.
		\end{align*}
		(ii) Since $x_0$ is the common fixed point of $\F$,   Proposition \ref{propA} (i) implies
		\[
		\fo_\F(a)\subseteq\big\{\pm r^{n_1k_1+\ldots+n_mk_m}(a-x_0)+x_0:(n_1,\ldots,n_m)\in \N^m\backslash\{\textbf{0}\}\big\}\subseteq\big\{\pm r^n(a-x_0)+x_0:n>0\big\},
		\]
		and the conclusion follows.

		(iii) First, we prove the sufficiency. Since all mappings in $\F$ are surjective, for $a\notin K(\F^{-1})$, by Proposition \ref{onX} (vi), $\fo_{\F}(a)$ is locally finite. It remains to show  that $\fo_{\F}(a)$ is locally finite for $a\in K(\F^{-1})\cap \Q(\lambda)$. 
		
		Let $d$ denote the degree of $\lambda$. We introduce the Galois embeddings $\sigma_t$, which correspond to the conjugate roots $\lambda^{(t)}=\sigma_t(\lambda)$ for $t = 2, \dots, d$. Since $\lambda$ is a Pisot number, it is an algebraic integer and $|\lambda^{(t)}|<1$. Moreover, since every algebraic number can be expressed as the quotient of an algebraic integer and a non-zero rational integer, there exists a common   integer $q >0$ such that $qa$ and all $qb_i$ are algebraic integers. Since
		\[
		\fo_\F(a)=\Big\{\lambda^{\sum_{j=1}^nk_{i_j}}a+\sum_{\ell=1}^{n-1}b_{i_{\ell+1}}\lambda^{\sum_{j=1}^{\ell}k_{i_j}}+b_{i_1}:i_1\ldots i_n\in\Si^*\Big\},
		\]
		all the elements of $q\fo_\F(a)$ are  algebraic integers. For  $t=2,3,\ldots, d$  and $x\in \fo_{\F}(a)$, we have  
		\begin{align*}
			\big|\sigma_t(qx)\big|\le \big|\sigma_t(qa)\big|+\max_{1\le i\le m}|\sigma_t(qb_i)\big|\sum_{n=0}^\infty\big|\lambda^{(t)}\big|^n=\big|\sigma_t(qa)\big|+\f{\max_{1\le i\le m}\big|\sigma_t(qb_i)\big|}{1-|\lambda^{(t)}|}\eqcolon C_t.
		\end{align*}
		Fix $h>0$. For each $x \in \fo_{\F}(a) \cap [-h,h]$,   all roots of the minimal polynomial of $qx$ are uniformly bounded in the complex plane. Hence, Vieta's formulas imply that the coefficients of the minimal polynomial are also bounded. Since these coefficients are restricted to $\mathbb{Z}$, there exist only finitely many polynomials satisfying these conditions. Consequently,   the set $\{qx : x \in \fo_{\F}(a) \cap [-h,h]\}$ is finite,  i.e., $\fo_{\F}(a) \cap [-h,h]$ is   finite. Hence  $\fo_{\F}(a)$ is locally finite.
		
		Next, we prove the necessity by contradiction. Suppose $\fo_{\F}(a)$ is locally finite, but $a\in K(\F^{-1})\cap \Q(\lambda)^c$. Then there exists $\bi\in\Si^\infty$ such that
		\[
		x_n\coloneq f_{\widetilde{\bi|_n}}(a)\in K(\F^{-1})\cap \fo_\F(a),\qquad n=1,2,\ldots.
		\]
		If $x_{n_1}=x_{n_2}$ for some $n_1\ne n_2$, then 
		\[
		a=\f{\sum_{j=1}^{n_2}b_{i_j}\lambda^{S_{n_2}-S_j}-\sum_{j=1}^{n_1}b_{i_j}\lambda^{S_{n_1}-S_j}}{\lambda^{S_{n_1}}-\lambda^{S_{n_2}}}\in \Q(\lambda),
		\]
		where $S_j=\sum_{\ell=1}^jk_{i_\ell}$, which contradicts $a\in \Q(\lambda)^c$. Thus the $x_n$ are pairwise distinct. Since each $x_n$ lies in the bounded set $K(\F^{-1})$, the forward orbit $\fo_{\F}(a)$ contains infinitely many distinct points in a bounded region, contradicting the assumption that $\fo_{\F}(a)$ is locally finite. Therefore, $a\notin K(\F^{-1})\cap \Q(\lambda)^c$.

		(iv) The local finiteness of $\fo_\F(a)$ follows from Proposition \ref{onX}(vi) when $a \notin K(\F^{-1})$, and from part (i) when $a \in K(\F^{-1}) \cap \Q$. The necessity proof is analogous to that of (iii), replacing $\Q(\lambda)$ by $\Q$ and using the fact that $r_i$s are integers.

		(v) 	For a given $C > 0$, we define
		\begin{multline*}
			P(C) = \big\{ P_n(\lambda) : P_n \in \Z[x], \, \deg(P_n) = n \ge 1, \text{ and all coefficients of } P_n\\ \text{ have absolute value at most } C \big\}.
		\end{multline*}
		For  $a\in \Z[\lambda]$, we have
		$$
		\{x-y:(x,y)\in \fo_\F(a)\times\fo_\F(a)\} \subseteq P(M)
		$$
		for some $M>0$.
		Since $\lambda$ is a Pisot number,  by Lemma \ref{Garsia}, the set $P(M)$ is uniformly discrete, and so is $\fo_\F(a)$. 
	\end{proof}

	\subsection{Dimension drop and separation}
	Hochman \cite{Hochman} introduced the exponential separation condition, providing a breakthrough characterization of the Hausdorff dimension drop for self-similar sets with overlaps on the line. Considering a contracting system $\Phi = \{\varphi_i\}_{i=1}^m$ with $\varphi_i(x) = r_i x + b_i$, he introduced the distance metric for all $\bi, \bj \in \Sigma^*$ as
	\[
	d(\varphi_\bi,\varphi_\bj)=\left\{\begin{array}{cl}
		\infty \qquad & r_\bi\ne r_\bj,\\
		|\varphi_\bi(0)-\varphi_\bj(0)|\quad& r_\bi= r_\bj.
	\end{array}
	\right.
	\] 
	Correspondingly, he defined the minimum separation at the $n$-th step as  
	\[
	\Delta_n = \min \{ d(\varphi_\bi, \varphi_\bj) \colon \bi, \bj \in \Sigma^n, \, \bi \ne \bj \}
	\]
	for $n \in \Z_{>0}$.
	
	\begin{theorem}\cite{Hochman}\label{Hochman}
		Let $X$ be the attractor of a contracting self-similar system on $\R$. If $\dimH X<\min\{1,s\}$ where $s$ is the similarity dimension, then  $\Delta_n\to 0$ super-exponentially, i.e., 
		\[
		\lim_{n\to\infty}-\f{\log\Delta_n}{n}=\infty.
		\]
	\end{theorem}
	As an application of the preceding results, we are able to establish a characterization of the uniform discreteness and overlapping behavior for the forward orbits of an RIFS whose dual system exhibits a dimension drop. Notably, this characterization aligns perfectly with mathematical intuition.
	
	\begin{corollary}\label{dimensiondrop}
		Let $\F$ be an RIFS given by \eqref{IFSform}. If $\dimH K(\F^{-1}) < \min\{1, s\}$, then $\F$ admits no non-singleton forward orbit that simultaneously exhibits uniform discreteness and finite overlaps. More precisely, if $\F^{-1}$ exhibits exact overlaps, then no non-singleton forward orbit has finite overlaps; otherwise, no forward orbit is uniformly discrete. 
	\end{corollary}
	\begin{proof}
		Fix $a \in \R$. If $\F^{-1}$ exhibits exact overlaps—which is equivalent to $\Delta_n = 0$ for all sufficiently large $n$—then Proposition \ref{propA} (ii) yields that $\fo_\F(a)$ has infinite overlaps, provided it is not a singleton.

		Otherwise, $f_\bi \ne f_\bj$ for all finite words $\bi\neq  \bj \in \Sigma^*$, and it implies $\Delta_n > 0$ for all $n \ge 1$.
		For each integer $n \ge 2$, let $\bi_n, \bj_n \in \Sigma^n$ be the specific pair of distinct words that achieves the minimum $\Delta_n$, i.e.,
		\[
		\Delta_n = d\big(f_{\bi_n}^{-1}, f_{\bj_n}^{-1}\big).
		\]
		Since the scalar multiplication commutes, one can easily find distinct words in $\Sigma^n$ with identical contraction ratios (e.g., by swapping  two distinct characters). This guarantees that $\Delta_n < \infty$, which strictly forces $r_{\bi_n} = r_{\bj_n}$.  On one hand,
		\begin{align*}
			|f_{\bi_n}(a)-f_{\bj_n}(a)|=|f_{\bi_n}(0)-f_{\bj_n}(0)|=|r_{\bi_n}||f_{\bi_n}^{-1}(0)-f_{\bj_n}^{-1}(0)|=|r_{\bi_n}|\Delta_n>0.
		\end{align*}
		On the other hand, since  $\dimH K(\F^{-1}) <\min\{1,s\}$,  by Theorem \ref{Hochman}, it follows that 
		\[
		|f_{\bi_n}(a)-f_{\bj_n}(a)|=|r_{\bi_n}|\Delta_n\le \Big(\max_{1\le i\le m}|r_i|\Big)^n\Delta_n\longrightarrow 0\ (n\to\infty).
		\]
		This demonstrates that $\fo_\F(a)$ contains distinct points that become arbitrarily close, and thus it fails to be uniformly discrete.
	\end{proof}
	\begin{remark}
		As revealed by the proof of the preceding corollary, a striking difference between the RIFS $\F$ and the classical system $\F^{-1}$ is that when $\Delta_n > 0$ decays to $0$ at a super-exponential rate, it induces a dimension drop in $\F^{-1}$, yet it simultaneously increases the density of the forward orbits of $\F$.
	\end{remark}

	\section{Semigroup growth and dimensions of affine forward orbits}\label{sec_MBdim}
	In this section we study the upper and lower mass dimensions and the Beurling dimension of forward orbits of the RIFS given by \eqref{IFSform}. The behavior of these dimensions varies markedly with the nature of the
	system: in particular, it depends on whether the system is degenerate or not, and on whether the forward orbit is locally finite or uniformly locally finite. Several examples illustrating these dependencies are collected at the end of the section.

	\subsection{The main dimension theorem}
	
		\begin{theorem}\label{massandbeu}
			Let $\F$ be an RIFS of the form \eqref{IFSform}. Then the limit
			\[
			\lim_{R\to+\infty}
			\f{\log\#\{f_{\bi}:|r_{\bi}|\leq R\}}{\log R}
			\]
			exists and satisfies
			\[
			\dimB K(\F^{-1})\le\lim_{R\to+\infty}\f{\log\#\{f_\bi:|r_\bi|\le R\}}{\log R}\le s.
			\]
			Denote this value by $d_{\F}$. Moreover, for any $a\in\R$,
			\[
			\dimM \fo_{\F}(a)=\begin{cases}
				d_\F, & \text{\ if\, }\fo_{\F}(a)\text{\ is locally finite},\\
				\infty, & \text{otherwise},
			\end{cases}
			\]
			and
			\[
			\dimbe \fo_{\F}(a)=\begin{cases}
				d_\F\le 1, & \text{\ if\, }\fo_{\F}(a)\text{\ is uniformly locally finite},\\
				\infty, & \text{otherwise}.
			\end{cases}
			\]
		\end{theorem}

Since mass and Beurling dimensions are finitely stable, Theorems
\ref{discreteinvariant} and \ref{massandbeu} immediately yield the
following corollary.

	\begin{corollary}
	Let $\F$ be an RIFS given by \eqref{IFSform}, and let $K$ be an invariant set of $\F$. Then
	\[
	\dimM K=\begin{cases}
		d_\F, & \text{\ if\, }K\text{\ is locally finite},\\
		\infty, & \text{otherwise},
	\end{cases}
	\]
	and
	\[
	\dimbe K=\begin{cases}
		d_\F\le 1, & \text{\ if\, }K\text{\ is uniformly locally finite},\\
		\infty, & \text{otherwise}.
	\end{cases}
	\]
\end{corollary}	

	Note that under local finiteness and non-degeneracy, the identity
\[
\dimM\fo_\F(a)=d_\F
\]
may hold even when $d_\F>1$; see Example \ref{ex_LF}. By contrast, the
stronger assumption of uniform local finiteness necessarily forces
$d_\F\leq 1$, so the identity
\[
\dimbe\fo_\F(a)=d_\F
\]
is consistent with the ambient dimension of $\R$.

	\subsection{Existence and bounds for \texorpdfstring{$d_\F$}{dF}}

 We  denote the set of all proper prefixes of $\bi=i_1\ldots i_n\in\Si^*$  by
\[
(\bi)=\{\emptyset,i_1,i_1i_2,\ldots,i_1i_2\ldots i_{n-1}\}.
\]
Given $\de > 1$ and $k \in \Z_{>0}$, we define
\begin{gather*}
	\Lambda_k(\de) = \big\{ \bi \in \Si^* \colon |r_{\bi^{-}}| < \de^k \le |r_{\bi}| \big\}, \qquad
	\Xi_k(\de) = \bigcup_{\bi \in \Lambda_k(\de)} (\bi).
\end{gather*}
It is well known that $\Lambda_k(\de)$ is a cut set of the symbolic tree. Equivalently, the symbolic space admits the following disjoint decomposition:
\begin{equation}\label{cutset}
	\Si^\infty = \bigsqcup_{\bi \in \Lambda_k(\de)} \Big\{ \bi\,\bj\, \colon \bj \in \Si^\infty \Big\}.
\end{equation}
For convenience,  set
\[
\ur = \min_{1\le i\le m} |r_i|,\qquad \oor = \max_{1\le i\le m} |r_i|,\qquad b = \max_{1\le i\le m} |b_i|,\qquad  c = \f{b}{\ur-1}
\] 
and write $G(R)=G_\F(R)$, as in \eqref{defofdf}.

The following lemma, due to de Bruijn and Erdős \cite{Erdos1952}, provides a criterion for the existence of the limit of a nearly subadditive sequence. 
\begin{lemma}\label{nearsubadditivity}
	Let $f \colon [1,+\infty) \to \R$ be a non-negative monotonically increasing function satisfying $\sum_{n=1}^\infty f(n)/n^2 < \infty$. Suppose $M\ge 1$ is given and $\{a_n\}$ is a sequence such that
	\[
	a_{p+q} \le a_p + a_q + f(p+q), \qquad p,q \in \Z_{>M}.
	\]
	Then the limit $\lim_{n\to\infty}a_n/n$ exists.
\end{lemma}

	\begin{lemma}\label{dfexists}
		The limit
		\[
		\lim_{R\to+\infty}\f{\log G(R)}{\log R}
		\]
		which defines $d_\F$ exists.
	\end{lemma}
	\begin{proof}
		Let $R_1$ and $R_2$ be sufficiently large, and let
		$\bi\in\Si^*$ satisfy $|r_{\bi}|\leq R_1R_2$. If
		$|r_{\bi}|>R_1$, write $\bi=\bu\bv$, where $\bu$ is the longest
		prefix of $\bi$ such that $|r_{\bu}|\leq R_1$. By the maximality
		of $\bu$, we have
		\[
		|r_\bv|=\f{|r_\bi|}{|r_\bu|}<\oor\f{|r_\bi|}{R_1}\le \oor R_2.
		\]
	If $|r_{\bi}|\leq R_1$, we instead take $\bu=\bi$ and
	$\bv=\emptyset$. Since $|r_{\emptyset}|=1\leq\oor R_2$, it follows
	in either case that
		\[
		G(R_1R_2)\leq G(R_1)G(\oor R_2).
		\]

		Set $\alpha=\Big\lceil\f{\log \oor}{\log \ur}\Big\rceil$. Then $\ur^\alpha\geq\oor$. For each $\bj\in\Si^*$ satisfying
		$|r_{\bj}|\leq\oor R_2$, remove the last
		$\min\{\alpha,|\bj|\}$ symbols from $\bj$. The remaining prefix,
		denoted by $\bu'$, satisfies $|r_{\bu'}|\leq R_2$.
		Since at most $\alpha$ symbols are removed, we obtain
		\[
		G(\oor R_2)\le \sum_{k=0}^\alpha m^k G(R_2)<m^{\alpha+1}G(R_2).
		\]
		Therefore, 
		\[
		G(R_1R_2)\le G(R_1)G(\oor R_2)<m^{\alpha+1}G(R_1)G(R_2)
		\]
		for all sufficiently large $R_1$ and $R_2$.
		Let $a_n = \log G(2^n)$ for $n \in \mathbb{Z}_{>0}$. It follows that, for all sufficiently large $p$ and $q$, we have
		\[
		a_{p+q}< a_p+a_q+(\alpha+1)\log m.
		\]
		By Lemma \ref{nearsubadditivity}, the limit  $\lim_{n\to\infty} a_n/n$ exists. Hence $\lim_{R \to +\infty} \f{\log G(R)}{\log R}$ exists, completing the proof.
	\end{proof}
	
		Kesseböhmer and Niemann \cite[Proposition 5.11]{Niemann2022} established the existence of the box dimension for the attractor of any $C^1$-IFS in $\R$, which immediately yields the following conclusion.
		\begin{proposition}\label{boxexists}
		Let $\F$ be an RIFS given by \eqref{IFSform}. Then $\dimB K(\F^{-1})$ exists.
	\end{proposition}
	
	\begin{lemma}\label{tree}
		For all $\de>1$ and $k\in\Z_{>0}$, we have $(m-1)\#\Xi_k(\de)=\#\Lambda_k(\de)-1$.
	\end{lemma}
	
	\begin{proof}
		Fix $\de>1$. Observe that $\bi j\in \Lambda_k(\de)\cup \Xi_k(\de)$ for each $\bi\in \Xi_k(\de)$ and $1\le j\le m$. Therefore, $\Lambda_k(\delta) \cup \Xi_k(\delta)$ is precisely the vertex set of a full $m$-ary tree, where $\Lambda_k(\delta)$ is the set of leaves and $\Xi_k(\delta)$ is the set of internal vertices. This conclusion follows from a classical result in graph theory, see, for example, \cite[Section 11.1, Theorem 4]{rosen2019}.
	\end{proof}

	\begin{lemma}\label{dfrange}
		We have
		\[
		\dimB K(\F^{-1})\leq d_{\F}\leq s.
		\]
	\end{lemma}
\begin{proof}
	Since $\sum_{i=1}^m|r_i|^{-s}=1$ and $\Lambda_k(\oor)$ is a cut set of the symbolic tree, we have \[
	\sum_{\bi\in\Lambda_k(\oor)}|r_{\bi}|^{-s}=1
	\]
	for every $k\in\Z_{>0}$. It follows that
	\begin{equation}\label{cutnumber}
	\oor^{ks}\le\#\Lambda_k(\oor)< \oor^{(k+1)s},\qquad k\in\Z_{>0}.
	\end{equation}
	Observe that every word $\bi\in\Si^*$ satisfying $|r_{\bi}|\leq\oor^k$
	belongs either to $\Lambda_k(\oor)$ or to $\Xi_k(\oor)$.
	Therefore, Lemma \ref{tree} and \eqref{cutnumber} give
	\begin{align*}
	 G(\oor^k)&\le \#\big\{\bi\in\Si^*:|r_\bi|\le\oor^k \big\}\\
	 &\le \#\Lambda_k(\oor)+\#\Xi_k(\oor)=\f{m\#\Lambda_k(\oor)-1}{m-1}\le\f{m\oor^{(k+1)s}-1}{m-1},\quad k\in\Z_{>0}.
	\end{align*}
Thus $d_\F\le s$.

For the reverse inequality, choose $D>0$ such that $K(\F^{-1})\subseteq (-D,D)$. By \eqref{cutset}, for each $k\in \Z_{>0}$,
\[
K(\F^{-1})\subseteq \bigcup_{\bi \in \Lambda_k(\oor)}f_{\tilde{\bi}}^{-1}((-D,D))=\bigcup_{\bi \in \Lambda_k(\oor)}B\Big(f_{\tilde{\bi}}^{-1}(0),\f{D}{|r_\bi|}\Big)\subseteq\bigcup_{\bi \in \Lambda_k(\oor)}B\big(f_{\tilde{\bi}}^{-1}(0),\oor^{-k}D\big).
\]
Consequently,
\[
N_{D\oor^{-k}}\bigl(K(\F^{-1})\bigr)
\le\#\bigl\{
f_{\widetilde{\bi}}^{-1}(0):
\bi\in\Lambda_k(\oor)
\bigr\}.
\]
Combining this with Lemma \ref{dfexists} and Proposition \ref{boxexists}, we obtain
\begin{align*}
	\dimB K(\F^{-1})&\le \limsup_{k\to +\infty} \f{\#\big\{f_{\tilde{\bi}}^{-1}(0):\bi\in\Lambda_k(\oor)\big\}}{k\log\oor-\log D}\\
	&\le \limsup_{k\to \infty} \f{\#\big\{f_{\tilde{\bi}}:\bi\in\Lambda_k(\oor)\big\}}{k\log\oor-\log D}\le\limsup_{k\to \infty}\f{\log G(\oor^{k+1})}{k\log\oor-\log D}=d_\F.
\end{align*}
This completes the proof.
\end{proof}

\begin{lemma}\label{simplex}
There exists a constant $C>0$ such that
\[
\#\bigl\{r_{\bi}:|r_{\bi}|\leq R\bigr\}\le C(\log R)^m
\]
for all sufficiently large $R$.
\end{lemma}
\begin{proof}
	For each $\bi\in\Si^*$, let $k_i$ denote the number of occurrences
	of the symbol $i$ in $\bi$. Then
		\begin{align*}
		\#\big\{r_\bi:|r_\bi|\le R\big\}
		&\le\#\big\{(k_1,\ldots,k_m)\in \Z_{\ge0}^m\backslash\{\mathbf 0\}:  |r_1^{k_1}\ldots r_m^{k_m}|\le R\big\}\\
		&=\#\Big\{(k_1,\ldots,k_m)\in \Z_{\ge0}^m\backslash\{\mathbf 0\}: \sum_{i=1}^m k_i\log |r_i|\le \log R\Big\}\\
		&< \prod_{i=1}^m\bigg(\f{\log R}{\log |r_i|}+1\bigg)\le C(\log R)^m
	\end{align*}
for some constant $C>0$ and all sufficiently large $R$.
\end{proof}
\begin{corollary}\label{df=0}
	If $\F$ is degenerate, then $d_\F=0$.
\end{corollary}
\begin{proof}
By Proposition \ref{propA} (i) and the argument used in the proof of
Lemma \ref{simplex}, there exists a constant $C>0$ such that
\[
\begin{aligned}
	G(R)\le\#\big\{(k_1,\ldots,k_m)\in \Z_{\ge0}^m\backslash\{\mathbf 0\}:  |r_1^{k_1}\ldots r_m^{k_m}|\le R\big\}\le C(\log R)^m
\end{aligned}
\]
for all sufficiently large $R$. Consequently,
\[
d_{\F}=\lim_{R\to\infty}\f{\log G(R)}{\log R}=0.
\]
\end{proof}

	\subsection{Mass and Beurling dimension formulas}

\begin{theorem}\label{massdimoflocallyfiniteorbits}
	If $\fo_{\F}(a)$ is locally finite, then
	\[
	\dimM\fo_{\F}(a)=d_{\F}.
	\]
\end{theorem}
\begin{proof}
	For $h>0$, $N(h)=\#(\fo_{\F}(a)\cap[-h,h])$. We first consider the case $a\notin K(\F^{-1})$. In this case,
	Proposition \ref{onX} (vi) ensures that $\fo_{\F}(a)$ is locally
	finite.
	
	Observe that if $f_\bi(a)=f_\bj(a)$ and $r_\bi=r_\bj$ for some $\bi,\bj\in\Si^*$, then $f_\bi=f_\bj$. Consequently, Lemma \ref{simplex} yields
	\begin{align}
	\#\big\{f_\bi(a):|r_\bi|\le R\big\}\le G(R)&\le \#\{r_\bi:|r_\bi|
	\le R\}
	\#\big\{f_\bi(a):|r_\bi|\le R\big\}\notag\\
	&\le C(\log R)^m\#\big\{f_\bi(a):|r_\bi|\le R\big\}.\label{est1}
	\end{align}
Set 
\[
\de \coloneq \inf \big\{ |a - f_\bi^{-1}(0)| \colon \bi \in \Sigma^* \text{ and } f_\bi^{-1}(0) \ne a \big\}.
\]
Since the accumulation points of
$\{f_{\bi}^{-1}(0):\bi\in\Si^*\}$ are contained in
$K(\F^{-1})$, the assumption $a\notin K(\F^{-1})$ implies that
$\delta>0$. Furthermore, for $\bi=i_1\ldots i_n\in\Si^*$, we have
\begin{equation}\label{premageof0}
	|f_{\bi}^{-1}(0)|=\bigg|\sum_{k=1}^n\f{b_{i_k}}{\prod_{j=k}^nr_{i_j}}\bigg|<b\sum_{k=1}^\infty\ur^{-k}=c.
\end{equation}
It follows that, for all sufficiently large $h$,
\begin{equation}\label{est2}
	\begin{aligned}
		N(\de h)&\le 1+\#\big\{f_\bi(a):0<|f_\bi(a)|=|r_\bi(a-f_\bi^{-1}(0))|\le \de h\big\}\\
		&\le 1+\#\big\{f_\bi(a):|r_\bi|\le h\big\}\\
		&\le 2\#\big\{f_\bi(a):|r_\bi|\le h\big\}
	\end{aligned}
\end{equation}
and
\begin{equation}\label{est3}
	\begin{aligned}
	N\big((|a|+c)h\big)
	&=\#\big\{f_\bi(a):|f_\bi(a)|=|r_\bi(a-f_\bi^{-1}(0))|\le (|a|+c) h\big\}\\
	&\ge  \#\big\{f_\bi(a):|r_\bi|\le h\}.
 \end{aligned}
\end{equation}
Combining Lemma \ref{dfexists} with \eqref{est1}, \eqref{est2} and \eqref{est3}, we obtain
\begin{equation}\label{mass=df}
\lim_{h\to+\infty}\f{\log N(h)}{\log h}=\lim_{h\to+\infty}\f{\log\#\{f_\bi(a):|r_i|\le h\}}{\log h}=d_\F.
\end{equation}
Therefore $\dimM \fo_{\F}(a)=d_\F$.

Now suppose that $a\in K(\F^{-1})$ and that $\fo_{\F}(a)$ is
locally finite. By Proposition \ref{onX} (v), the orbit $\fo_{\F}(a)$ is
either a singleton or an infinite set. In the former case, Corollary \ref{df=0} yields $\dimM\fo_{\F}(a)=d_\F=0$. It remains to consider the case where $\fo_{\F}(a)$ is infinite.  Applying Proposition \ref{decoposition} with
$A=K(\F^{-1})$, we obtain
\[
\fo_{\F}(a)=(\fo_{\F}(a)\cap K(\F^{-1}))\cup \mathcal{E}_\F(a)\cup \bigcup_{b\in \mathcal{E}_\F(a)}\fo_\F(b),
\]
where $\mathcal E_{\F}(a)$ is a nonempty finite subset of
$\R\setminus K(\F^{-1})$. By the first part of the proof,
\[
\dimM\fo_{\F}(b)=d_{\F}\quad\text{for every}\quad b\in\mathcal E_{\F}(a).
\]
Since mass dimension is stable under finite unions, it follows that
\[
\dimM\fo_{\F}(a)=d_{\F}.
\]
\end{proof}

\begin{theorem}\label{bedimoflocallyfiniteorbits}
	If $\fo_{\F}(a)$ is uniformly locally finite, then
	\[
	\dimbe\fo_{\F}(a)=d_{\F}\le 1.
	\]
	Otherwise $\dimbe\fo_{\F}(a)=\infty.$
\end{theorem}
\begin{proof}
	If
	\[
	Q\coloneq\sup_{x\in\R} \#(\fo_{\F}(a)\cap[x-1,x+1])=\infty,
	\]
	then $\dimbe\fo_{\F}(a)=\infty.$ Now suppose that $Q<\infty$ and hence $\fo_{\F}(a)$ is uniformly locally finite.
	
	Fix  $h>0$. For every word $\bi\in\Si^*$ satisfying $|r_\bi|>h$, we have $\bi=\bu\bv$ for which $h\le |r_\bu|<\oor h$. Thus
	\begin{equation}\label{be1}
	\big\{f_\bi(a):|r_\bi|> h,\bi\in\Si^*\big\}\subseteq\bigcup_{\substack{g\in G^*(\F)\\h\le |g'|<\oor h}}\Big(g\big(\fo_{\F}(a)\big)\cup \{g(a)\}\Big).
	\end{equation}
For every such $g$ and every $x\in\R$, we have
 \begin{align}
 	\#\big(g(\fo_{\F}(a))\cap [x-h,x+h]\big)&=\#\big(\fo_{\F}(a)\cap g^{-1}([x-h,x+h])\big)\notag\\
 	&=\#\Big(\fo_{\F}(a)\cap \Big[g^{-1}(x)-\f{h}{|g'|},g^{-1}(x)+\f{h}{|g'|}\Big]\Big)\notag\\
 	&\le Q.\label{be2}
 \end{align}
Combining \eqref{be1} and \eqref{be2}, we
obtain
\begin{align*}
	\sup_{x\in \R}\#\big(\fo_{\F}(a)\cap [x-h,x+h]\big)\le \#\big\{f_\bi(a):|r_\bi|\le h\big\}+ (Q+1)G(\oor h).
\end{align*}
By \eqref{mass=df}, it follows that
\[
\dimbe \fo_{\F}(a)\le d_\F.
\]
On the other hand, Theorem \ref{massdimoflocallyfiniteorbits} implies $\dimbe \fo_{\F}(a)\ge \dimM \fo_{\F}(a)=d_\F.$  

Finally, for every $h>1$ and $x\in\R$, the interval
$[x-h,x+h]$ can be covered by at most $\lceil h\rceil+1$
intervals of length $2$. Consequently, 
\[
\sup_{x\in \R}\#\big(\fo_{\F}(a)\cap [x-h,x+h]\big)\le Q(\lceil h\rceil+1).
\]
By the definition, $\dimbe\fo_{\F}(a)\le 1$, which completes the proof.
\end{proof}

\begin{proof}[proof of Theorem \ref{massandbeu}]
	The result follows immediately from Lemmas \ref{dfexists} and \ref{dfrange} and Theorems \ref{massdimoflocallyfiniteorbits} and \ref{bedimoflocallyfiniteorbits}.
\end{proof}

	\subsection{Pressure, exact overlaps, and degeneracy}

The following proposition gives two equivalent alternative
definitions of $d_{\F}$.
\begin{proposition}\label{equivalencedefofdf}
	For every $t\geq0$, the limit
	\[
	P(t):=\lim_{n\to\infty}\f{1}{n}\log\sum_{g\in\{f_{\bi}:\bi\in\Si^n\}}|g^\prime|^{-t}
	\]
	exists. Moreover, the following three quantities coincide:\\
	(i) $d_\F$;\\
	(ii) the unique zero of $P(t)$;\\
	(iii) $\inf\Big\{t\ge 0:\sum_{g\in G^*(\F)}|g^\prime|^{-t}<\infty\Big\}$.
\end{proposition}
\begin{proof}
	For $t\geq0$ and $n\in\Z_{>0}$, set
	\[
	Z_n(t)=\sum_{g\in\{f_{\bi}:\bi\in\Si^n\}}|g^\prime|^{-t}.
	\]
	Every element of $\{f_{\bi}:\bi\in\Si^{n+k}\}$ can be expressed
	as the composition of an element generated by a word of length
	$n$ and one generated by a word of length $k$. Consequently,
	\[
	Z_{n+k}(t)\leq Z_n(t)Z_k(t).
	\]
	Thus $\{\log Z_n(t)\}_{n\geq1}$ is subadditive, and Fekete's Lemma  gives that
	\begin{equation}\label{P(t)_1}
		P(t)=\lim_{n\to\infty}\f{\log Z_n(t)}n=\inf_{n\ge 1}\f{\log Z_n(t)}n
	\end{equation}
	is well defined for every $t\ge0$. We next show that $P(t)$ has a unique zero. Let $0\leq t<u$. For every
	$g\in\{f_{\bi}:\bi\in\Si^n\}$, we have
	\[
	\ur^n\leq |g^\prime|\leq\oor^n.
	\]
	Consequently,
	\[
	\oor^{-n(u-t)}Z_n(t)\leq Z_n(u)=\sum_{g\in\{f_{\bi}:\bi\in\Si^n\}}|g^\prime|^{-t}|g^\prime|^{-(u-t)}\leq \ur^{-n(u-t)}Z_n(t).
	\]
	Passing to the limit gives
	\begin{equation}\label{P(t)_2}
		P(t)-(u-t)\log\oor\le P(u)\le P(t)-(u-t)\log\ur.
	\end{equation}
	It follows that $P$ is continuous and strictly decreasing on
	$[0,\infty)$. Moreover, $0\le P(0)\le \log m$, while
	\[
	P(t)\leq P(0)-t\log\ur\longrightarrow-\infty
	\quad\text{as\ }t\to\infty.
	\]
	Hence $P(t)$ has a unique zero in $[0,\infty)$, which we denote by
	$\eta$.
	
	By the Bohr--Cahen formula for general Dirichlet series, together
	with Lemma \ref{dfexists}, we have
	\[
	d_{\F}=\lim_{R\to+\infty}\f{\log G(R)}{\log R}=\inf\left\{t\geq0:\sum_{g\in G^*(\F)}|g'|^{-t}<\infty\right\}.
	\]
	
	For any $t>\eta$, by \eqref{P(t)_2}, for all sufficiently large $n\in \Z_{>0}$, 
	\[
	Z_n(t)\le \ur^{n(\eta-t)}.
	\]
	Therefore,
	\[
	\sum_{g\in G^*(\F)}|g^\prime|^{-t}\le \sum_{n=1}^\infty Z_n(t)<\infty.
	\]
	This yields that $d_\F\le\eta$. Conversely, suppose that
	\[
	\sum_{g\in G^*(\F)}|g'|^{-t}=C_t<\infty
	\]
	for some $t\ge0$. Then $Z_n(t)\le C(t)$ for all $n\in\Z_{>0}$. It follows that $P(t)\le 0=P(\eta)$ and hence $t\ge \eta$, which completes the proof. 
\end{proof}

\begin{proposition}\label{df=s}
	We have $d_\F=s$ if and only if $\F^{-1}$ has no exact overlaps.
\end{proposition}
\begin{proof}
	Let $P(t)$ be the pressure function defined in
	Proposition \ref{equivalencedefofdf}, and set
	\[
	Z_n(t)
	=
	\sum_{g\in\{f_{\bi}:\bi\in\Si^n\}}|g'|^{-t}.
	\]
	
	First assume that $f_\bi\ne f_{\bj},\bi,\bj\in \Si^*$ whenever $\bi\ne \bj$. Then
	\[
	P(t)=\lim_{n\to\infty}\f{1}{n}\log\sum_{\bi\in\Si^n}|r_\bi|^{-t}=\log \sum_{i=1}^m|r_i|^{-t}.
	\]
	According to Proposition \ref{equivalencedefofdf}, $d_\F=s$ is the unique zero of $P(t)$.
	
Conversely, suppose that $\F^{-1}$ has an exact overlap. We first claim that the two words can be chosen to have the same length. Indeed, without loss of generality, assume that $|\bu|<|\bv|$. Then $f_{\bu\bv}=f_{\bv\bu}$. Moreover, $\bu\bv\neq\bv\bu$. Otherwise, comparing the first
$|\bu|$ symbols in the equality $\bu\bv=\bv\bu$ gives $\bu=\bv|_{|\bu|}.$ Hence $|r_\bu|<|r_\bv|$, which contradicts $f_{\bu}=f_{\bv}$. Therefore, we obtain another exact overlap between two words of same length. We may assume that there exist distinct words $\bu,\bv$ with $|\bu|=|\bv|$ such that $f_\bu=f_\bv$.
	
    Consequently,
    \[
    Z_{|\bu|}(s)\le \sum_{\bi\in \Si^{|\bu|}}|r_\bi|^{-s}-|r_\bu|^{-s}=1-|r_\bu|^{-s}<1.
    \] 
    It follows from \eqref{P(t)_1} that
    \[
    P(s)=\inf_{n\ge 1}\f{\log Z_n(s)}n\le \f{\log Z_{|\bu|}(s)}{|\bu|}<0.
    \]
    Since $P(t)$ is strictly decreasing on $[0,\infty)$ and its unique
    zero is $d_{\F}$ by Proposition \ref{equivalencedefofdf}, we conclude
    that $d_\F<s$.
\end{proof}

In the degenerate case, Corollary \ref{df=0}, Proposition \ref{onX},
and Theorem \ref{massdimoflocallyfiniteorbits} imply that the mass
dimension of every forward orbit is equal to
$d_{\F}=0$. The following proposition establishes the converse,
showing that such a vanishing mass dimension can occur only when the
system is degenerate. Moreover, by applying the results obtained in
the previous section, we provide a sufficient condition under which
the Beurling dimension of every forward orbit also vanishes in the
degenerate case. This reveals a connection between the sparsity of
forward orbits and certain arithmetic properties of the system.

	\begin{proposition}\label{degenerate}
	We have $d_\F=0$ if and only if $\F$ is degenerate. Moreover, if $\F$ is degenerate and $r_1,\ldots,r_m$ are algebraic numbers, then
	\[
	\dimbe\fo_{\F}(a)=0
	\]
	for every $a\in\R$.
	\end{proposition}
\begin{proof}
	If $\F$ is degenerate, $d_\F=0$ follows immediately from Corollary \ref{df=0}. 
	
	Conversely, suppose that $\F$ is non-degenerate. Then there exist
	two maps in $\F$, say
	\[
	f_1(x)=r_1x+b_1,\qquad f_2(x)=r_2x+b_2,
	\]
	with distinct fixed points $x_1$ and $x_2$, respectively. Since $|r_i| > 1,i=1,2$, there exists an integer $N$ such that
	\begin{equation}\label{SSC}
		\big(|r_1|^{-N} + |r_2|^{-N}\big)\diam K(\F^{-1}) < |x_1-x_2|.
	\end{equation}
	For the sake of simplicity, we denote $\mathcal{S} = \{f_1, f_2\}$ and $\mathcal{S}^{-N} = \{f_1^{-N}, f_2^{-N}\}$. Since $f_i^{-N}(x) = r_i^{-N}(x-x_i) + x_i$ for $i=1,2$, and Proposition \ref{onX} (iii) implies that $x_1, x_2 \in K(\F^{-1})$, we deduce that
	\[
	f_i^{-N}(K(\mathcal{S}^{-N}))\subseteq f_i^{-N}(K(\mathcal{S}^{-1}))\subseteq \overline{B(x_i,|r_i|^{-N}\diam K(\mathcal{S}^{-1}))},\quad i=1,2.
	\]
	Combining it with \eqref{SSC} implies that $f_1^{-N}(K(\mathcal{S}^{-N}))\cap f_2^{-N}(K(\mathcal{S}^{-N}))=\emptyset$. In other words, the system $\mathcal{S}^{-N}$ satisfies the strong separation condition. Therefore, by Lemma \ref{dfrange}, we obtain 
	\[
	d_\F \ge \dimB K(\F^{-1}) \ge \dimB K(\mathcal{S}^{-N}) = s' > 0,
	\]
	where $s'$ is the unique positive solution to the equation $|r_1|^{-Ns'} + |r_2|^{-Ns'} = 1$.
	
	Finally, suppose that $\F$ is degenerate and that
	$r_1,\ldots,r_m$ are algebraic numbers. By
	Lemma \ref{algebraicratios}, every non-singleton forward orbit of
	$\F$ is uniformly discrete. Hence Theorem \ref{bedimoflocallyfiniteorbits}
	implies
	\[
	\dimbe\fo_{\F}(a)=d_{\F}=0
	\]
	for every $a\in\R$.
\end{proof}

	\begin{lemma}\label{overlap}
		Given sets $E_1, \ldots, E_m$, if there exists a constant $C>0$ such that $\#(E_i \cap E_j) \le C$ for all $i \neq j$, then
		\[
		\sum_{i=1}^m \# E_i-\#\Big(\bigcup_{i=1}^m E_i\Big)\le \f{m(m-1)}2 C.
		\]
	\end{lemma}
	\begin{proof}
		Let $S=\bigcup_{i=1}^m E_i$ and define
		\[
		k(x)=\#\{1\le i\le m:x\in E_i\},\qquad x\in S.
		\]
		Suppose that $x\in E_{i_1}\cap E_{i_2}\cap\ldots\cap E_{i_{k(x)}}$ for $1\le i_1<\ldots<i_{k(x)}\le m.$ Then
		\[
		\sum_{1\le i<j\le m}\#(E_i\cap E_j)=\sum_{x\in S}\sum_{1\le s<t\le k(x)}1=\sum_{x\in S} \begin{pmatrix}
			k(x)\\
			2
		\end{pmatrix}.
		\]
		Therefore,
		\begin{equation*}
			\sum_{i=1}^m \# E_i-\# S=\sum_{x\in S}(k(x)-1)\le \sum_{x\in S}\begin{pmatrix}
				k(x)\\
				2
			\end{pmatrix}=\sum_{1\le i<j\le m}\#(E_i\cap E_j)\le \begin{pmatrix}
				m\\
				2
			\end{pmatrix}C.    \qedhere
		\end{equation*}  
	\end{proof}

	\subsection{Polynomial orbit growth under finite overlaps}

Theorems \ref{massdimoflocallyfiniteorbits} and \ref{bedimoflocallyfiniteorbits}, together with Proposition \ref{df=s}, show that, under the corresponding discreteness assumptions, the relevant dimensions of forward orbits of a non-degenerate system are equal to $s$. We now strengthen these dimension identities by imposing the additional finite-overlap condition. More precisely, we prove that, for all sufficiently large $h$, the quantities
\[
\#\bigl(\fo_\F(a) \cap [-h,h]\bigr)\quad\text{and}\quad\sup_{x\in\R} \#\bigl(\fo_\F(a) \cap [x-h, x+h]\bigr)
\]
under their respective hypotheses, are comparable to $h^s$.
	\begin{theorem}\label{mass}
		Let $\F$ be a non‑degenerate RIFS given by \eqref{IFSform} and let $a \in \R$. If the orbit $\fo_\F(a)$ is locally finite and has finite overlaps, then there exists a constant $C_1 > 1$ such that
		\[
		C_1^{-1} h^s \le \#\bigl(\fo_\F(a) \cap [-h,h]\bigr) \le C_1 h^s
		\]
		for all sufficiently large $h$, where $s$ is defined in \eqref{dimension}. In particular, $d_\F=\dimM \fo_{\F}(a)=s$.
		
		Furthermore, if $\fo_\F(a)$ is uniformly locally finite and has finite overlaps, then there exists a constant $C_2 > 1$ such that
		\[
		C_2^{-1} h^s \le \sup_{x\in\R} \#\bigl(\fo_\F(a) \cap [x-h, x+h]\bigr) \le C_2 h^s
		\]
		for all sufficiently large $h$. In particular, $d_\F=\dimM \fo_{\F}(a)=\dimbe \fo_{\F}(a)=s\le1$.
	\end{theorem}

	\begin{proof}
	Let $\mu_a$ be the counting measure on $\fo_\F(a)$, and set $L=\fo_\F(a)\backslash\bigcup_{i=1}^m f_i(\fo_\F(a))$. Since $\fo_\F(a)$   has finite overlaps, there exists $C>0$ such that 
	\[
	\mu_a(f_i(\fo_\F(a))\cap f_j(\fo_\F(a)))\le C,\quad i\ne j.
	\]
	
	For a bounded set $E\subseteq\R$ and $\bi\in\Si^*\cup\{\emptyset\}$,
	define
	\[
	\begin{aligned}
		O_{\bi}(E)=\sum_{i=1}^m\#\bigl(
		f_{\bi}^{-1}(E)\cap f_i(\fo_{\F}(a))\bigr)-\#\left(\bigcup_{i=1}^m\bigl(f_{\bi}^{-1}(E)\cap f_i(\fo_{\F}(a))\bigr)\right),
	\end{aligned}
	\]
	where $f_{\emptyset}$ is the identity map. By
	Lemma \ref{overlap},
	\begin{equation}\label{-}
		0\leq O_{\bi}(E)\leq\f{m(m-1)}{2}C.
	\end{equation}
	Since each $f_i$ is bijective, we have
	\[
	\begin{aligned}
		\mu_a\bigl(f_{\bi}^{-1}(E)\bigr)&=\mu_a\Big(\bigcup_{i=1}^m\big(f_{\bi}^{-1}(E)\cap f_i(\fo_\F(a))\big)\Big)+\mu_a( E\cap L)\\
		&=\sum_{i=1}^m\mu_a\bigl(f_{\bi i}^{-1}(E)\bigr)
		+\mu_a\bigl(f_{\bi}^{-1}(E)\cap L\bigr)
		-O_{\bi}(E).
	\end{aligned}
	\]
	As in the proof of Lemma \ref{tree},
	$\Lambda_k(\oor)\cup\Xi_k(\oor)$ forms the vertex set of a finite
	full $m$-ary tree whose leaves are the elements of
	$\Lambda_k(\oor)$. Iterating the preceding identity along this tree
	gives
	\begin{equation}\label{countcut}
		\begin{aligned}
			\mu_a(E)=
			\sum_{\bi\in\Lambda_k(\oor)}
			\mu_a\bigl(f_{\bi}^{-1}(E)\bigr)+
			\sum_{\bi\in\Xi_k(\oor)}
			\mu_a\bigl(f_{\bi}^{-1}(E)\cap L\bigr)-
			\sum_{\bi\in\Xi_k(\oor)}O_{\bi}(E).
		\end{aligned}
	\end{equation}
	
	For every $\bi\in\Si^*$, we have
	\[
	f_{\bi}^{-1}(x)
	=
	r_{\bi}^{-1}x+f_{\bi}^{-1}(0).
	\]
	By \eqref{premageof0}, $|f_{\bi}^{-1}(0)|<c$, and hence
	\[
	|r_{\bi}|^{-1}|x|-c
	\leq
	|f_{\bi}^{-1}(x)|
	\leq
	|r_{\bi}|^{-1}|x|+c.
	\]
	Consequently, whenever $|r_{\bi}|^{-1}h>c$,
	\begin{equation}\label{contain}
		\begin{aligned}
			[-|r_{\bi}|^{-1}h+c,\ |r_{\bi}|^{-1}h-c]\subseteq
			f_{\bi}^{-1}([-h,h])\subseteq
			[-|r_{\bi}|^{-1}h-c,\ |r_{\bi}|^{-1}h+c].
		\end{aligned}
	\end{equation}
	Proposition \ref{onX} (iv) gives
	\[
	\mu_a\bigl(f_{\bi}^{-1}(E)\cap L\bigr)
	\leq
	\mu_a(L)
	\leq m.
	\]
	Combining \eqref{countcut}, \eqref{-}, and \eqref{contain} with
	Lemma \ref{tree} and \eqref{cutnumber}, we obtain
	\begin{align}
	&\,\Big(\mu_a([-\oor^{-k-1}h+c,\oor^{-k-1}h-c])-\f{Cm\oor^{s}}{2}\Big)\oor^{ks}\notag\\
	\le&\, \sum_{\bi\in\Lambda_k(\oor)}\mu_a([-|r_\bi|^{-1}h+c,|r_\bi|^{-1}h-c])-\f{m(m-1)}2C\#\Xi_k(\oor)\notag\\
	\le&\ \mu_a([-h,h])\notag\\
	\le&\, \sum_{\bi\in\Lambda_k(\oor)}\mu_a([-|r_\bi|^{-1}h-c,|r_\bi|^{-1}h+c])+m\#\Xi_k(\oor)\notag\\
	\le&\,\big( \mu_a([-\oor^{-k}h-c,\oor^{-k}h+c])+m\big)\oor^{(k+1)s}.\label{increaseconmtrol}
\end{align}
	
Since $\F$ is non-degenerate, Proposition \ref{onX} (v) implies
that $\fo_{\F}(a)$ is infinite and unbounded. We may therefore choose $\ell\in\Z_{>0}$ sufficiently
large that $\oor^{\ell-1}>c$ and
\[
\mu_a\bigl(
[-\oor^{\ell-1}+c,\ \oor^{\ell-1}-c]
\bigr)
\geq
\f{Cm\oor^s}{2}+1.
\]
For each sufficiently large $h$, let $k=k(h)$ be the unique integer
such that
\[
\oor^{k+\ell}\leq h<\oor^{k+\ell+1}.
\]
Applying the upper estimate in \eqref{increaseconmtrol}, we obtain
\[
\begin{aligned}
	\mu_a([-h,h])
	&\le\left(\mu_a\bigl([-\oor^{\ell+1}-c,\ \oor^{\ell+1}+c]\bigr)+m\right)\oor^{(k+1)s}\\
	&\leq\left(\mu_a\bigl([-\oor^{\ell+1}-c,\oor^{\ell+1}+c]\bigr)+m\right)h^s.
\end{aligned}
\]
On the other hand, the lower estimate gives
\begin{align}
	\mu_a([-h,h])\geq\left(\mu_a\bigl([-\oor^{\ell-1}+c,\oor^{\ell-1}-c]\bigr)-\f{Cm\oor^s}{2}\right)\oor^{ks}\geq\oor^{ks}>\oor^{-(\ell+1)s}h^s.\label{kb}
\end{align}
Thus there exists $C_1>1$ such that
\begin{equation}\label{h^s}
	C_1^{-1}h^s
	\leq
	\#\bigl(\fo_{\F}(a)\cap[-h,h]\bigr)
	\leq
	C_1h^s
\end{equation}
for all sufficiently large $h$. It follows that
$\dimM\fo_{\F}(a)=s$, and
Theorem \ref{massdimoflocallyfiniteorbits} therefore yields
\[
d_{\F}=\dimM\fo_{\F}(a)=s.
\]

Suppose now that $\fo_{\F}(a)$ is uniformly locally finite, and set
\[
Q=\sup_{x\in\R}\#\bigl(\fo_{\F}(a)\cap[x-1,x+1]\bigr)<\infty.
\]
For each sufficiently large $h$, let $k'=k'(h)$ be the unique
integer satisfying
\[
\oor^{k'-1}\leq h<\oor^{k'}.
\]
For every $\bi\in\Lambda_{k'}(\oor)$, we have
$|r_{\bi}|\geq\oor^{k'}>h$. Hence
$f_{\bi}^{-1}([x-h,x+h])$ is an interval of radius
$h/|r_{\bi}|<1$, and therefore
\[
\mu_a\bigl(f_{\bi}^{-1}([x-h,x+h])\bigr)
\leq Q.
\]
Using \eqref{countcut}, \eqref{cutnumber}, and
$\mu_a(L)\leq m$, we obtain
\[
\begin{aligned}
	\mu_a([x-h,x+h])\le Q\#\Lambda_{k'}(\oor)+
	m\#\Xi_{k'}(\oor)\le(Q+m)\oor^{(k'+1)s}\le(Q+m)\oor^{2s}h^s.
\end{aligned}
\]
Taking the supremum over $x\in\R$ and combining this estimate with
\eqref{h^s}, we find a constant $C_2>1$ such that
\[
C_2^{-1}h^s
\leq
\sup_{x\in\R}
\#\bigl(\fo_{\F}(a)\cap[x-h,x+h]\bigr)
\leq
C_2h^s
\]
for all sufficiently large $h$. Finally,
Theorem \ref{bedimoflocallyfiniteorbits} gives
\[
d_{\F}=\dimM\fo_{\F}(a)=\dimbe\fo_{\F}(a)=s\leq1.
\]
\end{proof}
	
	Theorem \ref{mass} extends Theorem 3.1 of \cite{Deng2009} by relaxing several assumptions: specifically, $\F$ need not be defined on a uniformly discrete set in Euclidean space, the starting point $a$ need not belong to $P(\F)$, and the forward orbit is not required to be non-overlapping.
\begin{corollary}
	Suppose that there exists $a\in\R$ such that $\fo_{\F}(a)$ is
	uniformly discrete and has finite overlaps. Then
	\[
	d_\F=\dimB K(\F^{-1})=\dimH K(\F^{-1})=\begin{cases}
		0,\quad &\text{\ if\ }\F\text{\ is degenerate},\\
		s, &\text{\ if\ }\F\text{\ is non-degenerate}.
	\end{cases}
	\]
\end{corollary}
\begin{proof}
	Suppose first that $\F$ is degenerate. Then $K(\F^{-1})$ is a
	singleton, and Corollary \ref{df=0} gives
	\[
	d_\F=\dimB K(\F^{-1})=\dimH K(\F^{-1})=0.
	\]
	
	Now suppose that $\F$ is non-degenerate. By
	Proposition \ref{onX} (v), the orbit $\fo_{\F}(a)$ is infinite.
	Theorem \ref{mass} yields $d_\F=s\le 1$. Corollary \ref{dimensiondrop} therefore gives $\dimH K(\F^{-1})=s$. On the other hand, Lemma \ref{dfrange} gives
	\[
	\dimH K(\F^{-1})\le\dimB K(\F^{-1})\le d_{\F}
	\le s.
	\]
	Combining these relations, we obtain the desired result.
\end{proof}

	Finally, we provide two examples that highlight the role of the algebraic conditions and the interplay between local finiteness, uniform discreteness, and the various notions of dimension.
	
	We first revisit Example \ref{Be=infty} to establish that the algebraic hypothesis on the expansion ratios $r_i$ in Proposition \ref{degenerate} is indispensable. Specifically, in the absence of this condition, the Beurling dimension of a forward orbit may diverge to infinity, even if its mass dimension remains zero. More generally, the Beurling dimension of a forward orbit need not be bounded by the ambient dimension unless the set satisfies uniform local finiteness.
	\begin{example}[Example \ref{Be=infty} revisited]\label{Bedimsion=infty}
		Let $\{a_n\}_{n=1}^\infty, \{b_n\}_{n=1}^\infty$ be the sequences defined by
		\[
		a_1=b_1=1,\quad a_{n+1}=a_n\big\lceil\e^{nb_n}\big\rceil,\quad b_{n+1}=a_{n+1}\f{b_n}{a_n}+1,\quad n\ge 1,
		\]
		and let $\delta = \lim_{n\to\infty} b_n/a_n$.  Consider the RIFS $\F=\{\e x,\ \e^{\delta} x\}$. Then for every $a\in\R\setminus\{0\}$,
		\[
		\dimM\fo_\F(a)=0 \qquad\text{and}\qquad \dimbe\fo_\F(a)=\infty.
		\]
	\end{example}
	
	\begin{proof}
		Since $\F$ is degenerate at $0$, by Proposition \ref{degenerate}, we have $\dimM\fo_\F(a)=0$ for all $a\in\R$.
		
		However, recall from Example \ref{Be=infty} that for $a \neq 0$ and any $h > 0$, we have 
		\[
		\sup_{x\in\R}\#\bigl(\fo_\F(a)\cap[x-h,x+h]\bigr)=\infty .
		\]
		It is a direct consequence of the definition of the Beurling dimension that $\dimbe\fo_\F(a) = \infty$ whenever $a \neq 0$.
	\end{proof}

	The previous example dealt with a degenerate system. We now present a non-degenerate RIFS admitting a forward orbit that is locally finite and non-overlapping, yet whose mass dimension can indeed be strictly greater than $1$. This illustrates the first part of Theorem \ref{mass} and demonstrates that, absent the stronger assumption of uniform discreteness, the mass dimension of a locally finite orbit need not be constrained by the dimension of the ambient space.
	
	\begin{example}\label{ex_LF}
	Let $N\geq 1$ be an integer, and consider the RIFS
	\[
	\F=\{f_1(x)=\e x,\ f_2(x)=\e x+1,\ldots,
	f_{N+1}(x)=\e x+N\}
	\]
	on $\R$, where $\e$ denotes Euler's number. Then the forward orbit
	$\fo_\F(1)$ is locally finite and non-overlapping, and
	\[
	\dimM\fo_\F(1)=s=\log(N+1).
	\]
	Moreover, the inverse system $\F^{-1}$ has no exact overlaps.
	
	If $N\geq2$, then $\F$ admits no uniformly locally finite forward
	orbits and, consequently, no uniformly locally finite invariant sets.
	\end{example}
	
	\begin{proof}
		The convex hull of $K(\F^{-1})$ is $[-\f{N}{\e-1},0]$ for every $N\ge 1$. Since every map of $\F$ is surjective and $1\notin K(\F^{-1})$, Proposition \ref{onX} (vi) implies that $\fo_\F(1)$ is locally finite.
		
		We next show that the orbit is non-overlapping. Suppose that $
		f_\bi(1)=f_\bj(1)$
		for some finite words $\bi,\bj\in\Sigma^*$. We may write
		\[
		f_\bi(1)=\e^{|\bi|}+P_\bi(\e),
		\qquad
		f_\bj(1)=\e^{|\bj|}+P_\bj(\e),
		\]
		where $P_\bi,P_\bj\in\Z[x]$, with $
		\deg P_\bi<|\bi|$ and $
		\deg P_\bj<|\bj|.$
		The coefficients of these polynomials are precisely the translation
		digits associated with the corresponding words. Since $\e^{|\bi|}+P_\bi(\e)=\e^{|\bj|}+P_\bj(\e)$
	and $\e$ is transcendental, the two integer polynomials
		\[
		x^{|\bi|}+P_\bi(x)
		\quad\text{and}\quad
		x^{|\bj|}+P_\bj(x)
		\]
		must coincide. It follows first that $|\bi|=|\bj|$ and then, by
		comparing their coefficients, that $\bi=\bj$. Hence
		\[
		f_\bi(1)\neq f_\bj(1)
		\qquad\text{whenever}\qquad\bi\neq\bj.
		\]
		In particular, the sets $
		f_i\bigl(\fo_\F(1)\bigr),
		\, 1\leq i\leq N+1,$
		are pairwise disjoint. Therefore, $\fo_\F(1)$ is non-overlapping. By Theorem \ref{mass}, we have 
		$$
		\dimM\fo_{\F}(1)=s=\ln(N+1).
		$$ 
		The fact that $f_{\bi}(1)\neq f_{\bj}(1)$ for $\bi\neq\bj$ also shows that the maps $f_{\bi}$ and $f_{\bj}$ are distinct; thus $\F^{-1}$ admits	no exact overlaps.
		
		Finally, suppose that $N\geq2$ and that $\fo_\F(a)$ is uniformly
		locally finite for some $a\in\R$. By
		Theorem \ref{bedimoflocallyfiniteorbits} and
		Proposition \ref{df=s}, we have
		\[
		\dimbe\fo_{\F}(a)=d_\F=s=\ln(N+1)\le 1,
		\]
		which is impossible.
		Therefore, $\F$ admits no uniformly locally finite forward orbits, and consequently, no uniformly locally finite invariant sets.
	\end{proof}
	
	\begin{remark}
		For $N \ge 2$, it is clear that $K(\F^{-1})=[-\f{N}{\e-1},0]$ and the dual IFS does not satisfy the open set condition. Even though $K(\F^{-1})$ exhibits a considerable degree of overlap, $\F$ still admits an abundance of non-overlapping locally finite forward orbits.  
	\end{remark}

	\section{Renewal asymptotics for forward-orbit counting}\label{sec_Cden}

	In this section, we apply the renewal theorem to determine the asymptotic central densities of forward orbits.
	For convenience, we recall that, for $\alpha>0$, the upper and lower $\alpha$-central densities of a countable set $E\subseteq\R$ are defined by
	\begin{equation}\label{def_CA+-}
		\mathcal C_\alpha^+(E)
		=
		\limsup_{h\to\infty}
		\f{\#(E\cap[-h,h])}{h^\alpha},
		\qquad
		\mathcal C_\alpha^-(E)
		=
		\liminf_{h\to\infty}
		\f{\#(E\cap[-h,h])}{h^\alpha}.
	\end{equation}
	If these quantities agree, their common value is denoted by $\mathcal C_\alpha(E)$.

	\subsection{The main asymptotic theorem}
	
	\begin{theorem}\label{asymptotic}
		Let $\F$ be a non-degenerate RIFS of the form \eqref{IFSform}, let
		$a\in\R$, and let $s$ be defined by \eqref{dimension}. Suppose that
		$\fo_\F(a)$ is non-overlapping and uniformly locally finite. Then the
		following statements hold.
		\begin{enumerate}[(i)]
			\item Suppose that $
			\f{\log|r_i|}{\log|r_j|}\notin\Q$
			for some $i\neq j$. Then the upper and lower $s$-central densities
			of $\fo_\F(a)$ coincide. Their common value is
			\[
		\begin{aligned}
			\mathcal{C}_s\bigl(\fo_\F(a)\bigr)&= \f{1}{ s\sum_{i=1}^{m} |r_i|^{-s}\log |r_i|}    \Biggl[\,
			\sum_{x\in\fo_\F(a)\cap (-1,1)}\Bigl(\,\sum_{i=1}^{m}\min\big\{1,|r_ix|^{-s}\big\}-1\Bigr)  \\
			&\hspace{3.9cm} 
			+ \,\sum_{x\in\fo_\F(a)} \sum_{i=1}^{m}
			\Bigl( \min\big\{1,|r_i x+b_i|^{-s}\big\} - \min\big\{1,|r_i x|^{-s}\big\} \Bigr)\\
			&\hspace{6.8cm} + \sum_{x\in\fo_\F(a)\setminus\bigcup_{i=1}^m f_i(\fo_\F(a))}\min\big\{1,|x|^{-s}\big\}
			\Biggr],
		\end{aligned}
		\]
			Here and throughout, we adopt the convention $0^{-s}=\infty$.
			
			\item Suppose that $
			\f{\log|r_i|}{\log|r_j|}\in\Q
			\,\text{for all }1\leq i,j\leq m.$
			Equivalently, there exist $r>1$ and positive integers
			$k_1,\ldots,k_m$ with
			\[
			\gcd(k_1,\ldots,k_m)=1
			\qquad\text{and}\qquad
			|r_i|=r^{k_i}
			\quad (1\leq i\leq m).
			\]
			Then there exist a constant $C'>1$ and a function
			$\phi\colon[1,\infty)\to(0,\infty)$ such that
			\[
			{C'}^{-1}\leq\phi(h)\leq C',
			\qquad
			\phi(rh)=\phi(h)
			\quad\text{for all }h\geq1,
			\]
			and
			\[
			\lim_{h\to\infty}
			\left|
			\f{
				\#\bigl(\fo_\F(a)\cap[-h,h]\bigr)
			}{h^s}
			-\phi(h)
			\right|
			=0.
			\]
		\end{enumerate}
	\end{theorem}

In the arithmetic case, the $s$-central density exists precisely when the periodic profile $\phi$ is constant; Example \ref{eg_cdnote} shows that
nonconstant profiles do occur.
	The proof of Theorem \ref{asymptotic}  relies on the following renewal theorem, and we refer readers to \cite{Asmussen2003} for details.	
	\begin{theorem}\label{renewal}
		Assume that $Z(t)$ is directly Riemann integrable on $[0,+\infty)$, and let $H(t)$ be the locally bounded solution to the renewal equation
		\[
		H(t)=Z(t)+\int_0^t H(t-s) \text{d} F(s),\quad t\ge 0,
		\]
		where $F(s)$ is a cumulative distribution function on $[0,+\infty)$. Then $H=\big(\sum_{n=0}^{\infty}F^{*n}\big)*Z$. 
		\begin{enumerate}[(i)]
			\item If $F$ is non-arithmetic, then
			\[
			\lim_{t\to+\infty}H(t)=\left\{\begin{array}{ll}
				\f{\int_0^{+\infty}Z(t)\text{d}t}{\int_0^{+\infty}t\text{d}F(t)},\quad &\int_0^{+\infty}t\dd F(t)<+\infty\vspace{0.3em},\\
				0, &\int_0^{+\infty}t\text{d}F(t)=+\infty.
			\end{array}\right.
			\]
			\item If $F$ is arithmetic with span $d$, then for $0\le x<d$,
			\[
			\lim_{n\to\infty}H(x+nd)=\left\{\begin{array}{ll}
				\f{d}{\int_0^{+\infty}t\text{d}F(t)}\sum_{k=0}^\infty Z(x+kd),\quad &\int_0^{+\infty}t\text{d}F(t)<+\infty\vspace{0.3em},\\
				0, &\int_0^{+\infty}t\text{d}F(t)=+\infty.
			\end{array}\right.
			\]
		\end{enumerate}
	\end{theorem}

	\begin{proof}[Proof of Theorem \ref{asymptotic}]
		We restrict our attention to the case where all $r_i$ are positive numbers; the general case can be treated analogously.

		Let $\mu_a$ be the counting measure on $\fo_\F(a)$ and let  $L=\fo_\F(a)\backslash\bigcup_{i=1}^m f_i(\fo_\F(a))$. Since $\fo_\F(a)$ is non-overlapping, 
		$$
		\mu_a(E)=     \sum_{i=1}^m \mu_a(f_i^{-1}(E))+\mu_a(E\cap L).
		$$
		It follows  that for all $h>0$, 
		\begin{equation}\label{=}
			\mu_a([-h,h])=\sum_{i=1}^m\mu_a\Big(\Big[-\f{h+b_i}{r_i},\f{h-b_i}{r_i}\Big]\Big)+\mu_a([-h,h]\cap L)=\sum_{i=1}^m\mu_a\Big(\Big[-\f{h}{r_i},\f{h}{r_i}\Big]\Big)+\widetilde{L}(h),
		\end{equation}
		where 
		\[
		\widetilde{L}(h)=\sum_{i=1}^m\Big(\mu_a\Big(\Big[\f{-h-b_i}{r_i},\f{h-b_i}{r_i}\Big]\Big)-\mu_a\Big(\Big[-\f{h}{r_i},\f{h}{r_i}\Big]\Big)\Big)+\mu_a([-h,h]\cap L).
		\]
		Note that for each $1\le i\le m$, the points of discontinuity of the function 
		\[
		h \mapsto \mu_a \left( \left[ \f{-h-b_i}{r_i}, \f{h-b_i}{r_i} \right] \right) - \mu_a \left( \left[ -\f{h}{r_i}, \f{h}{r_i} \right] \right)
		\]
		are contained in the uniformly locally finite set
		\[
		\big\{ r_i|x| : x \in \fo_\F(a) \big\} \cup \big\{ |r_i x + b_i| : x \in \fo_\F(a) \big\},
		\]
		and the function $h \mapsto \mu_a([-h,h]\cap L)$ has at most $m$ discontinuities. Since $\fo_{\F}(a)$ is uniformly locally finite and $\mu_a([-h,h]\cap L)\le m$, there exists a constant $M>0$ such that
		\begin{equation}\label{boundedabove}
			\begin{aligned}
				|\widetilde{L}(h)|\le\sum_{i=1}^m\left|\mu_a\Big(\Big[\f{-h-b_i}{r_i},\f{h-b_i}{r_i}\Big]\Big)-\mu_a\Big(\Big[-\f{h}{r_i},\f{h}{r_i}\Big]\Big)\right|+\mu_a([-h,h]\cap L)\le M.
			\end{aligned}
		\end{equation}
		Therefore, $\widetilde{L}(h)$ is Riemann integrable over any finite interval in $[1,+\infty)$. 
		
		Denote $H(t)=\e^{-s t}\mu_a([-\e^t,\e^t])$ and $G(t)=\e^{-st}\widetilde{L}(\e^t)$ for $ t\in \R$ with $\sum_{i=1}^m r_i^{-s} = 1$. Note that
		\[
		\int_0^t H(t-y)\, \text{d}F(y)=\sum_{i:\log r_i\le t}r_i^{-s}H(t-\log r_i),\qquad t\ge0,
		\]
		where $F$ is the probability measure that puts mass $r_i^{-s}$ at $\log r_i$ for $1\le i\le m$.
		Therefore, we have
		\[
		H(t)=Z(t)+\int_0^t H(t-y)\,\text{d}F(y),\quad t\ge0,
		\]
		where, by \eqref{=},
		\begin{equation}\label{Z(t)}
			Z(t)=H(t)-\sum_{i:\log r_i\le t}r_i^{-s}H(t-\log r_i)=G(t)+\sum_{i:\log r_i> t}r_i^{-s}H(t-\log r_i).
		\end{equation}
		By \eqref{boundedabove}, $|G(t)|$ is bounded by a monotonically non-increasing, Lebesgue integrable function. Since $G(t)$ is Riemann integrable on finite intervals, it is directly Riemann integrable on $[0,+\infty)$ according to \cite[Chapter 5, Proposition 4.1]{Asmussen2003}. Furthermore, as the sum $\sum_{i:\log r_i> t} r_i^{-s} H(t-\log r_i)$ is bounded, piecewise continuous, and compactly supported in $[0,+\infty)$, it follows by the same proposition that $Z(t)$ is also directly Riemann integrable.
		
		(i) Since $\f{\log r_i}{\log r_j}\notin \Q$ for some $i$ and $j$,  $F$ is non-arithmetic.  
		By \eqref{Z(t)}, we are able to explicitly evaluate $\int_0^\infty Z(t)\dd t$ in two parts.  Since 
		\[
		\int_{0}^{\log r_i} H(t-\log r_i)\dd t
		= \int_{-\log r_i}^{0} H(u)\dd u
		= \int_{r_i^{-1}}^{1} v^{-s-1}\mu_a([-v,v])\dd v, 
		\]
		we have 
		\begin{equation} \label{eq_dpart}
			\int_{0}^{\infty}\!\!\!\!\!  \sum_{i:\log r_i>t}\!\!\! r_i^{-s} H(t-\log r_i) \dd t
			= \sum_{i=1}^{m} r_i^{-s}\int_{0}^{\log r_i}\!\!\! H(t-\log r_i)\dd t =\sum_{i=1}^{m} r_i^{-s}\int_{r_i^{-1}}^{1} \f{\mu_a([-v,v])}{v^{s+1}}\dd v .
		\end{equation}
		The integral involves only orbit points with $|v|\le 1$, so it is a finite sum.
		
		Since   $\mu_a=\sum_{x\in\fo_\F(a)}\delta_x$ where $\delta_x$ denotes the Dirac mass at $x$, we may rewrite $\widetilde L(h)$ as 
		\[
		\widetilde L(h)=\sum_{x\in\fo_\F(a)}\Bigl(\sum_{i=1}^{m}
		\Bigl[\mathbf{1}_{[\,\f{-h-b_i}{r_i},\f{h-b_i}{r_i}\,]}(x)
		-\mathbf{1}_{[-\f{h}{r_i},\f{h}{r_i}]}(x)\Bigr]
		+ \ind_{L}(x)\mathbf{1}_{[-h,h]}(x)\Bigr).
		\]
		For a fixed $x$, the condition $x \in [\f{-h-b_i}{r_i}, \f{h-b_i}{r_i}]$ is equivalent to $|r_ix + b_i| \le h$, while $x \in [-\f{h}{r_i}, \f{h}{r_i}]$ corresponds to $ |r_ix|\le h$. Consequently, these indicators, when viewed as functions of $h$, can be reformulated as
		\[
		\mathbf{1}_{[\,\f{-h-b_i}{r_i},\f{h-b_i}{r_i}\,]}(x)
		=\mathbf{1}_{[\,|r_i x+b_i|,\infty)}(h),\qquad
		\mathbf{1}_{[-\f{h}{r_i},\f{h}{r_i}]}(x)
		=\mathbf{1}_{[\,|r_ix|,\infty)}(h).
		\]
		A direct computation shows that for any $A,B\ge0$,
		\begin{align*}
			\int_{1}^{\infty} h^{-s-1}\bigl(\mathbf{1}_{[A,\infty)}(h)
			-\mathbf{1}_{[B,\infty)}(h)\bigr)\dd h
			&=\f1s\Bigl(\big(\max\{1,A\}\big)^{-s}-\big(\max\{1,B\}\big)^{-s}\Bigr)\\
			&=\f1s\Bigl(\min\{1,A^{-s}\}-\min\{1,B^{-s}\}\Bigr).
		\end{align*}
		Applying this with $A=|r_i x+b_i|$, $B=|r_ix|$, we obtain
		\begin{eqnarray*}
			&&\int_{0}^{\infty} G(t)\dd t    =  \int_{0}^{\infty} e^{-st}\widetilde L(e^{t})\dd t
			= \int_{1}^{\infty} h^{-s-1}\widetilde L(h)\dd h  \nonumber\\ 
			&=& \f1s\sum_{x\in\fo_\F(a)}\Biggl(\sum_{i=1}^{m}
			\Bigl(\min\{1,|r_i x+b_i|^{-s}\}-\min\{1,|r_i x|^{-s}\}\Bigr)
			+ \ind_{L}(x)\,\min\{1,|x|^{-s}\}\Biggr).
		\end{eqnarray*}
		The absolute convergence of this double series is guaranteed by \eqref{boundedabove}. Moreover, we note that the last term simplifies to
		\[
		\sum_{x\in\fo_\F(a)}\ind_{L}(x)\,\min\{1,|x|^{-s}\}=\sum_{x\in L}\min\{1,|x|^{-s}\}.
		\]
		Therefore, we obtain
		\begin{equation}\label{eq_G-part}
			\int_{0}^{\infty} G(t)\dd t=\f1s\Biggl(\,\sum_{x\in\fo_\F(a)}\sum_{i=1}^{m}
			\Bigl(\min\{1,|r_i x+b_i|^{-s}\}-\min\{1,|r_i x|^{-s}\}\Bigr)
			+ \sum_{x\in L}\min\{1,|x|^{-s}\}\Biggr).
		\end{equation}
		
		Analogous arguments yield the following 
		\begin{align}
			\sum_{i=1}^{m} r_i^{-s}\int_{r_i^{-1}}^{1} \f{\mu_a([-v,v])}{v^{s+1}}\dd v&=\f1 s\sum_{x\in\fo_\F(a)\cap[-1,1]}\sum_{i=1}^{m} r_i^{-s}\bigl(\min\{r_i^s,|x|^{-s}\}-1\bigr)\notag\\
			&=\f1 s\sum_{x\in\fo_\F(a)\cap[-1,1]}\sum_{i=1}^{m} \bigl(\min\{1,|r_ix|^{-s}\}-r_i^{-s}\bigr)\notag\\
			&=\f1 s\sum_{x\in\fo_\F(a)\cap[-1,1]}\Bigl(\sum_{i=1}^{m} \min\{1,|r_ix|^{-s}\}-1\Bigr)\notag\\
			&=\f1 s\sum_{x\in\fo_\F(a)\cap(-1,1)}\Bigl(\sum_{i=1}^{m} \min\{1,|r_ix|^{-s}\}-1\Bigr).\label{LLL}
		\end{align}
		By combining \eqref{eq_dpart}, \eqref{eq_G-part}, and \eqref{LLL}, the total integral of $Z(t)$ admits the representation
		\begin{equation}
			\begin{aligned}
				\int_{0}^{\infty} Z(t)\dd t
				&= \f{1}{ s}    \Biggl[\,
				\sum_{x\in\fo_\F(a)\cap (-1,1)}\Bigl(\,\sum_{i=1}^{m}\min\big\{1,|r_ix|^{-s}\big\}-1\Bigr)  \\
				&\hspace{3cm} 
				+ \,\sum_{x\in\fo_\F(a)} \sum_{i=1}^{m}
				\Bigl( \min\big\{1,|r_i x+b_i|^{-s}\big\} - \min\big\{1,|r_i x|^{-s}\big\} \Bigr)\\
				&\hspace{8.6cm} + \sum_{x\in L}\min\big\{1,|x|^{-s}\big\}
				\Biggr].
			\end{aligned}
		\end{equation}
		The existence of $\int_0^{\infty} Z(t) dt$ is thus established. Furthermore, observing that the mean of the renewal distribution is given by
		\[
		\int_0^{+\infty} y\,\text{d}F(y)=\sum_{i=1}^m r_i^{-s}\log r_i<+\infty,
		\]
		it follows from  Theorem \ref{renewal} that 
		\[
		\mathcal{C}^+_s\bigl(\fo_\F(a)\bigr)=\mathcal{C}^-_s\bigl(\fo_\F(a)\bigr) =\lim_{h\to+\infty}\f{\mu_a([-h,h])}{h^{s}}
		= \lim_{t\to+\infty} H(t)     = \f{ \int_{0}^{\infty} Z(t)\dd t}
		{ \sum_{i=1}^{m} r_i^{-s}\log r_i}.
		\]
		This yields exactly the conclusion stated in the theorem.

		(ii) Since  there exist a real number $r > 1$ and positive integers $k_1, \dots, k_m$ with $\gcd(k_1, \dots, k_m) = 1$ such that $r_i = r^{k_i}$ for all $1 \le i \le m$, $F$ is arithmetic  with span $\log r$. Similar to (i), for $x\in[0,\log r)$, we have
		\begin{align*}
			&\sum_{n=0}^\infty Z(x+n\log r)=\lim_{n\to\infty}\Big(\sum_{k=0}^n H(x+k\log r)-\sum_{k=0}^n\sum_{i:\log r_i\le x+k\log r}r_i^{-s}H(x+k\log r-\log r_i)\Big)\\
			&=\lim_{n\to\infty}\Big(\sum_{k=0}^n H(x+k\log r)-\sum_{i=1}^m\sum_{k=k_i}^nr_i^{-s}H(x+k\log r-\log r_i)\Big)\\
			&=\e^{-s x}\lim_{n\to\infty}\Big(\sum_{k=0}^n r^{-s k}\mu_a([-\e^xr^k,e^xr^k])-\sum_{i=1}^m\sum_{k=k_i}^nr^{-s k}\mu_a([-\e^xr^{k-k_i},\e^xr^{k-k_i}])\Big)\\
			&\ge\e^{-s x} \varliminf_{n\to\infty}\Big(\sum_{k=0}^n r^{-s k}\mu_a([-\e^xr^k,e^xr^k])-\sum_{i=1}^mr^{-s k_i}\sum_{k=0}^{n-\min_ik_i}r^{-s k}\mu_a([-\e^xr^k,e^xr^k])\Big)\\
			&=\e^{-s x}\varliminf_{n\to\infty}\sum_{k=n-\min_i k_i+1}^nr^{-s k}\mu_a([-\e^xr^k,e^xr^k]).
		\end{align*}
		By Theorem \ref{mass}, we obtain
		\[
		\sum_{n=0}^\infty Z(x+n\log r)\ge C\ur, \qquad x\in [0,\log r).
		\]
		By  \eqref{boundedabove} and \eqref{Z(t)}, we further obtain that
		\begin{align}
			&\sum_{n=0}^\infty |Z(x+n\log r)|=\sum_{n=0}^{\max k_i} |Z(x+n\log r)|+\sum_{n=\max k_i+1}^{\infty}|G(x+n\log r)|\notag\\
			&\le (\max k_i+1)\sup_{t\in [0,(\max k_i+1)\log r)}|Z(t)|+M\e^{-sx}\sum_{n=\max k_i+1}^{\infty}r^{-sn}\notag\\
			&< (\max k_i+1)\sup_{t\in [0,(\max k_i+1)\log r)}|Z(t)|+\f{M}{1-r^{-s}}\eqcolon M',\qquad x\in [0,\log r).\label{M'}
		\end{align}
		
		Define
		\[
		\tilde{\phi}(x)=\f{\sum_{n=0}^\infty Z(x\bmod \log r+n\log r)}{\sum_{i=1}^m r_i^{-s}\log r_i}\log r, \qquad x\ge 0.
		\]
		It follows that $\tilde{\phi}(x)$ is bounded and bounded away from zero on $[0,+\infty)$, and is periodic with period $\log r$. Since $\int_0^{+\infty} y\,\text{d}F(y)<+\infty,$ by Theorem \ref{renewal}, it suffices to show that $\{H(x+n\log r)\}_{n=1}^\infty$ converges uniformly to $\tilde{\phi}(x)$ on $[0,\log r)$. 
		
		Since $F = \sum_{i=1}^m r^{-k_i s} \delta_{k_i \log r}$ in this case, it follows that 
		\[
		\sum_{n=0}^\infty F^{*n} = \sum_{k=0}^\infty u_{k} \delta_{k \log r},
		\]
		where $\delta_a$ denotes the Dirac mass at $a$. Moreover,
		by \cite[Chapter 1, Corollary 2.3]{Asmussen2003}, we have
		\begin{equation*}
			\lim_{k\to\infty} u_k=\f{\log r}{\sum_{i=1}^m r_i^{-s}\log r_i}\eqcolon u.
		\end{equation*}
		Hence, $\sup_k|u_k|<C'$ for some $C'>0$ and  for each $\eps>0$, there is $N$ such that
		\begin{equation}\label{controlh}
			|u_k-u|<\eps,\quad k>N.
		\end{equation}
		It follows from Theorem \ref{renewal} that for each $n\in \Z_{>0}$,
		\begin{align*}
			H(x+n\log r)&=\Big(\sum_{k=0}^\infty F^{*k}\Big)*Z(x+n\log r)\\
			&=\sum_{k=0}^{\lfloor \f{x+n\log r}{\log r}\rfloor}u_k Z( x+n\log r-k\log r)\\
			&=\sum_{k=0}^nu_{n-k}Z( x+k\log r),\quad x\in[0,\log r).
		\end{align*}
		Finally, by \eqref{boundedabove}, \eqref{Z(t)}, \eqref{M'} and \eqref{controlh}, for sufficiently large $n$ (independent of $x$), we obtain
		\begin{align*}
			&\big|H( x+n\log r)-\tilde{\phi}(x)\big|\\
			\le\ & \sum_{k=0}^{n-N-1}|u_{n-k}-u||Z( x+k\log r)|+\sum_{k=n-N}^{n}|u_{n-k}-u||Z( x+k\log r)|+u\sum_{k=n+1}^{\infty}|Z( x+k\log r)|\\
			=\ & \sum_{k=0}^{n-N-1}|u_{n-k}-u||Z( x+k\log r)|+\sum_{k=n-N}^{n}|u_{n-k}-u||G( x+k\log r)|+u\sum_{k=n+1}^{\infty}|G( x+k\log r)|\\
			<\ & M'\eps+2C'(N+1)Mr^{-s(n-N)}+uM\sum_{k=n+1}^{\infty}r^{-sk}\longrightarrow M'\eps\quad (n\to\infty),
		\end{align*}
		yielding that
		\[
		\sup_{x\in [0,\log r)}\big|H( x+n\log r)-\tilde{\phi}(x)\big|\to 0\ (n\to\infty).
		\]
		Thus, $\tilde{\phi}\circ\log$ meets the requirements.
	\end{proof}

	\subsection{Homogeneous digit systems}

	The following example evaluates the upper and lower $s$-central densities of $\fo_\F(0)$ for a classical family of homogeneous RIFS on the integer lattice, thereby illustrating that the central density need not exist. This behavior is consistent with Theorem \ref{asymptotic} (ii), which establishes that the asymptotic density function exhibits multiplicatively periodic oscillations (although the amplitude of these oscillations may decay, allowing the ultimate limit to still exist).
	\begin{example}\label{eg_cdnote}
		Let $\lambda > 1$ be an integer and consider the RIFS $\F_D = \{f_i(x) = \lambda x + d_i\}_{d_i\in D}$, where $0 \in D \subseteq \{0, 1, \ldots, \lambda-1\}$ and $\#D \ge 2$.  Let $\mu_D$ be the homogeneous self-similar measure 
		 supported on 
		\[
		-K(\F_D^{-1})=\bigg\{\sum_{j=1}^\infty \lambda^{-j}d_{i_j}:\text{ each } d_{i_j}\in D\bigg\}.
		\]
		Let $s = \f{\log \#D}{\log \lambda}$,  and define 
		\[
		\qquad \phi_D(h)=\f{ \mu_D([0,h])}{h^s},\qquad h\in (0,1].
		\]
		Then $\fo_{\F_D}(0)$ is uniformly discrete, non-overlapping,  and the upper and lower $s$-central densities of $\fo_{\F_D}(0)$ are given by 
		\begin{equation}\label{eq_Cmm}
			\begin{split}
				\mathcal{C}^+_s\bigl(\fo_{\F_D}(0)\bigr)=\limsup_{h\to 0^+}\phi_D(h)=\max_{h\in [\lambda^{-1},1)}\phi_D(h),\\
				\mathcal{C}^-_s\bigl(\fo_{\F_D}(0)\bigr)=\liminf_{h\to 0^+}\phi_D(h)=\min_{h\in [\lambda^{-1},1)}\phi_D(h).
			\end{split}
		\end{equation}	
		In particular, when $\lambda=3$, we obtain the following exact values:
		\begin{equation}\label{eq_lambda=3}
			\begin{aligned}
				\mathcal{C}^+_s\bigl(\fo_{\F_{\{0,2\}}}(0)\bigr)&=1, &\qquad \mathcal{C}^-_s\bigl(\fo_{\F_{\{0,2\}}}(0)\bigr)&=2^{-\f{\log 2}{\log 3}},\\
				\mathcal{C}^+_s\bigl(\fo_{\F_{\{0,1\}}}(0)\bigr)&=2^{\f{\log 2}{\log 3}}, &\qquad \mathcal{C}^-_s\bigl(\fo_{\F_{\{0,1\}}}(0)\bigr)&=1.
			\end{aligned}
		\end{equation}
		Moreover, the $s$-central density of $\fo_{\F_{\{0,1,2\}}}(0)$ exists and $\mathcal{C}_s\bigl(\fo_{\F_{\{0,1,2\}}}(0)\bigr)=1.	$	
	\end{example}
	\begin{proof}
		By Proposition \ref{discrete} (i),  $\fo_{\F_D}(0)$ is uniformly discrete. Since $x \equiv d_i \pmod{\lambda}$ for all $x \in f_i\big(\fo_{\F_D}(0)\big)$, it is clear that
		\[
		f_i\big(\fo_{\F_D}(0)\big) \cap f_j\big(\fo_{\F_D}(0)\big) = \emptyset   \quad \text{ for}\quad  i \ne j. 
		\]
		Thus $\fo_{\F_D}(0)$ is non-overlapping. 
		
		For every integer $n\ge 1$, observe that $\big\{f_\bi(0):\bi\in \Si^n\big\}\subseteq \big\{f_\bi(0):\bi\in \Si^{n+1}\big\}\subset [0,\infty)$ and 
		\[
		\min \Big(\big\{f_\bi(0):\bi\in \Si^{n+1}\big\}\setminus\big\{f_\bi(0):\bi\in \Si^{n}\big\}\Big)=\min\big(D\setminus\{0\}\big)\lambda^n\ge\lambda^n.
		\]
		Write $g_i(x) = \lambda^{-1}(x+d_i)$ for $d_i \in D$. Then, for any $h \in [\lambda^{n-1}, \lambda^n)$, we have
		\begin{align*}
			\#\big(\fo_{\F_D}(0)\cap [-h,h]\big)&=\#\big(\fo_{\F_D}(0)\cap [0,h]\big)\\
			&=\#\big\{f_{\bi}(0):f_{\bi}(0)\le h,\bi\in \Si^n\big\}\\
			&=\#\big\{g_{\bi}(0): g_{\bi}(0)\le \lambda^{-n}h,\bi\in\Si^n\big\}.
		\end{align*}
		Since $\mu_D(g_{\bi}([0,1]))$
		$=(\#D)^{-n}$ for each $\bi\in \Si^n$ and $g_i((0,1))\cap g_j ((0,1))=\emptyset$ for distinct $i,j$, it follows that
		\[
		\#\big(\fo_{\F_D}(0)\cap [-h,h]\big)=(\#D)^n\mu_D([0,\lambda^{-n}h])+\epsilon(h),
		\]
		where $|\epsilon(h)|\le 1$. Therefore, we obtain
		\begin{equation}\label{density}
			\f{\#\big(\fo_{\F_D}(0)\cap [-h,h]\big)}{h^s}=\f{(\#D)^n\mu_D([0,\lambda^{-n}h])}{h^s}+\epsilon(h)h^{-s}=\phi_D(\lambda^{-n}h)+\epsilon(h)h^{-s}.
		\end{equation}
		Note that $\lambda^{-n}h\in [\lambda^{-1},1)$ and  $\phi_D(h)=\f{\mu_D([0,h])}{h^s},\,h\in (0,1]$ is continuous and satisfies
		\begin{align*}
			\phi_D(\lambda^{-1}h)&=\f{{(\# D)}^{-1}\sum_{d_i\in D}\mu_D(g_i^{-1}([0,\lambda^{-1}h]))}{\lambda^{-s}h^{s}}\\
			&=\f{{(\# D)}^{-1}\sum_{d_i\in D}\mu_D([-d_i,h-d_i]))}{\lambda^{-s}h^{s}}\\
			&=\f{{(\# D)}^{-1}\mu_D([0,h])}{\lambda^{-s}h^{s}}\\
			&=\f{\mu_D([0,h])}{h^s}=\phi_D(h).
		\end{align*}
		Since $\phi_D$ is continuous and $\phi_D(1) = \phi_D(\lambda^{-1})$, combining \eqref{density} gives
		\[
		\mathcal{C}^+_s\bigl(\fo_{\F_D}(0)\bigr)=\limsup_{h\to 0^+}\phi_D(h)=\max_{h\in [\lambda^{-1},1)}\phi_D(h).
		\]
		The same argument works for the lower central density.

		Let $\lambda = 3$. We begin by analyzing the case $D = \{0, 2\}$. Over the sub-interval $[1/3, 2/3]$, the measure remains constant, yielding
		\begin{gather*}
			\max_{h\in[1/3,\,2/3]}\phi_{\{0,2\}}(h)=\phi_{\{0,2\}}(1/3)=3^s\mu_{\{0,2\}}([0,1/3])=1,\\ \min_{h\in[1/3,\,2/3]}\phi_{\{0,2\}}(h)=\phi_{\{0,2\}}(2/3)=\Big(\f 32\Big)^s\mu_{\{0,2\}}([0,2/3])=2^{-s}.
		\end{gather*}
		For $h\in (2/3,1)$, the self-similarity of $\mu_{\{0,2\}}$ implies that 
		\begin{align*}
			\phi_{\{0,2\}}(h)=\f{\mu_{\{0,2\}}([0,1/3])+\mu_{\{0,2\}}([2/3,h])}{h^s}
			=\f{2^{-1}+2^{-1}\mu_{\{0,2\}}([0,3h-2])}{h^s}.
		\end{align*}
		Let $M_{\{0,2\}}=\max_{h\in[1/3,1)}\phi_{\{0,2\}}(h)\ge 1$ and
		$m_{\{0,2\}}=\min_{h\in[1/3,1)}\phi_{\{0,2\}}(h)\le 2^{-s}$.
		With the change of variables $t=3h-2$, the extrema on $(2/3,1)$
		satisfy
		\begin{gather*}
			\sup_{h\in(2/3,\,1)}\phi_{\{0,2\}}(h)=\sup_{t\in(0,1)}\f{1+\mu_{\{0,2\}}([0,t])}{(2+t)^s}\le \sup_{t\in(0,1)}\f{1+M_{\{0,2\}}t^s}{(2+t)^s}=\f{1+M_{\{0,2\}}}{2}\le M_{\{0,2\}},\\
			\inf_{h\in(2/3,\,1)}\phi_{\{0,2\}}(h)=\inf_{t\in(0,1)}\f{1+\mu_{\{0,2\}}([0,t])}{(2+t)^s}\ge \inf_{t\in(0,1)}\f{1+m_{\{0,2\}}t^s}{(2+t)^s}=2^{-s}.
		\end{gather*}
		Hence the global extrema on $[1/3,1)$ are  attained on the
		first subinterval $[1/3,2/3]$, giving
		\[
		M_{\{0,2\}}=\max_{h\in[1/3,\,2/3]}\phi_{\{0,2\}}(h)=1,\qquad
		m_{\{0,2\}}=\min_{h\in[1/3,\,2/3]}\phi_{\{0,2\}}(h)=2^{-s}.
		\]
		
		An analogous argument applies to the case $D = \{0, 1\}$. Since $-K(\mathcal{F}_D^{-1}) \subset [0, 1/2]$, the measure remains constant over the interval $[1/2, 1)$. This yields
		\begin{gather*}
			\max_{h\in[1/2,\,1)}\phi_{\{0,1\}}(h)=\phi_{\{0,1\}}(1/2)=2^s\mu_{\{0,1\}}([0,1/2])=2^s,\\ \inf_{h\in[1/2,\,1)}\phi_{\{0,1\}}(h)=\phi_{\{0,1\}}(1)=\mu_{\{0,1\}}([0,1])=1.
		\end{gather*}
		Furthermore, for $h\in [1/3,1/2)$, we have
		\[
		\phi_{\{0,1\}}(h)=\f{\mu_{\{0,1\}}([0,1/3])+\mu_{\{0,1\}}([1/3,h])}{h^s}
		=\f{2^{-1}+2^{-1}\mu_{\{0,1\}}([0,3h-1])}{h^s}.
		\]
		By analogous reasoning, we deduce that
			\begin{gather*}
			\sup_{h\in[1/3,\,1/2)}\phi_{\{0,1\}}(h)=\sup_{t\in[0,1/2)}\f{1+\mu_{\{0,1\}}([0,t])}{(1+t)^s}\le \sup_{t\in[0,1/2)}\f{1+M_{\{0,1\}}t^s}{(1+t)^s}=\f{2^s+M_{\{0,1\}}}{2}\le M_{\{0,1\}},\\
			\inf_{h\in[1/3,\,1/2)}\phi_{\{0,1\}}(h)=\inf_{t\in[0,1/2)}\f{1+\mu_{\{0,1\}}([0,t])}{(1+t)^s}\ge \inf_{t\in[0,1/2)}\f{1+m_{\{0,1\}}t^s}{(1+t)^s}=1.
		\end{gather*}
		Combining this with the fact that $\phi_{\{0,1\}}(1/3) = \phi_{\{0,1\}}(1) = 1$, we deduce that the global extrema are completely determined by the interval $[1/2, 1)$, yielding
		\[
		M_{\{0,1\}}=\max_{h\in[1/2,\,1)}\phi_{\{0,1\}}(h)= 2^s,\qquad m_{\{0,1\}}=\min_{h\in[1/2,\,1)}\phi_{\{0,1\}}(h)= 1.
		\]
		
		Finally, for $D = \{0, 1, 2\}$, since $\mathcal{O}_{\mathcal{F}_{\{0,1,2\}}}^+(0) = \mathbb{Z}_{\ge 0}$, we have $\mathcal{C}_s\bigl(\fo_{\F_{\{0,1,2\}}}(0)\bigr)=1.$
	\end{proof}

	\section{Discrete Hausdorff dimensions of lattice orbits}\label{sec_DHDlat}

	We first recall the definitions introduced by Barlow and Taylor \cite{BarlowTaylor1992}. For $x\in\Z^d$ and $n\in\Z_{>0}$, let
	\[
	\begin{aligned}
		C(x,n)&=\{y\in\Z^d:x_i\le y_i<x_i+n,\ 1\le i\le d\},\\
		V(x,n)&=\{y\in\Z^d:x_i-n/2\le y_i<x_i+n/2,\ 1\le i\le d\},
	\end{aligned}
	\]
	and set $\mathscr C=\{C(x,n):x\in\Z^d,\ n\in\Z_{>0}\}$. For a finite set $A\subset\Z^d$, write
	\[
		d(A)=\min\{r\in\Z_{>0}:A\subseteq C(x,r)\text{ for some }x\in\Z^d\}.
	\]
	Let $V_n=V(0,2^n)$ for $n\ge1$, let $S_1=V_1$, and let $S_n=V_n\setminus V_{n-1}$ for $n\ge2$. Thus $d(V_n)=d(S_n)=2^n$. For finite $F\subset\Z^d$, $A\subseteq\Z^d$, and $\alpha>0$, define
	\[
		\nu_\alpha(A,F)=\min\left\{\sum_i\left(\f{d(U_i)}{d(F)}\right)^\alpha:U_i\in\mathscr C,\quad A\cap F\subseteq\bigcup_iU_i
		\right\}.
	\]
	The discrete Hausdorff dimension and the lower discrete Hausdorff dimension of $A$ are, respectively,
	\[
		\dimdH A=\inf\left\{\alpha\ge0:\sum_{n=1}^\infty\nu_\alpha(A,S_n)<\infty\right\},
	\]
	and
	\[
		\dimL A=\inf\left\{\alpha\ge0:
		\lim_{n\to\infty}\nu_\alpha(A,S_n)=0\right\}.
	\]
	The resulting dimensions are independent of the choice of origin; see
	Remark \ref{origin}. The upper mass and Beurling dimensions on $\Z^d$
	are defined analogously using Euclidean balls (or the cubes $V(x,R)$). Moreover, for every $A\subseteq\Z^d$, we have
	\[
		\dimL A\le\dimdH A\le\dimUM A\le\dimbe A.
	\]
\begin{remark}\label{origin}
	All the dimensions defined above are independent of the choice of
	base point whenever a base point occurs in their definitions. For
	$x\in\Z^d$ and $n\in\Z_{>0}$, write
	\[
	S_n(x)=V(x,2^n)\setminus V(x,2^{n-1}).
	\]
	Given $x_1,x_2\in\Z^d$, for
	every sufficiently large $n$,
	\[
		S_n(x_1)\subseteq V\bigl(x_2,2^n+2||x_1-x_2||_{\infty}\bigr)\setminus V\bigl(x_2,2^{n-1}-2||x_1-x_2||_{\infty}\bigr)\subseteq
		\bigcup_{j=-1}^{1}S_{n+j}(x_2).
	\]
	Moreover, for every $r>0$, $
	V(x_1,r)\subseteq V\bigl(x_2,r+2||x_1-x_2||_{\infty}\bigr).$
	Interchanging $x_1$ and $x_2$ gives the corresponding reverse comparisons. Since the multiplicative and additive constants arising
	from these inclusions disappear in the relevant asymptotic limits, none of the dimensions depends on the chosen base point.
\end{remark}

\begin{remark}\label{Rmk_fsta}
	The dimensions $\dimL$, $\dimdH$, $\dimUM$, and $\dimbe$ are finitely
	stable. More precisely, for any finite collection of sets
	$E_1,\ldots,E_n$,
	\[
	\dim\left(\bigcup_{i=1}^n E_i\right)=\max_{1\leq i\leq n}\dim E_i,\qquad\dim\in\{\dimL,\dimdH,\dimUM,\dimbe\}.
	\]
	In contrast, the lower mass dimension $\dimLM$ need not be finitely
	stable.
	
	None of these dimensions is countably stable in general. By
	monotonicity, one always has
	\[
	\dim\left(\bigcup_{i=1}^{\infty}E_i\right)
	\geq
	\sup_{i\geq1}\dim E_i,
	\]
	but the inequality may be strict. For example, $
	\Z^d=\bigcup_{x\in\Z^d}\{x\},$
	while every singleton has dimension zero and $\Z^d$ has dimension $d$.
\end{remark}

	In this section we study invariant sets of the RIFS $\F=\{f_i(x)=r_i x+b_i : \Z\to\Z\}_{i=1}^m$ with $|r_i|\in\Z_{>1}$ and $b_i\in\Z$, first investigated by
	Strichartz \cite{Strichartz1996} in 1996. Under the finite overlap condition, we obtain a complete determination of all dimensions – a  contrast with
	classical self‑similar sets, where such formulas usually require the open set condition.
	\begin{corollary}\label{cor_dssSt}
		Let $K\subseteq \Z$ be an invariant set of a non-degenerate RIFS on $\Z$ of the form \eqref{IFSform} satisfying $|r_i| \in \Z_{>1}$ and $b_i \in \Z$. If $K$ has finite overlaps, then we have
		\[
		\dimL K=\dimdH K=\dimH K(\F^{-1})=\dimM K=\dimbe K=d_\F=s,
		\]
		where $s$ is given by \eqref{dimension}.
	\end{corollary}
In view of Theorem \ref{discreteinvariant} and the finite stability property stated in Remark \ref{Rmk_fsta}, the preceding result is an
immediate consequence of the following theorem on forward orbits.
	\begin{theorem}\label{thm:main}
		Let $\F=\{f_i(x)=r_ix+b_i\}_{i=1}^m$ be a non-degenerate RIFS such that $|r_i|\in \Z_{>1}$ and $b_i\in \Q$ for $1\le i\le m$, and let $a\in \Q$. If $\fo_\F(a)$ has finite overlaps, then there exists an integer $N$ such that $N\fo_\F(a)\subset\Z$ and
		\[
		\dimL (N\fo_\F(a))=\dimdH  (N\fo_\F(a))=\dimM\fo_\F(a)=\dimbe \fo_\F(a)=\dimH K(\F^{-1})=d_\F=s,
		\]
		where $s$ is given by \eqref{dimension}.
	\end{theorem}
	
	To prove this theorem, we use a discrete version of the mass distribution principle proved by Barlow and Taylor \cite{BarlowTaylor1992}.
	\begin{lemma}\label{distr}
		Let  $F$ be a finite set in $\Z^d$ with $d(F)=2^n$. Let $A\subseteq F$ and $\mu$ be a measure on $A$. If there exists a constant $c>0$ such that for every $x\in\Z^d$ and $0\le k\le n$, 
		\[
		\mu(A\cap V(x,2^k))\le c2^{(k-n)\alpha},
		\]
		then $\nu_\alpha (A,F)\ge c^{-1}2^{-1}\mu(A)$.
	\end{lemma}

	\begin{proof} [Proof of Theorem \ref{thm:main}]
		We write $\F'=\{f'_i(x)=r_ix+b_iN\}_{i=1}^m$. 
		It follows from the proof of Proposition \ref{discrete} (i) that there exists an integer $N$ such that
		\[
		N\fo_\F(a)=\fo_{\F'}(aN)\subseteq \Z.
		\]
		By Theorem \ref{mass}, for each $a\in \Q$, there exist positive numbers $C_1,C_2$ such that
		\begin{equation}\label{control}
			\mu_a([-h,h])\ge C_1h^{s},\qquad \mu_a([x-h,x+h])\le C_2h^{s}
		\end{equation}
		for each $x\in \R$ and large $h$, where $\mu_a$ is the counting measure on $\fo_{\F'}(aN)$. Furthermore,
		\[
		\dimM(N\fo_\F(a))=\dimM\fo_{\F'}(aN)=s.
		\]
		
		For each large $n$, we define a measure $\mu_{a,n}$ on $\fo_{\F'}(aN)\cap V(0,2^n)$ by
		\[
		\mu_{a,n}(E)=\f{\mu_a(E)}{\mu_a(V_n)}, \qquad E\subseteq \fo_{\F'}(aN)\cap V_n.
		\]
		For all integers $x$ and $0\le k\le n$ , by \eqref{control}, we have
		\[
		\mu_{a,n}(\fo_{\F'}(aN)\cap V(x,2^k)\cap V_n)=\f{\mu_a( V(x,2^k))}{\mu_a(V_n)}\le \f{C_2}{C_1}2^{(k-n)s}.
		\]
		Since $d(V_n)=2^n$, it follows from Lemma \ref{distr} that
		\[
		\nu_s(\fo_{\F'}(aN),V_n)=\nu_s(\fo_{\F'}(aN)\cap V_n,V_n)\ge \f{C_1}{2C_2}\mu_{a,n}(\fo_{\F'}(aN)\cap V_n)=\f{C_1}{2C_2}>0.
		\]
		According to the definition of $\dimL(\fo_{\F'}(aN))$, we have
		\[
		\dimL(\fo_{\F'}(aN))\ge s
		\]
		By Proposition \ref{discrete} (i),   $\fo_\F(a)$ is uniformly discrete.  Since $\fo_\F(a)$ also has finite overlaps, it follows from Corollary \ref{dimensiondrop} and Theorem \ref{mass} that
		\[
		\dimH K(\F^{-1})=\dimM \fo_\F(a)=\dimbe \fo_\F(a)=d_\F=s,
		\]
		and the result follows.
	\end{proof}

	\section{A \texorpdfstring{$p$-adic}{p-adic} Counterpart}\label{sec_padic}
	
	Given a prime number $p$, the field $\mathbb{Q}_p$ of $p$-adic numbers is defined as the completion of $\mathbb{Q}$ with respect to the non-Archimedean $p$-adic norm $|\cdot|_p$. This norm is defined by $|0|_p = 0$; for $x \neq 0$, write $x = p^{\nu_p(x)} \f{m}{n}$, where $\nu_p(x) = \operatorname{ord}_p x \in \mathbb{Z}$, and $m$ and $n$ are not divisible by $p$, then $|x|_p = p^{-\nu_p(x)}$. This norm in $\mathbb{Q}_p$ satisfies the strong triangle inequality $|x + y|_p \leq \max(|x|_p, |y|_p)$, and consequently $d_p(x,y) = |x-y|_p$ defines an ultrametric on $\mathbb{Q}_p$. For a subset $E \subset \mathbb{Q}_p$, we denote its closure by $\pcl{E}$.

	Let $E$ be a non-empty bounded subset in $\Q_p$. Denote by $N_r(E)$ the smallest number of discs of radii $r$ that cover $E$.  The upper and lower $p$-adic
	box-counting dimensions of $E$ are respectively defined as
	\[
	\overline{\dim}_{\mathrm{B},\mathrm{p}}E=\limsup_{r\to 0^+}\f{\log N_r(E)}{-\log r},\qquad \underline{\dim}_{\mathrm{B},\mathrm{p}}E=\liminf_{r\to 0^+}\f{\log N_r(E)}{-\log r}.
	\]
	If $\overline{\dim}_{\mathrm{B},\mathrm{p}}E=\underline{\dim}_{\mathrm{B},\mathrm{p}}E$, then the $p$-adic box-counting dimension of $E$ exists, denoted by $\dimpB E$.
	See \cite{Qiuhua2019} for details.

	Consider the  RIFS 
	\begin{equation}\label{padicIFS}
		\F = \{f_i(x) = (-1)^{a_i}p^{k_i}x + b_i\}_{i=1}^m
	\end{equation}
	on $\Q$, where $\,b_i \in \mathbb{Q}$, $a_i \in \{0,1\}$ and $k_i \in \mathbb{Z}_{>0}$ for each $1 \le i \le m$. However, since $f_i:\Q_p\to \Q_p$ is contractive with ratio $|(-1)^{a_i}p^{k_i}|_p=p^{-k_i}<1$, $\F$ has a unique non-empty compact attractor $K\subset \Q_p$ satisfying
	\[
	K=\bigcup_{i=1}^m f_i(K).
	\]
	
	In the following, given $E\subset \Q_p$, for each $k\in\Z_{>0}$, we write $N_k(E)$ for the smallest number of discs of diameters $p^{-k}$ covering $E$.
	
	The preceding results allow us to identify the mass and Beurling
	dimensions of every rational forward orbit with the $p$-adic
	box-counting dimensions of both the orbit and the attractor.

	\begin{theorem}\label{thm_dimpadic}
	Let $\F$ be an RIFS of the form \eqref{padicIFS}, and let $K\subseteq\Q_p$ be its attractor. Then the $p$-adic box-counting
	dimension of $K$ exists. Moreover, for every $a\in\Q$, the $p$-adic
	box-counting dimension of $\fo_\F(a)$ exists, and
			\[
			\dimM\fo_\F(a)=\dimbe\fo_{\F}(a)=\dim_{\mathrm{B},\mathrm{p}}\fo_\F(a)= \dimpB K=d_\F.
			\]	
	\end{theorem}

	We begin by recalling a classical result in $\mathbb{Q}_p$. For a detailed proof, we refer the reader to \cite[Corollary 5.1.2]{Fernando2020}.
	\begin{lemma}\label{convergence}
		Given a sequence  $\{a_n\}$ in $\Q_p$,  $\sum_{n=1}^\infty a_n$ converges if and only if $\lim_{n\to\infty}|a_n|_p=0$.
	\end{lemma}
	
	\begin{lemma}\label{closure}
		Let $\F$ be an RIFS given by \eqref{padicIFS}. For each $a \in \mathbb{Q}$, we have
		\[
		\pcl{\fo_\F(a)}=\fo_\F(a)\cup K.
		\]
		In particular, if $a\in K\cap \Q$, then $\pcl{\fo_\F(a)}=K$.
	\end{lemma}
	\begin{proof}
		Since
		\[
		\lim_{\ell\to+\infty}\Big|p^{\sum_{j=1}^\ell k_{i_j}}(-1)^{\sum_{j=1}^\ell a_{i_j}}b_{i_{\ell+1}}\Big|_p=\lim_{n\to+\infty}\Big|p^{\sum_{j=1}^n k_{i_j}}(-1)^{\sum_{j=1}^na_{i_j}}a\Big|_p=0,
		\]
		it follows from Lemma \ref{convergence} that the limit
		$\lim_{n\to\infty}f_{\bi|_n}(a)$ exists in $\Q_p$ for each $a\in \Q$ and $\bi\in \Si^\infty$. 
		
		We define the coding mapping $\pi_{p,a}:\Si^*\cup \Si^\infty\longrightarrow \Q_p$ by
		\[
		\pi_{p,a}(\bi)=\left\{\begin{array}{ll}
			f_{\bi}(a)\quad & \text{if\, } \bi\in\Si^*,\\
			\lim_{n\to\infty}f_{\bi|_n}(a) & \text{if\, }\bi\in\Si^\infty.
		\end{array}
		\right.
		\]
		Recall from Corollary \ref{denseofsymbolic} that the domain $\Si^* \cup \Sigma^\infty$ is compact. Since $\pi_{p,a}$ is continuous on a compact domain, it commutes with the closure operator. By Corollary \ref{denseofsymbolic}, we obtain
		\begin{align*}
			\pcl{\fo_\F(a)}&=\pcl{\pi_{p,a}(\Si^*)}\\
			&=\pi_{p,a}\big(\overline{\Si^*}\big)\\
			&=\pi_{p,a}\big(\Si^*\cup \Si^\infty\big)\\
			&=\pi_{p,a}\big(\Si^*\big)\cup \pi_{p,a}\big( \Si^\infty\big)\\
			&=\fo_\F(a)\cup K.
		\end{align*}
		From the self-similarity relation  $K=\bigcup_{i=1}^m f_i(K)$, it is straightforward to deduce that if $a\in K$, then $\fo_\F(a)\subset K$. 
	\end{proof}
	\begin{remark}
		A close inspection of the preceding proof reveals that the conclusion remains valid for the orbits of uniformly contracting systems on complete metric spaces.
	\end{remark}

	\begin{lemma}\label{massandbox}
		Let $\F$ be an RIFS given by \eqref{padicIFS}. For each $a\in \Q$, there exist $C_1,C_2>0$ and $\delta>0$ such that for all $k\in \Z_{>0}$,
		\[
		\#\big(\fo_\F(a)\cap(-C_1p^k,C_1p^k)\big)\le N_k(\fo_\F(a))\le \#\big(\fo_\F(a)\cap[-C_2p^{k+\delta},C_2p^{k+\delta}]\big).
		\]
	\end{lemma}
	\begin{proof}
		Fix $a\in \Q$ and $k\in \Z_{>1}$. 
		It follows from the proof of Proposition \ref{discrete} (i) that there exists $M>0$ such that the absolute values of the denominators of the elements in $\fo_\F(a)$ are bounded above by $M$. Note that for $x,y\in \fo_\F(a)\cap(-\f{p^k}{2M^2},\f{p^k}{2M^2}), x=a_1/b_1\ne y=a_2/b_2, \gcd(a_1,b_1)=\gcd(a_2,b_2)=1$, we have 
		\[
		|a_1b_2-a_2b_1|=|b_1b_2||x-y|<p^k,
		\]
		which implies that $\nu_p(a_1b_2-a_2b_1)<k$. It follows that
		\[
		|x-y|_p=\Big|\f{a_1b_2-a_2b_1}{b_1b_2}\Big|_p=\f{|a_1b_2-a_2b_1|_p}{|b_1b_2|_p}\ge |a_1b_2-a_2b_1|_p>p^{-k}.
		\]
		Thus no two distinct elements of $\fo_\F(a)\cap(-\f{p^k}{2M^2},\f{p^k}{2M^2})$ can lie in the same ball of radius $p^{-k}$ in $\mathbb{Q}_p$. That is, 
		\[
		\#\Big(\fo_\F(a)\cap\Big(-\f{p^k}{2M^2},\f {p^k}{2M^2}\Big)\Big)\le N_k(\fo_\F(a)).
		\]
		
		Fix $x\in \fo_\F(a)$. Then there exists $n\in \Z_{>0}$ and $\bi=i_1\ldots i_n\in \Si^*$ such that
		\[
		x=f_\bi(a)= p^{\sum_{j=1}^n k_{i_j}}(-1)^{\sum_{j=1}^na_{i_j}}a+\sum_{\ell=1}^{n-1}p^{\sum_{j=1}^\ell k_{i_j}}(-1)^{\sum_{j=1}^\ell a_{i_j}}b_{i_{\ell+1}}+b_{i_1}.
		\]
		Set $N=\max\{0,-\nu_p(a),-\nu_p(b_1),\ldots,-\nu_p(b_m)\}, b=\max_i |b_i|$. If there exists $t\in\{2,\ldots,n\}$ such that $\sum_{j=1}^{t-1}k_{i_j}\le k+N< \sum_{j=1}^{t}k_{i_j}$, then
		\begin{align*}
			&\Big|x-f_{\bi|_t}(a)\Big|_p\\
			=&\,\Big|p^{\sum_{j=1}^n k_{i_j}}(-1)^{\sum_{j=1}^na_{i_j}}a-p^{\sum_{j=1}^t k_{i_j}}(-1)^{\sum_{j=1}^ta_{i_j}}a+\sum_{\ell=t}^{n-1}p^{\sum_{j=1}^\ell k_{i_j}}(-1)^{\sum_{j=1}^\ell a_{i_j}}b_{i_{\ell+1}}\Big|_p\\
			\le &\,\max\Big\{\Big|p^{\sum_{j=1}^n k_{i_j}}a\Big|_p,\,\Big|p^{\sum_{j=1}^t k_{i_j}}a\Big|_p,\,\Big|p^{\sum_{j=1}^t k_{i_j}}b_{i_{t+1}}\Big|_p,\ldots,\,\Big|p^{\sum_{j=1}^{n-1} k_{i_j}}b_{i_n}\Big|_p\Big\}\\
			<&\, p^{-k}.
		\end{align*}
		Consequently, it follows from the ultrametricity of $\mathbb{Q}_p$ that if $f_{\bi|_t}(a)$ is contained in a ball $B \subset \mathbb{Q}_p$ of radius $p^{-k}$, then $x \in B$.
		Note that
		\[
		|f_{\bi|_t}(a)|\le p^{k+N+\max_ik_i}|a|+b\sum_{i=0}^{k+N-1}p^i<C_2p^{k+N+\max_ik_i},
		\]
		where $C_2=|a|+\f{b}{p-1}$. Moreover, if $\sum_{j=1}^{n}k_{i_j}\le k+N$, then  it is straightforward that 
		\[
		x\in[-C_2p^{k+N},C_2p^{k+N}]\subset\Big[-C_2p^{k+N+\max_ik_i},C_2p^{k+N+\max_ik_i}\Big].
		\]
		Consequently, it follows that
		\[
		N_k(\fo_\F(a))\le \#\Big(\fo_\F(a)\cap\Big[-C_2p^{k+N+\max_ik_i},C_2p^{k+N+\max_ik_i}\Big]\Big).
		\]
	\end{proof}

	\begin{proof}[Proof of Theorem \ref{thm_dimpadic}]
			Recall that taking the $p$-adic closure does not change either the
			upper or the lower $p$-adic box-counting dimension. Thus, for every
			$a\in\Q$,
		\[
		\overline{\dim}_{\mathrm{B},\mathrm{p}}\fo_\F(a)=\overline{\dim}_{\mathrm{B},\mathrm{p}}\pcl{\fo_\F(a)},\qquad \underline{\dim}_{\mathrm{B},\mathrm{p}}\fo_\F(a)=\underline{\dim}_{\mathrm{B},\mathrm{p}}\pcl{\fo_\F(a)}.
		\]
		By  Proposition \ref{discrete} (i), Theorem \ref{massandbeu} and Lemma \ref{massandbox}, for every $a\in \Q$, we have
		\[
		\overline{\dim}_{\mathrm{B},\mathrm{p}}\fo_\F(a)=\underline{\dim}_{\mathrm{B},\mathrm{p}}\fo_\F(a)=\dimM\fo_\F(a)=\dimbe\fo_{\F}(a).
		\]
		Consequently, the $p$-adic box-counting dimension of
		$\fo_\F(a)$ exists and equals $d_\F$ for every $a\in\Q$.
		
	Since
	\[
	P(\F)\subseteq K\cap\Q,
	\]
	we have $K\cap\Q\neq\emptyset$. Choose $a\in K\cap\Q$. By Lemma \ref{closure}, $\pcl{\fo_\F(a)}=K.$
		Therefore,
		\[
		\overline{\dim}_{\mathrm{B},\mathrm{p}}K
		=\overline{\dim}_{\mathrm{B},\mathrm{p}}\fo_\F(a)
		=d_\F=\underline{\dim}_{\mathrm{B},\mathrm{p}}K
		=\underline{\dim}_{\mathrm{B},\mathrm{p}}\fo_\F(a).
		\]
		Hence the $p$-adic box-counting dimension of $K$ exists and satisfies
		\[
		\dimpB K=d_\F.
		\]
		This completes the proof.
	\end{proof}

	\section*{Acknowledgments}
	The initial research question was suggested by Prof. Lifeng Xi, and the authors wish to thank Prof. Xi for his helpful comments during the research.

\end{document}